\documentclass[11pt,fleqn]{article}
\usepackage{mathtools,amsfonts,amsmath, amsthm,amssymb}
\usepackage{accents}
\usepackage{chngcntr}
\usepackage{yfonts}

\usepackage[active]{srcltx}

\usepackage{latexsym}
\usepackage{esint}

\usepackage[margin=22mm]{geometry} 
\numberwithin{equation}{section}

\usepackage{color}
\definecolor{db}{rgb}{0.0,0.0,0.8} 
\definecolor{dg}{rgb}{0.0,0.55,0.14}
\definecolor{dr}{rgb}{0.5,0,0.07}

\usepackage{hyperref}
\hypersetup{frenchlinks=true,colorlinks=true,linkcolor=db,citecolor=dg,
	bookmarks=true}

\usepackage{enumerate}
\usepackage{authblk}

\newtheorem{theorem}{Theorem}[section]
\newtheorem{proposition}{Proposition}[section]

\newtheorem{lemma}[proposition]{Lemma}
\newtheorem{corollary}[proposition]{Corollary}
\theoremstyle{definition}

\theoremstyle{definition}

\theoremstyle{definition}

\theoremstyle{definition}

\theoremstyle{definition}
\newtheorem{definition}[theorem]{Definition}
\theoremstyle{definition}
\newtheorem{remark}{Remark}[section]
\theoremstyle{definition}

\newtheorem{open-problem}{Open Problem}
\newcounter{step}

\newcommand{\rlemma}[1]{Lemma~\ref{#1}}
\newcommand{\rth}[1]{Theorem~\ref{#1}}
\newcommand{\rcor}[1]{Corollary~\ref{#1}}

\newcommand{\rprop}[1]{Proposition~\ref{#1}}

\def\be{\begin{equation}}
\def\ee{\end{equation}}
\def\bes{\begin{equation*}}
\def\ees{\end{equation*}}
\def\bt{\begin{theorem}}
	\def\et{\end{theorem}}
\def\bpr{\begin{proposition}}
	\def\epr{\end{proposition}}
\def\bl{\begin{lemma}}
	\def\el{\end{lemma}}
\def\bc{\begin{corollary}}
	\def\ec{\end{corollary}}
\def\br{\begin{remark}}
	\def\er{\end{remark}}
\def\ben{\begin{enumerate}}
	\def\bena{\begin{enumerate}[a)]}
		\def\een{\end{enumerate}}
	\def\bit{\begin{itemize}}
		\def\iit{\end{itemize}}
	
	\def\supp{\operatorname{supp}}
	\def\dist{\operatorname{dist}}

	\def\deg{\operatorname{deg}}

	\newcommand{\Prod}{\mathop{\prod}\limits}
	\DeclareMathAlphabet{\mathonebb}{U}{bbold}{m}{n}
	
	\def\R{{\mathbb R}}
	\def\N{{\mathbb N}}
	
	\def\C{{\mathbb C}}
	\def\Z{{\mathbb Z}}

	\usepackage{csquotes}

	\def\fo{\forall\, }

	\def\va{\varphi}
	\def\ue{u_\varepsilon}
	
	\def\d{\displaystyle}
	
	\def\ve{\varepsilon}
	\def\p{\partial}

	\renewcommand{\vec}[1]{\boldsymbol{#1}}

	\def\supp{\operatorname{supp}}
	\def\dist{\operatorname{dist}}

	\def\deg{\operatorname{deg}}

	\DeclareMathOperator{\Div}{div}
	\makeatletter
	\def\moverlay{\mathpalette\mov@rlay}
	\def\mov@rlay#1#2{\leavevmode\vtop{%
			\baselineskip\z@skip \lineskiplimit-\maxdimen
			\ialign{\hfil$\m@th#1##$\hfil\cr#2\crcr}}}
	\newcommand{\charfusion}[3][\mathord]{
		#1{\ifx#1\mathop\vphantom{#2}\fi
			\mathpalette\mov@rlay{#2\cr#3}
		}
		\ifx#1\mathop\expandafter\displaylimits\fi}
	\makeatother

\begin{document}
\title{A variational singular perturbation problem motivated by
   Ericksen's model for nematic liquid crystals}
\title{A variational singular perturbation problem motivated by
	Ericksen's model for nematic liquid crystals}
\author[1]{Dmitry Golovaty\thanks{dmitry@uakron.edu}}
\author[2]{Itai Shafrir\thanks{shafrir@math.technion.ac.il}}
\affil[1]{Department of Mathematics, 
	The University of Akron, Akron, Ohio 44325, USA}  
\affil[2]{Department of Mathematics, Technion - I.I.T., 32 000 Haifa, ISRAEL}

\newcommand\blfootnote[1]{%
	\begingroup
	\renewcommand\thefootnote{}\footnote{#1}%
	\addtocounter{footnote}{-1}%
	\endgroup
}


\maketitle
\begin{abstract}
We study the asymptotic behavior,  when $\varepsilon\to0$, of the
minimizers $\{u_\varepsilon\}_{\varepsilon>0}$ for the energy
  \begin{equation*}
 E_\varepsilon(u)=\int_{\Omega}\Big(|\nabla u|^2+\big(\frac{1}{\varepsilon^2}-1\big)|\nabla|u||^2\Big),
\end{equation*}
over the class of maps $u\in H^1(\Omega,{\mathbb R}^2)$ satisfying the
boundary condition $u=g$ on $\partial\Omega$, where $\Omega$ is a
smooth, bounded and simply connected domain in ${\mathbb R}^2$ and
$g:\partial\Omega\to S^1$ is a smooth boundary data of degree
$D\ge1$. 
 The motivation comes from a simplified version of the Ericksen model
 for nematic liquid crystals with variable degree of orientation. We
 prove convergence (up to a subsequence) of $\{u_\varepsilon\}$
 towards a singular $S^1$--valued harmonic map $u_*$, a result that
 resembles the one  obtained in \cite{BBH} for an analogous problem for
 the Ginzburg-Landau energy. There are however two striking
 differences between our result and the one involving the
 Ginzburg-Landau energy. First, in our problem the singular limit
 $u_*$ may have singularities of degree strictly larger than
 one. Second, we find that the principle of \enquote{equipartition} holds for
 the energy of the minimizers, i.e., the contributions of the two terms in
 $E_\varepsilon(u_\varepsilon)$ are  essentially equal.  
 \end{abstract}

\section{Introduction}
Let $\Omega\subset\R^2$ be a smooth, bounded and simply connected
domain and $g:\partial\Omega\to S^1$ a smooth boundary condition.
For each $\varepsilon>0$ consider the energy
\begin{equation}\label{eq:energy}
 E_\varepsilon(u)=\int_{\Omega}\Big(|\nabla u|^2+\big(\frac{1}{\varepsilon^2}-1\big)|\nabla|u||^2\Big)\,,
\end{equation}
and let $u_\varepsilon$ denote a minimizer for $E_\varepsilon$ over
$$
H^1_g(\Omega)=H^1_g(\Omega;\R^2):=\{u\in H^1(\Omega;\R^2)\text{ s.t. }u=g\text{ on
}\partial\Omega\}.
$$
We are interested in the limit of $u_\varepsilon$
when $\varepsilon$ goes to zero.

This problem can be viewed as a relaxation of the problem
\begin{equation}
  \label{eq:197}
  \min\{\int_\Omega |\nabla v|^2\,:\,v\in
   H^1_g(\Omega;S^1)\}.
\end{equation}
 In fact, when the {\em degree} of $g$---to be denoted hereafter by $D$---is zero, no relaxation is needed since the problem \eqref{eq:197} has a
 solution. In this case there exists a (smooth) scalar
 function $\varphi_0$ such that $g=e^{i\varphi_0}$ and the (unique) minimizer in
 \eqref{eq:197} is given by
 $u_0=e^{i\widetilde\varphi_0}$, where $\widetilde\varphi_0$ is the
 harmonic extension of $\varphi_0$ to $\Omega$. When $D=0$, we prove in \rth{th:deg-zero} that
 $u_\varepsilon\to u_0$
 in $C^m(\overline\Omega)$, $\forall m$. This result
   is analogous  to the one treating the zero degree case of the Ginzburg-Landau
 energy in \cite{bbh}.
  
  The more interesting situation arises when $D=\deg g\ne 0$ because for such $g$ the set of competitors $H^1_g(\Omega;S^1)$ is empty (see e.g.,
  \cite[Introduction]{BBH}) and the problem \eqref{eq:197} has no solution. Even though the minimization
  problem \eqref{eq:197} is by itself meaningless, one may still consider the limit of
  $u_\varepsilon$ when $\varepsilon$ goes to zero, as a
  \enquote{generalized minimizer}.

This type of relaxation was carried out in
  the past for different energies. In their famous work,
  Bethuel, Brezis and H\'elein~\cite{BBH} (see also \cite{struwe}) studied the limit of
  the minimizers $\{v_\varepsilon\}$ for the energy
  \begin{equation}
    \label{eq:198}
    F_\varepsilon(u)=\int_{\Omega}\Big(|\nabla u|^2+\frac{1}{2\varepsilon^2}(1-|u|^2)^2\Big)\,,
  \end{equation}
 over $H^1_g(\Omega)$. In the case $\deg g=D\ge1$ they showed for a
 subsequence that
 \begin{equation}
   \label{eq:199}
   v_{\varepsilon_n}\to u_*=e^{i\varphi}\Prod_{j=1}^D
   \frac{z-a_j}{|z-a_j|}\text{ in }C^{1,\alpha}_{\text{loc}}({\overline\Omega}\setminus\{a_1,\ldots,a_D\}),
 \end{equation}
  where $\varphi$ is a harmonic function determined by the constraint
  $u_*=g$ on $\partial\Omega$. Moreover,
  \begin{equation}
    \label{eq:202}
    \lim_{\varepsilon\to0} F_\varepsilon(v_\varepsilon)-2\pi
    D|\ln\varepsilon|=\min_{{\bf b}\in\Omega^D}W({\bf b})+D\gamma,
  \end{equation}
  where $\gamma$ is a universal constant and $W$ is the {\em
    renormalized energy} that was introduced in \cite{BBH}, see
  \eqref{eq:184} and \eqref{eq:186} below. In summary, the limit of a
  sequence of minimizers has $D$ singularities of degree one, with
  their locations determined by minimization of $W$ over all
  configurations of $D$ {\em distinct} points in $\Omega$. Interestingly, the same
  type of limit as in \eqref{eq:199} is also obtained for a different relaxation,
  studied by Hardt and Lin~\cite{hl95}. In contrast with the
  case $p=2$, the set $W^{1,p}_g(\Omega;S^1)\ne\emptyset$ for $p\in[1,2)$. Denoting by $w_p$ a minimizer
  for $\int_\Omega|\nabla u|^p$ over $W^{1,p}_g(\Omega;S^1)$ for each
  $p\in[1,2)$, they showed for a subsequence $p_n\nearrow 2$ that an
  analogous result to \eqref{eq:199} holds, namely,
  \begin{equation}
    \label{eq:200}
     w_{p_n}\to u_*=e^{i\varphi}\prod_{j=1}^D
   \frac{z-a_j}{|z-a_j|}\text{ in }C^{1,\alpha}_{\text{loc}} ({\overline\Omega}\setminus\{a_1,\ldots,a_D\}).
  \end{equation}
  Moreover, an analogous formula to \eqref{eq:202} holds in this case
  as well and the locations of the singularities $a_1,\ldots,a_D$ are still 
  determined by minimizing the same renormalized energy as above.

  \par In view of these two examples, one may suspect that any \enquote{reasonable} relaxation would
  lead to the same limit. Somewhat surprisingly, we find that this
  isn't the case for the limit of the minimizers $\ue$ of
  $E_\varepsilon$ over $H^1_g(\Omega)$. We will show that, for a
  subsequence, we have
\begin{equation}
  \label{eq:201}
     u_{\varepsilon_n}\to u_*=e^{i\varphi}\prod_{j=1}^N
   \left(\frac{z-a_j}{|z-a_j|}\right)^{d_j}\text{ in }C^m_{\text{loc}}({\overline\Omega}\setminus\{a_1,\ldots,a_N\}),
 \end{equation}
 with degrees $d_j\ge1$,$\forall j$, i.e., the limit is the {\em canonical
 harmonic map} associated with $g$, the singularities and their degrees
 (see \cite{BBH}). However, in contrast to \eqref{eq:199} and
 \eqref{eq:200}, we might have $d_j\ge2$ for
 some values of $j$, so that  a strict inequality $N<D$ may
 occur (see \rcor{cor:rad} and \rprop{prop:B} below). Moreover, the location of the singularities and their degrees
 are determined by minimizing a different function than $W$.

An important property of the energy \eqref{eq:energy} is its conformal
invariance, that is, we have $E_\varepsilon(u)=E_\varepsilon(u\circ
F)$ for every conformal map $F$. We shall often use this property in the sequel. For example, it
allows us to assume that the simply connected domain $\Omega$ is the
unit disc (thanks to the Riemann mapping theorem).
 Our first result for the case $D\ge1$ 
 provides a convergence result and a partial description of the limit.
\begin{theorem}
\label{th:main1}
 Let $\Omega$ be a smooth, bounded, simply connected domain in $\R^2$.
 Let $g:\partial\Omega\to S^1$ be a smooth
 boundary condition of degree $D\ge1$. Then,
 \begin{equation}
   \label{eq:42}
   \frac{2\pi D}{\varepsilon} \le E_\varepsilon(u_\varepsilon)\le \frac{2\pi D}{\varepsilon}+C.
 \end{equation}
Moreover,  up to a subsequence we have 
\begin{equation}
\label{eq:91}
 u_{\varepsilon_n}\to
  u_* \text{ in }C^{m}_{\text{loc}}\left(\overline\Omega\setminus \{a_1,\ldots,
    a_N\}\right),~\forall m,
\end{equation}
where $u_*$ is a smooth $S^1$--valued harmonic map in $\overline\Omega\setminus \{a_1,\ldots,
    a_N\}$.
The singularities 
$a_1,\ldots,a_N$ are distinct points in $\Omega$, the degree of $u_*$
around each $a_j$ is an integer $d_j> 0$, and 
 the compatibility condition $\sum_{j=1}^N d_j=D$ holds. Moreover,
 $u_*$ is the {\em canonical} harmonic map associated with $g$, the
 points $a_1,\ldots,a_N$ and the degrees $d_1,\ldots,d_N$.
\end{theorem}

Our second result establishes a precise asymptotic expansion of the energy
$E_\varepsilon(u_\varepsilon)$ by computing
$$\lim_{\varepsilon\to0} E_\varepsilon(u_\varepsilon)-\frac{2\pi
  D}{\varepsilon}.
$$
This allows us to obtain a criterion for the choice of the
points $a_1,\ldots,a_N$ and their associated degrees
$d_1,\ldots,d_N$.
In order to state the next theorem, we will need the following definitions.

   For each integer $D\ge1$ we set
   \begin{multline}
     \label{eq:79}
     \mathcal{H}_D(\partial\Omega)=\{g\in
     C^1(\partial\Omega;S^1):\deg g=D\text{ and }
     g=G\big|_{\partial\Omega}\text{ for some holomorphic }\\G\in
     C^1(\overline\Omega;\C) \text{ s.t. }G(\partial\Omega)=\partial B_1\}.
   \end{multline}
  An explicit description of $\mathcal{H}_D(\partial\Omega)$ is
  available using the concept of Blaschke products. Indeed, when
  $\Omega=B_1$, to any  configuration of $D\ge1$ points ${\bf a}\in
  B_1^D$  we associate a Blaschke product
     \[{\cal B}_{\bf a}(z):=\Prod_{j=1}^D \frac{z-a_j}{1-\bar a_jz}.\]
   Then we have,
   \begin{equation}
     \label{eq:83}
     \mathcal{H}_D(\partial B_1)=\{e^{i\alpha}{\cal B}_{\bf
       a}(z)\big|_{\partial B_1}:\alpha\in\R, {\bf a}\in B_1^D\}.
   \end{equation}
   For an arbitrary smooth and simply connected $\Omega$ we may fix a
   Riemann mapping $F:\Omega\to B_1$ (with smooth extension to the
   boundary) and then clearly
   \begin{equation}
     \label{eq:93}
      \mathcal{H}_D(\partial\Omega)=\{g\circ F: g\in \mathcal{H}_D(\partial B_1)\},
    \end{equation}
   so any function in $\mathcal{H}_D(\partial\Omega)$ has the form
   $e^{i\alpha}\Prod_{j=1}^D \frac{F(z)-a_j}{1-\bar
       a_jF(z)}$, for some $\alpha\in\R$ and $ {\bf a}\in B_1^D$.
 
   Let $g_1,g_2:\partial\Omega\to S^1$ be two smooth maps, or more
 generally, maps in $H^{1/2}(\partial\Omega;S^1)$ with the same
 degree. We define a {\em distance} between the maps as follows:
 \begin{equation}
   \label{eq:dist}
   d_{H^{1/2}}(g_1,g_2)=\inf\{\|\nabla w\|_{L^2(\Omega)}\,:\,w\in
   H^1(\Omega;S^1), w=g_1{\bar g}_2\text{ on }\partial\Omega\}\,.
 \end{equation}
  Note that the assumption $\deg g_1=\deg g_2$ implies that $\deg
  g_1{\bar g}_2=0$, whence we may write on $\partial\Omega$, $
  g_1{\bar g}_2=e^{i\psi}$ for some scalar function $\psi$ on
  $\partial\Omega$ (with $\psi$ smooth, or more generally in $H^{1/2}(\partial\Omega)$). It is then
  clear that
  \begin{equation}
    \label{eq:65}
     d_{H^{1/2}}(g_1,g_2)=\|\nabla\widetilde\psi\|_{L^2(\Omega)}\,,
   \end{equation}
   where $\widetilde\psi$ denotes the harmonic extension of
   $\psi$. Naturally we denote for $g\in C^1(\partial\Omega;S^1)$ of degree $D$,
   \begin{equation}
     \label{eq:94}
     d_{H^{1/2}}(g,\mathcal{H}_D(\partial\Omega))=\inf_{f\in\mathcal{H}_D(\partial\Omega)}d_{H^{1/2}}(g,f)\,.
   \end{equation}
   It is easy to see that the infimum in \eqref{eq:94} is
   actually attained. Note that when $\Omega=B_1$ we have
   \begin{equation}
     \label{eq:110}
  d_{H^{1/2}}^2(g, \mathcal{H}_D(\partial B_1))=
\min_{{\bf b}\in B_1^D} \Big\{\int_{B_1}|\nabla\widetilde\varphi|^2:
\varphi\in H^{1/2}(\partial B_1)\text{ s.t. }
e^{i\varphi}=g{\overline{\cal B}_{\bf
    b}}\Big\},
\end{equation}
 where as usual $\widetilde\varphi$ denotes the harmonic extension of
 $\varphi$. 
A similar expression can be written for a general $\Omega$, using
the Riemann mapping $F:\Omega\to B_1$. 
\par    The \enquote{excess energy} 
   $d^2_{H^{1/2}}(g,\mathcal{H}_D(\partial\Omega))$ is related to the 
   notion of {\em renormalized energy} $W$ from
   \cite{BBH},  but there are important differences
     between the two, see Remark~\ref{rem:W} below. Next, we present an explicit expression
   for $d^2_{H^{1/2}}(g,\mathcal{H}_D(\partial\Omega))$ using
   quantities that also appear in $W$.
   We begin by recalling one of the equivalent definitions of $W$ from
   \cite{BBH}. It is convenient to denote by $(\Omega^N)^*$ the subset of $\Omega^N$ consisting
   only of configurations of
   {\em distinct} points. Given a boundary condition $g:\partial\Omega\to S^1$ of
   degree $D>0$,  the points ${\bf
     a}\in (\Omega^N)^*$, and the degrees ${\bf
     d}\in {\Z}^N$ satisfying $\sum_{j=1}^Nd_j=D=\deg
   g$, we first consider the associated {\em canonical harmonic map}
   $$
   u_0=e^{i\widetilde\varphi}\prod_{j=1}^N\Big(\frac{z-a_j}{|z-a_j|}\Big)^{d_j},
   $$
    where $\widetilde\varphi$ is the harmonic extension of $\varphi$,
    which in turn is determined (up to an additive constant in
    $2\pi\Z$) by the requirement that $u_0=g$ on $\partial\Omega$. 
    Thm\,I.8 in \cite{BBH} asserts that
    \begin{equation}
      \label{eq:184}
      \int_{\Omega\setminus\bigcup_{j=1}^NB_\lambda(a_j)}|\nabla
      u_0|^2=2\pi\Big(\sum_{j=1}^Nd_j^2\Big)\ln(1/\lambda)+W+O(\lambda^2),\text{
        as }\lambda\to0^+.
    \end{equation}
    An explicit expression for 
   $W=W({\bf a},{\bf d},g)$ is given in \cite[Thm\,I.7]{BBH} (note that
   there is a factor of $2$ difference between our definition and the
   one in \cite{BBH}). This expression involves the solution
   $\widetilde\Phi_0$ of
   \begin{equation}
\label{eq:185}     
\left\{
\begin{aligned}
\Delta {\widetilde\Phi}_0&=2\pi\sum_{j=1}^Nd_j\delta_{a_j}\text{ in }\Omega,\\
\frac{\partial{\widetilde\Phi}_0}{\partial\nu}&=g\times g_\tau \text{ on }\partial\Omega,
\end{aligned}
\right.
\end{equation}
 with the normalization condition
 $\int_{\partial\Omega}{\widetilde\Phi}_0=0$.
 Setting $R_0(x)=\widetilde\Phi_0(x)-\sum_{j=1}^Nd_j\ln|x-a_j|$, we
 have according to \cite{BBH},
 \begin{equation}
   \label{eq:186}
   W({\bf a},{\bf d},g)=\int_{\partial
  \Omega}{\widetilde\Phi}_0(g\times
g_\tau)\,d\tau-2\pi\sum_{j=1}^Nd_jR_0(a_j)-2\pi\sum_{i\ne j}d_id_j\ln|a_i-a_j|.
 \end{equation}
The relation between $d^2_{H^{1/2}}(g,\mathcal{H}_D(\partial\Omega))$
and $W$ is clarified in the next proposition. To state it, we define,
as in \cite{lr96},
\begin{equation}
     \label{eq:95}
     \widetilde W({\bf a},{\bf d})=\inf\,\big\{
     W({\bf a},{\bf d},f):f\in
     C^1(\partial\Omega;S^1), \deg f=D=\sum_{j=1}^Nd_j\big\}.
   \end{equation}
\begin{proposition}
  \label{prop:W}
  We have
   \begin{equation}
     \label{eq:96}
     d_{H^{1/2}}^2(g,\mathcal{H}_D(\partial\Omega))=\inf\big\{W({\bf a},{\bf d},g)-\widetilde
     W({\bf a},{\bf d}):{\bf a}\in (\Omega^N)^*, 
     {\bf d}\in\Z_{+}^N, \sum_{j=1}^N d_j=D, N\ge1\}.
   \end{equation}
Moreover,   when $\Omega=B_1$,
\begin{equation}
  \label{eq:187}
d_{H^{1/2}}^2(g, \mathcal{H}_D(B_1))=
\min_{\substack{N\ge1\\{\bf a}\in (B_1^N)^*\\d_j\ge1,\forall j\\\sum_{j=1}^N d_j=D}} \int_{\partial B_1}{\widetilde\Phi}_0(g\times g_\tau)\,d\tau-2\pi\sum_{j=1}^Nd_jR_0(a_j)-2\pi\sum_{i,j=1}^Nd_id_j\ln|1-a_i{\bar a}_j|.
\end{equation}
\end{proposition}
Comparing \eqref{eq:187} to \eqref{eq:186} we notice the absence from
\eqref{eq:187}  of
the last term in \eqref{eq:186},
$$
-2\pi\sum_{i\ne j}d_id_j\ln
|a_i-a_j|,
$$
responsible for repulsion between vortices. This might explain the fact that vortices of degree
$d_j\ge2$ are allowed for minimizers of $E_\varepsilon$. In the context of Ginzburg-Landau-type problems, we are not aware of another situation  where energy minimizers are characterized by point singularities that have unbounded energy in the limit $\varepsilon\to0$ {\it and} have degrees different from $\pm1$.

\begin{remark}
  There is an alternative simple expression to the one in \eqref{eq:187} in
  which the minimization is over all the configurations of $D$ points
  in $B_1$ (not necessarily distinct):
\begin{equation}
 \label{eq:196}
d_{H^{1/2}}^2(g, \mathcal{H}_D(B_1))=
\min_{{\bf a}\in B_1^D} \int_{\partial B_1}{\widetilde\Phi}_0(g\times g_\tau)\,d\tau-2\pi\sum_{j=1}^DR_0(a_j)-2\pi\sum_{i,j=1}^D\ln|1-a_i{\bar a}_j|,
\end{equation}
where ${\widetilde\Phi}_0$ is like in \eqref{eq:185}, but with $N=D$
and $d_j=1$ for all $j$ (and accordingly
$R_0(x)=\widetilde\Phi_0(x)-\sum_{j=1}^D\ln|x-a_j|$).
The verification of \eqref{eq:196} from \eqref{eq:187} is
straightforward.
\end{remark}
\begin{remark}
  \label{rem:W}
    We should emphasize that, although there are some
   common expressions in the explicit formulas for the  renormalized
   energy $W({\bf a},{\bf d},g)$ and the \enquote{excess energy}
   $d_{H^{1/2}}^2(g,\mathcal{H}_D(\partial\Omega))$, there is a basic
   difference between the two. The renormalized energy has an
   intrinsic meaning. To cite from the Introduction in \cite{BBH}: it is
   what remains in the energy after the singular 
   \enquote{core energy} $2\pi d|\log\lambda|$ has been removed in the
   problem of shrinking holes of radii $\lambda$,
   which is closely related to \eqref{eq:184}. This feature of $W$, namely,
   that it represents the \enquote{regular part} of the energy of singular $S^1$--valued
   harmonic maps,  is the reason behind  its appearance in many different
   variational problems. For example, in addition to the case involving the Ginzburg-Landau
   energy~\cite{BBH}, one can find analogous $W$ in the problems considered in \cite{an-sh,BPP,hang-lin,hl95}.
  The excess energy $d_{H^{1/2}}^2(g,\mathcal{H}_D(\partial\Omega))$
  is quite different in that it is specific to the particular
  choice of the energy $E_\varepsilon$ we consider. It does not represent
  a contribution from the phase alone (as does $W$) but rather
  a contribution from {\bf both} the phase {\bf and} the modulus of the maps.
\end{remark}

\par
We are now ready to state our second main theorem that provides a more
precise information about the asymptotic behavior of the energy and
the location of the singularities of the limit $u_*$. Note that we denote
$\rho_\varepsilon=|u_\varepsilon|$ throughout the manuscript.
   \begin{theorem}
     \label{th:main2}
      Let $\Omega$, $g$ and $u_*$ together with the singular points $a_1,\ldots,a_N$ and the degrees $d_1,\ldots,d_N$ be
      as in \rth{th:main1}. Then, up to a subsequence we have:\\[2mm]
{\rm (i)} $\lim_{\varepsilon\to0}
\frac{\ln\rho_\varepsilon}{\varepsilon}=\lim_{\varepsilon\to0}
\frac{\rho_\varepsilon-1}{\varepsilon}=\Phi_0$ in
$C^m_{\text{loc}}(\overline{\Omega}\setminus\{a_1,\ldots,a_N\})$,
$\forall m\ge 1$, where
$\Phi_0$ is the solution of
\begin{equation}
\label{eq:101}
\left\{
\begin{aligned}
  \Delta\Phi_0&=2\pi\sum_{j=1}^N d_j\delta_{a_j}\text{ in }\Omega,\\
 \Phi_0&=0 \text{ on }\partial\Omega.
\end{aligned}
\right. 
\end{equation}
{\rm (ii)} 
\begin{equation}
  \label{eq:102}
\lim_{\varepsilon\to0} E_\varepsilon(u_\varepsilon)-\frac{2\pi D}{\varepsilon}=d_{H^{1/2}}^2(g,\mathcal{H}_D(\partial\Omega)).
\end{equation}
{\rm (iii)} The configurations of points ${\bf a}=(a_1,\ldots,a_N)$ and
degrees ${\bf d}=(d_1,\ldots,d_N)$  realize the minimum in
\eqref{eq:96}. 
\end{theorem}

    The main feature of \rth{th:main2} is that it
   provides a simple criterion for the computation of the location of the
   singularities of the limit $u_*$ (whence of $u_*$ itself) for a
   given boundary condition $g$. For simplicity we describe it for the
   case $\Omega=B_1$. What one has to do is to find the nearest point
   projection of $g$ on the set $\mathcal{H}_D(\partial B_1)$ (i. e.,
   the set of
   all the finite Blaschke products with $D$ factors)  with respect to
   the metric $d_{H^{1/2}}$. The singularities 
   of $u_*$ are precisely the zeros of the (extension to $B_1$ of the) Blaschke product which
   is the nearest point projection. To be exact, since we do not know whether the
   nearest point projection of $g$ is unique, what we can only say is that
   $u_*$ (any possible limit of a subsequence of $\{u_\varepsilon\}$) is one of the nearest point projections of $g$. In the
   special case where $g\in\mathcal{H}_D(\partial B_1)$ we can
   immediately say that the singularities of $u_*$ coincide the zeros of
   $g$. Actually, much more can be said in this case: for each
   $\varepsilon>0$ we have an
   explicit formula for the (unique) minimizer $u_\varepsilon$, whose
   zeros are the zeros of $g$, see \rprop{prop:B} in Subsection\,\ref{subsec:blaschke} for details.
\par Our original motivation to study the energy $E_\varepsilon$ came from  Ericksen's model for nematic liquid crystals with variable degree of
orientation \cite{Ericksen91}. In this model the nematic, confined to
a domain $\Omega\subset\R^3$,  is described
by a pair $(s,n)$ with $s:\Omega\to\left(-\frac{1}{2},1\right)$ and
$n:\Omega\to S^2$. In its simplest form the energy of the nematic is given by 
\begin{equation}
\label{eva}
{\mathcal {F}}_{E}(s,n)=\int_\Omega \left\{{k\left|\nabla s\right|}^2+s^2{\left|\nabla n\right|}^2+f(s)\right\},
\end{equation}
 for some smooth potential function $f:\left(-\frac{1}{2},1\right)\to\R_+$ that vanishes at a single point $s_0\in\left(-\frac{1}{2},1\right)$ and diverges at the endpoints of its interval of definition.  A further
 simplification of the model can be achieved once we realize that the field $s$ can be forced to deviate not too much from $s_0$ in $\Omega$ by setting $s|_{\partial\Omega}=s_0$ and taking advantage of the fact that variations of $s$ are penalized by the corresponding gradient term in \eqref{eva}. Here, larger values of the parameter $k$ would result in smaller values of $s-s_0$ in $\Omega$. Hence, we drop the potential $f$ in
 \eqref{eva}, similarly to what Ambrosio and Virga did in \cite{AmbrosioVirga} for different reasons (see
 also \cite{Virga90,mrv}).  A
 possible physical justification for dropping $f$ for polymeric liquid crystals was given in \cite{Virga90}. More recently, the same simplification was used in a numerical work \cite{NWW} when simulating nematic configurations arising within the Ericksen model. 
 
   To see the connection between the energy \eqref{eq:energy} to the
one in \eqref{eva} we follow F.H.\,Lin~\cite{lin89} by representing
the pair $(s,n)$ (in the case $s\ge0$)  by a single vector--valued function
   $u=sn$, where
    $u:\Omega\to\R^3$, so that $s=|u|$ while $n=u/|u|$ on the set
    $\{s>0\}$. This allows us to rewrite the energy in \eqref{eva}, in
    the case $f=0$, as:
 \begin{equation}\label{eq:Genergy}
G_k(u)=\int_{\Omega}\Big((k-1)|\nabla|u||^2+|\nabla u|^2\Big).
\end{equation}
Replacing the parameter $k$ with
$\varepsilon=(1/k)^{1/2}$ we get that $G_k(u)=E_\varepsilon(u)$ with
$E_\varepsilon$ given by \eqref{eq:energy}. Note however that in
\eqref{eq:energy} we consider $\R^2$--valued maps, while here the
physical model leads us to consider $\R^3$--valued maps; we will return
to this point below.

Associating with $u$ the map $w:\Omega\to{\mathcal C}_k^3$ given by
$w(x)=(u(x),\sqrt{k-1}\,|u|(x))$ (assuming $k>1$) we notice that   $w$
takes its values in a {\em
      circular cone} $ {\mathcal C}_k^3$ given by
    \begin{equation}
  \label{eq:204}
  {\mathcal C}_k^3=\{(y,t)\in\R^3\times\R\,:\,t=\sqrt{k-1}\,|y|\}\,.
\end{equation}
Moreover, $(k-1)|\nabla|u||^2+|\nabla u|^2=|\nabla w|^2$, whence
\begin{equation*}
  G_k(u)=\int_{\Omega}|\nabla w|^2.
\end{equation*}
Hence the energy \eqref{eva} without the potential term has another
   interpretation, leading to the study of minimizing harmonic
   maps taking values in  the circular cone ${\mathcal C}_k^3$.
   Properties
of these maps, in particular their regularity, were studied
extensively by Lin~\cite{lin89,lin91} and Hardt and
Lin~\cite{hl93}. Delicate regularity results for the analogous problem
when the cone ${\mathcal C}_k^3$ is replaced by a cone over the real projective plane
were obtained recently by Alper, Hardt and Lin~\cite{ahl}
and Alper~\cite{Alper}.

Replacing the parameter $k$ with
$\varepsilon=(1/k)^{1/2}$ we get that $G_k(u)=E_\varepsilon(u)$ with
$E_\varepsilon$ given by \eqref{eq:energy}. There are however two
special features in the problem that are not present in the standard
physical model. These are the assumptions that both the domain and the
target are two-dimensional, i.e., the extension $w$ of the original
$\R^2$--valued map $u$ takes values in 
$$ {\mathcal
  C}_k^2=\{(y,t)\in\R^2\times\R\,:\,t=\sqrt{k-1}\,|y|\}\,.
$$
These assumptions need some justification. In the Appendix, we present a possible physical
 motivation that led us to consider the present model, by showing that it
 can be derived as a thin film limit of a problem set in three
 dimensions. 

The fact that our model constrains the minimizers to
 take values in $\R^2$ and, hence, prevents them from \enquote{escaping to the third
    dimension} (see \cite{BBCH}) is crucial to our results. This is
  the reason why the
energy of the minimizers blows up as $\varepsilon\to0$ (equiv.\,$k\to\infty$) whenever the boundary condition has a nonzero degree, leading to emergence of singularities in the limit $\varepsilon\to0$. Indeed, the energy of
minimizers, for the same boundary condition but for ${\mathcal
  C}_k^3$--valued maps, remains bounded uniformly in $k$. The
extra dimension of the target also influences the uniqueness
issue. In fact, for minimizing harmonic maps taking values in the
upper hemisphere  $S^2_{+}$, 
uniqueness holds
\enquote{almost always} (and in particular whenever the domain
$\Omega$ has a connected boundary). This is a special case of a result
by Sandier and Shafrir~\cite{SS} that takes advantage of a certain convexity property of the
energy that holds thanks to the extra dimension. This result was strengthened and generalized---using a different
elegant technique---in a recent paper by Ignat, Nguyen, Slastikov and
Zarnescu~\cite{INSZ}. We expect a similar phenomenon to hold in our
setting, that is,  for a boundary condition taking values in $S^1$ and each $k>1$ 
there should be a unique $({\mathcal C}_k^3)_{+}$--valued minimizing
harmonic map. Here $({\mathcal C}_k^3)_{+}$ is defined by
\begin{equation*}
 ({\mathcal C}_k^3)_{+}=\{(y,t)\in\R^3\times\R\,:\,t=\sqrt{k-1}\,|y|,\,y_3\ge0\}\,,
\end{equation*}
and the $({\mathcal C}_k^3)_{+}$--valued minimizer is one of two minimizers for
the problem among ${\mathcal C}_k^3$--valued maps -- the other one
is $({\mathcal C}_k^3)_{-}$--valued, and is obtained from the first one
via reflection w.r.t.~the plane $\{y_3=0\}$. For the problem we are
concerned with in  this manuscript, that is for ${\mathcal C}_k^2$--valued
maps, the question of uniqueness is widely open. The only two modest
results we have in that respect are \rprop{prop:B} that  establishes
uniqueness, for every $\varepsilon>0$,  in the special case of a boundary condition which is a
Blaschke product and \rth{th:uniq} that establishes
uniqueness when the boundary condition has degree zero and
$\varepsilon$ is sufficiently small.

\par In a related work, Ignat, Nguyen, Slastikov and
Zarnescu~\cite{INSZ20} considered a two-dimensional problem on
a disk, involving the Landau--de Gennes model for nematic liquid crystals. They proved that, for sufficiently large radius and a symmetric boundary condition carrying a topological defect of degree $D/2$ (for $D$ even), there exist exactly two minimizers---both retaining the symmetry of
the boundary data---as well as non--minimizing critical points with $D-$fold symmetry. An interesting open problem for us is whether for $\varepsilon\ll1$ and the boundary data
$g(e^{i\theta})=e^{iD\theta}$ with $D\ge2$ there exist {\em
  local minimizers} of $E_\varepsilon$ having $D$ vortices, each of degree one, arranged in a symmetric pattern. If true, the techniques required to prove this fact in our case, will likely have to be significantly different from \cite{INSZ20} because harmonic maps for the limiting problem with an even $D$ in \cite{INSZ20} have bounded energy due to escape into the third dimension.

 The paper is organized as follows. In
 Section~\ref{sec:case-bound-cond} we examine the case $\deg
 g=0$. The rest of the paper is devoted to the case $\deg
 g\ge1$. Section~\ref{sec:prelim} contains some preliminary
 results needed for the proof of the main
 theorems. Section~\ref{sec:main1} is devoted to the proof of
 \rth{th:main1} while Section~\ref{sec:main2} is devoted to the proof of
 \rth{th:main2}. The proof of \rprop{prop:W} is given in
 	Section~\ref{sec:prop}. Finally, in the Appendix, we outline the dimension reduction argument that motivates our model.

\subsubsection*{Acknowledgments} 
The research of the first author
(DG) was supported in part by the NSF grant DMS-1615952.  The research
of the second author (IS) was supported by the 
Israel Science Foundation (Grant No. 894/18).
 We thank the anonymous referees for their valuable comments and suggestions that helped us  to improve the presentation of the
 manuscript.

\tableofcontents

\section{Boundary condition of degree zero}
\label{sec:case-bound-cond}
Throughout this section we suppose that $g:\partial\Omega\to S^1$ is a
smooth boundary condition of degree zero and let $g=e^{i\varphi_0}$.
Denote by $\widetilde\varphi_0$ the harmonic extension of $\varphi_0$
and let $u_0=e^{i\widetilde\varphi_0}$.
We mention in passing that the inequality $|u_\varepsilon(x)|\le1$
always holds in
$\Omega$ (regardless of the value of $\deg g$). 
 Indeed, otherwise we could
 reduce the energy by replacing $u_\varepsilon(x)$ by
 $u_\varepsilon(x)/|u_\varepsilon(x)|$ on the set $\{x\in\Omega :
 |u_\varepsilon(x)|>1\}$. An alternative argument that yields the same
  inequality uses the
 sub-harmonicity of the function $|u_\varepsilon|^2$, see
 \eqref{eq:41} below.

\subsection{Convergence of the minimizers}
\label{sec:conv-minim}
\begin{proposition}
	\label{prop:H1}
We have $u_\varepsilon\to u_0$ strongly in $H^1(\Omega)$ as $\varepsilon\to0$.
\end{proposition}
\begin{proof}
    Since
  \begin{equation}
    \label{eq:205}
    E_\varepsilon(u_\varepsilon)=\int_{\Omega}\Big(|\nabla u_\varepsilon|^2+\big(\frac{1}{\varepsilon^2}-1\big)|\nabla|u_\varepsilon||^2\Big)\le
E_\varepsilon(u_0)=\int_{\Omega}|\nabla u_0|^2\,,
  \end{equation}
  there is a subsequence
satisfying $u_{\varepsilon_n}\rightharpoonup u$ weakly in $H^1(\Omega)$. Therefore,
\begin{equation}\label{eq:uu0}
\int_\Omega |\nabla u|^2\leq\liminf_{n\to\infty}\int_{\Omega }|\nabla u_{\varepsilon_n}|^2\leq \int_{\Omega}|\nabla u_0|^2.
\end{equation}
Denoting $\rho_\varepsilon=|u_\varepsilon|$ and $\bar\rho_\varepsilon=\frac{1}{|\Omega|}\int_{\Omega}\rho_\varepsilon$ we have by Poincar\'e inequality:
	\begin{equation}
	\label{eq:poincare}
	\int_{\Omega} |\rho_\varepsilon-\bar\rho_\varepsilon|^2\le C\int_{\Omega}|\nabla\rho_\varepsilon|^2\to 0.
	\end{equation}
	Passing to a further subsequence we may assume that $\bar\rho_{\varepsilon_n}\to R$  for some constant $R\in[0,1]$ and then by \eqref{eq:poincare}, $\rho_{\varepsilon_n}\to R$ strongly in $H^1$. It follows that $1=\text{Tr}(\rho_{\varepsilon_n})\to\text{Tr}(R)=R$ in $L^2(\partial\Omega)$, whence $R=1$. 
	It follows that $u\in H^1_g(\Omega;S^1)$ and the inequality 
	$\int_{\Omega} |\nabla u|^2\leq \int_{\Omega} |\nabla u_0|^2$
        implies that $u=u_0$. From \eqref{eq:uu0} we conclude
        that $u_{\varepsilon_n}\to u_0$ strongly in $H^1$,
        and the full convergence
          $u_{\varepsilon}\overset{H^1}\longrightarrow u_0$ follows
          from the uniqueness of $u_0$.
          Going back to \eqref{eq:205} we  deduce that also
          \begin{equation}
            \label{eq:206}
            \lim_{\varepsilon\to0} \big(\frac{1}{\varepsilon^2}-1\big)\int_\Omega|\nabla|\rho_\varepsilon||^2=0\,.
          \end{equation}
        \qed
	\end{proof}
	\begin{proposition}
	\label{prop:uniform}
	Under the same assumptions as in \rprop{prop:H1} we have:
        $\rho_\varepsilon\to 1$ uniformly on $\Omega$. More precisely,
        we have
        \begin{equation}\label{eq:prop}
	1-\rho_\varepsilon(x) \le{C}\varepsilon\,,\quad\forall x\in\Omega.
	\end{equation}
      \end{proposition}
      \begin{remark}
	In \rth{th:deg-zero} below we will improve the estimate in
        \eqref{eq:prop} to $1-\rho_\varepsilon(x) \le C\varepsilon^2$.
\end{remark}

\begin{proof}[of \rprop{prop:uniform}]
 Throughout the proof we will denote by $C$ different generic constants whose
 value is independent of $\varepsilon$.
 Thanks to the conformal invariance, we may assume that $\Omega=B_1$. 
	By \rprop{prop:H1} and, in particular, \eqref{eq:206} we have 
	\begin{equation}
	\label{eq:fromp}
	(\nabla u_\varepsilon,\big({(1/\varepsilon^2)-1}\big)^{1/2}\nabla\rho_\varepsilon)\xrightarrow{L^2} (\nabla u_0,0). 
	\end{equation}
	Therefore, for any $\delta_0\in(0,1)$ we can find $r_0>0$ such that
	\begin{equation}\label{eq:r0}
	\int_{B_{r_0}(x_0)\cap\Omega} |\nabla u_\varepsilon|^2+((1/\varepsilon^2)-1)|\nabla\rho_\varepsilon|^2\le\delta_0,~\forall x_0\in\Omega.
	\end{equation}
	For reasons to become clear later we fix a value of $\delta_0>0$ satisfying 
	\begin{equation}\label{eq:delta0}
	\delta_0<\frac{1}{4\pi}.
	\end{equation}
	In the sequel we shall suppress for simplicity the subscript
        $\varepsilon$ and write for short, $u=u_\varepsilon$,
        $\rho=\rho_\varepsilon$, etc. Recall that we also have
	 \begin{equation}\label{eq:bound-E}
	\int_{\Omega} |\nabla u|^2+\big((1/\varepsilon^2)-1\big)|\nabla\rho|^2\le C_0:=\int_\Omega |\nabla u_0|^2.
	\end{equation}
	
	We first consider the case $x_0=0$.
	By \eqref{eq:r0} we may  choose $r_0^\prime\in(r_0/2,r_0)$ such that
	 \begin{equation}\label{eq:fubini}
	 \int_{\partial B_{r_0^\prime}}  |\nabla
         u|^2+\big((1/\varepsilon^2)-1\big)|\nabla\rho|^2\leq
         \frac{2}{{r_0^\prime}}\int_{B_{r_0}\setminus B_{r_0/2}} |\nabla u|^2+\big((1/\varepsilon^2)-1\big)|\nabla\rho|^2\leq \frac{2\delta_0}{r_0^\prime}. 
	 \end{equation}
	 In particular, we deduce from \eqref{eq:fubini} that 
	 \begin{equation}\label{eq:close}
	 |u(x_1)-u(x_2)|\leq \int_{\partial B_{r_0^\prime}}\!\!|\nabla
         u|\leq (2\pi r_0^\prime)^{1/2}\left(\int_{\partial
             B_{r'_0}}\!\!|\nabla u|^2\right)^{1/2}\leq \big(4\pi\delta_0\big)^{1/2},~\forall x_1,x_2\in\partial B_{r_0^\prime}.
	 \end{equation}
	 Similarly,
	 \begin{equation}\label{eq:mod-close}
	 |\rho(x_1)-\rho(x_2)|\leq \big(4\pi\delta_0\big)^{1/2}\varepsilon,\quad\forall x_1,x_2\in\partial B_{r_0^\prime}.
	 \end{equation}
	 
	 We  next define the radial function 
	 \begin{equation}\label{eq:bar-rho}
	 \bar\rho(r)=\frac{1}{2\pi r}\int_{\partial B_r}\rho\,d\tau,\quad r\in(0,1].
	 \end{equation}
	 By \eqref{eq:bound-E} we have
	 \begin{equation}
	 \label{eq:721}
	 C_0\varepsilon^2\ge \int_{B_1\setminus B_{r_0^\prime}} |\nabla\rho|^2\ge\int_{B_1\setminus B_{r_0^\prime}} |\nabla\bar\rho|^2=
	 2\pi \int_{r'_0}^1 \left|\frac{d\bar\rho}{dr}\right|^2r\,dr\ge \frac{2\pi(1-\bar\rho(r_0^\prime))^2}{\ln(1/r_0^\prime)}\,,
	 \end{equation}
	 whence 
	 \begin{equation}\label{eq:rho-r'}
	 1-\bar\rho(r_0^\prime)\leq \left\{ \left(\frac{C_0\varepsilon^2}{2\pi}\right)\ln(1/r_0^\prime)\right\}^{1/2}\,.
	 \end{equation}
	 By \eqref{eq:mod-close} and \eqref{eq:rho-r'} we get that
         \begin{equation}
           \label{eq:174}
           \rho=1+O(\varepsilon)\text{ on }\partial B_{r_0^\prime},
         \end{equation}
         while  \eqref{eq:close} and \eqref{eq:delta0} imply that 
	 \begin{equation}
           \label{eq:203}
	 |u(x_1)-u(x_2)|<1\;\text{ on }\partial B_{r_0^\prime}.
	 \end{equation}
	 In particular, it follows from \eqref{eq:174}--\eqref{eq:203}
         that the image of $u/|u|\Big|_{\partial B_{r_0^\prime}}$ is
         contained {\em strictly} in $S^1$ (for sufficiently small $\varepsilon$), whence $\deg(u/|u|,\partial B_{r_0^\prime})=0$. Therefore we may write $u=\rho e^{i\varphi}$ on $\partial B_{r_0^\prime}$.
	 \par Denote by $\widetilde\rho$ and $\widetilde\varphi$ the harmonic
         extensions of $\rho$ and $\varphi$, respectively, from
         $\partial B_{r_0^\prime}$ to $B_{r_0^\prime}$. Recall that
         in dimension two any harmonic function $h$ satisfies:
	 \begin{equation}\label{eq:H2}
	 \int_{B_R} |\nabla h|^2\le R\int_{\partial B_R} \left|\frac{\partial h}{\partial\tau}\right|^2\,.
	 \end{equation}
	 Using \eqref{eq:H2} and the fact that $\rho^2\ge1/2$ on
         $\partial B_{r_0^\prime}$ (for small $\varepsilon$) thanks to
         \eqref{eq:174}, we obtain:
	   \begin{multline}
	   \label{eq:312}
	  \int_{B_{r_0/2}} |\nabla u|^2+\left(\frac{1}{\varepsilon^2}-1\right)|\nabla\rho|^2\le  \int_{B_{r'_0}} |\nabla u|^2+\left(\frac{1}{\varepsilon^2}-1\right)|\nabla\rho|^2\\
	   \le \int_{B_{r'_0}} \widetilde\rho^2|\nabla\widetilde\varphi|^2+\frac{1}{\varepsilon^2}|\nabla\widetilde\rho|^2\le 
	   r'_0\int_{\partial B_{r_0^\prime}}\left|\frac{\partial\varphi}{\partial\tau}\right|^2+\frac{1}{\varepsilon^2}\left|\frac{\partial\rho}{\partial\tau}\right|^2\\
	   \le 2r'_0\int_{\partial B_{r_0^\prime}}\rho^2\left|\frac{\partial\varphi}{\partial\tau}\right|^2+\frac{1}{\varepsilon^2}\left|\frac{\partial\rho}{\partial\tau}\right|^2
	   \le 2r'_0\int_{\partial B_{r_0^\prime}}|\nabla u|^2+\left(\frac{1}{\varepsilon^2}-1\right)|\nabla\rho|^2\\ \le 4\int_{B_{r_0}\setminus B_{r_0/2}}|\nabla u|^2+\left(\frac{1}{\varepsilon^2}-1\right)|\nabla\rho|^2\,,
	  \end{multline}
	  where in the last inequality we used \eqref{eq:fubini}.
	  An immediate consequence of \eqref{eq:312} is
	  \begin{equation}\label{eq:8/9}
	  \int_{B_{r_0/2}} |\nabla u|^2+\big(\frac{1}{\varepsilon^2}-1\big)|\nabla\rho|^2\le\frac{4}{5}\int_{B_{r_0}} |\nabla u|^2+\big(\frac{1}{\varepsilon^2}-1\big)|\nabla\rho|^2\le \frac{4\delta_0}{5}\,.
	  \end{equation}
	  \par Next, we set $r_1=r_0/2$ and choose, as in \eqref{eq:fubini}, $r'_1\in(r_1/2,r_1)$ such that 
	  \begin{equation}
	  \label{eq:fubini1}
	   \int_{\partial B_{r_1^\prime}}  |\nabla u|^2+\big(\frac{1}{\varepsilon^2}-1\big)|\nabla\rho|^2\leq \frac{2}{{r_1^\prime}}\int_{B_{r_1}\setminus B_{r_1/2}} |\nabla u|^2+\big(\frac{1}{\varepsilon^2}-1\big)|\nabla\rho|^2\leq \frac{4}{5}\cdot\frac{2\delta_0}{r_1^\prime}\,. 
	  \end{equation}
	 Similarly to \eqref{eq:close}--\eqref{eq:mod-close} we get
	 \begin{equation}
	 \label{eq:close1}
	 \begin{aligned}
	 |u(x_1)-u(x_2)|&\leq \Big(4\pi\delta_0\cdot(4/5)\Big)^{1/2}\\
	 |\rho(x_1)-\rho(x_2)|&\leq \Big(4\pi\delta_0\cdot(4/5)\Big)^{1/2}\varepsilon
	 \end{aligned}
	 ~,\quad \forall x_1,x_2\in\partial B_{r_1^\prime}.
\end{equation}
By a similar argument to the one used in \eqref{eq:721} we get 	 
\begin{equation}
\label{eq:722}
\delta_0\varepsilon^2\ge 
\int_{B_{r'_0}\setminus B_{r_1^\prime}} |\nabla\rho|^2
\ge \int_{B_{r'_0}\setminus B_{r_1^\prime}} |\nabla\bar\rho|^2 
\ge \frac{2\pi}{\ln(r'_0/r'_1)}|\bar\rho(r'_0)-\bar\rho(r'_1)|^2\geq \frac{2\pi}{\ln 4}|\bar\rho(r'_0)-\bar\rho(r'_1)|^2,
\end{equation}
whence
\begin{equation}\label{eq:723}
|\bar\rho(r'_0)-\bar\rho(r'_1)|\le \left(\frac{\delta_0\ln 4}{2\pi}\right)^{1/2}\varepsilon\,.
\end{equation}
Using the harmonic extensions of $\rho$ and $\varphi$ from $\partial B_{r'_1}$ to $B_{r'_1}$, as in \eqref{eq:312}, we obtain, analogously to \eqref{eq:8/9}:
\begin{equation}\label{eq:8/9-1}
\int_{B_{r_1/2}} |\nabla u|^2+\big(\frac{1}{\varepsilon^2}-1\big)|\nabla\rho|^2\le\frac{4}{5}\int_{B_{r_1}} |\nabla u|^2+\big(\frac{1}{\varepsilon^2}-1\big)|\nabla\rho|^2\le \left(\frac{4}{5}\right)^2\delta_0\,.
\end{equation}
\par We continue by defining recursively $r_j=r_{j-1}/2=r_0/2^j$ and then choose $r'_j\in(r_{j+1},r_j)$ satisfying
  \begin{equation}
\label{eq:fubini-j}
\int_{\partial B_{r_j^\prime}}  |\nabla u|^2+\big(\frac{1}{\varepsilon^2}-1\big)|\nabla\rho|^2\leq \frac{2}{{r_j^\prime}}\int_{B_{r_j}\setminus B_{r_j/2}} |\nabla u|^2+\big(\frac{1}{\varepsilon^2}-1\big)|\nabla\rho|^2\leq \Big(\frac{2}{r_j^\prime}\Big)\Big(\frac{4}{5}\Big)^j\delta_0\,. 
\end{equation}
Analogously to \eqref{eq:close1} we get
\begin{equation}
\label{eq:closej}
\begin{aligned}
|u(x_1)-u(x_2)|&\leq \Big(4\pi\delta_0\Big)^{1/2}\cdot\Big(\frac{4}{5}\Big)^{j/2}\\
|\rho(x_1)-\rho(x_2)|&\leq \Big(4\pi\delta_0\Big)^{1/2}\cdot\Big(\frac{4}{5}\Big)^{j/2}\varepsilon
\end{aligned}
~,\quad \forall x_1,x_2\in\partial B_{r_j^\prime}.
\end{equation}
The argument used to obtain \eqref{eq:723} yields
\begin{equation}\label{eq:724}
|\bar\rho(r'_{j-1})-\bar\rho(r'_j)|\le \left(\frac{\delta_0\ln 4}{2\pi}\right)^{1/2}\cdot\Big(\frac{4}{5}\Big)^{(j-1)/2}\varepsilon\,.
\end{equation}
Combining \eqref{eq:rho-r'} with \eqref{eq:724} gives
\begin{multline}\label{eq:bound}
1-\bar\rho(r'_j)\leq 1-\bar\rho(r'_0)+\sum_{i=1}^j|\bar\rho(r'_{i-1})-\bar\rho(r'_{i})|\\
\leq \left\{ \left(\frac{C_0}{2\pi }\right)\ln(2/r_0)\right\}^{1/2}\varepsilon+\left\{\sum_{i=1}^{j}\left(\frac{4}{5}\right)^{(i-1)/2}\right\}\Big(\frac{\delta_0\ln 4}{2\pi}\Big)^{1/2}\varepsilon\le C\varepsilon\,.
\end{multline}
Letting $j$ go to infinity in \eqref{eq:bound} yields $1-\rho(0)\leq C\varepsilon$, which is \eqref{eq:prop} for $x=0$.
\par Finally we consider the case $x\in B_1\setminus\{0\}$. First,
denote by $d_h$ the hyperbolic metric in $B_1$ with the convention
that $d_h(0,x)=\tanh^{-1}|x|$  (it is
half of the hyperbolic distance commonly used in Geometry).
 In particular, Let $D_r(x)$ denote hyperbolic disk of radius $r$, centered at $x$, that is 
\[ 
D_r(x)=\{y\in B_1:\,d_h(x,y)<r\}.
 \]
 For a given $x\ne 0$ and $r_0$ as in \eqref{eq:r0} we let $\widetilde
 r_0=\tanh^{-1} r_0$, so that $D_{\tilde r_0}(x)=M_x\big(D_{\tilde
   r_0}(0)\big)=M_x(B_{r_0})$, where $M_{x}$ denotes the M\"obius transformation sending $0$
 to $x$.  It is easy to see that $D_{\tilde r_0}(x)=B_{s}(y)$, for some $y\in B_1$ and $s<r_0$. 
  By \eqref{eq:r0} and the conformal invariance of the energy we obtain that $v:=u\circ M_{x}$ satisfies
  \begin{equation*}
  \int_{B_{r_0}} |\nabla v|^2+\big(\frac{1}{\varepsilon^2}-1\big)|\nabla|v||^2= \int_{B_{s}(y)} |\nabla u|^2+\big(\frac{1}{\varepsilon^2}-1\big)|\nabla\rho|^2\le \delta_0\,.
  \end{equation*}
  By the first part of the proof, $1-|u(x)|=1-|v(0)|\le C\varepsilon$ and \eqref{eq:prop} follows.\qed
\end{proof}	
In the next theorem we improve the estimate \eqref{eq:prop}.
\begin{theorem}
	\label{th:deg-zero}
	For a smooth boundary condition $g=e^{i\varphi_0}$ of degree zero we have:
\begin{equation}
\label{eq:thd0}
\|u_\varepsilon-u_0\|_{C^m(\overline{\Omega})}\le{C_m}\varepsilon^2\,,\quad\forall m\ge1.
\end{equation}
\end{theorem}
\begin{proof}
 Note that for $v$ with no zeros, i.e., of the form $v=\rho e^{i\varphi}$, the energy in $\eqref{eq:energy}$ takes the form
  \begin{equation}
  \label{eq:333}
 E_\varepsilon(v)=\int_{\Omega} \rho^2|\nabla\varphi|^2+\frac{1}{\varepsilon^2}|\nabla\rho|^2.
 \end{equation}
  By \rprop{prop:uniform}, for $\varepsilon$ small enough,  any
  minimizer $u=u_\varepsilon$ can be written as $u=\rho_\varepsilon e^{i\varphi_\varepsilon}=\rho e^{i\varphi}$. It follows from \eqref{eq:333} that the Euler-Lagrange system for $\rho$ and $\varphi$ reads
\begin{equation}\label{eq:EL}
\left\{
\begin{aligned}
&\Div(\rho^2\nabla\varphi)=0,\\
&-\Delta\rho+\varepsilon^2\rho|\nabla\varphi|^2=0.
\end{aligned}
\right.
\end{equation}  	
We write $\varphi=\widetilde{\varphi}_0+\psi$ which allows us to write the equation satisfied by $\psi$ as
\begin{equation}\label{eq:psi}
\left\{
\begin{aligned}
&\Delta\psi=\Div((1-\rho^2)\nabla\varphi)\text{ in }\Omega,\\
&\psi=0\text{ on }\partial\Omega.
\end{aligned}
\right.
\end{equation}
For any $p>2$ we have by standard elliptic estimates and \eqref{eq:prop},
\begin{equation*}
\|\nabla\psi\|_{p}\le C\|(1-\rho^2)\nabla\varphi\|_p\le
C\varepsilon\|\nabla\varphi\|_p\leq C\varepsilon(1+\|\nabla\psi\|_p).
\end{equation*}  	
It follows that $\|\nabla\psi\|_p\le C\varepsilon$, whence 
\begin{equation}\label{eq:nab-phi}
\|\nabla\varphi\|_p\le C_p\,,\quad\forall p>2\,.
\end{equation}
Plugging \eqref{eq:nab-phi} in the second equation in \eqref{eq:EL},
yields $\|\Delta\rho\|_p\leq C_p\varepsilon^2$, $\forall p>1$, whence,
since $1-\rho=0$ on $\partial\Omega$, 
\begin{equation}\label{eq:rho1}
\|1-\rho\|_{W^{2,p}}\le C_p\varepsilon^2\,,\quad\forall p>1.
\end{equation}
Using the first equation in \eqref{eq:EL} we obtain that 
\begin{equation}
\label{eq:Dpsi}
-\Delta\psi=-\Delta\varphi=\frac{2}{\rho}\left(\nabla\rho\cdot\nabla\varphi\right)\,,
\end{equation}
so we can now conclude from \eqref{eq:nab-phi} and \eqref{eq:rho1} that
$\|\Delta\varphi\|_p\leq C_p\varepsilon^2$, $\forall p>1$. Hence by
elliptic estimates we get that also
\begin{equation}\label{eq:psi1}
\|\psi\|_{W^{2,p}}\le C_p\varepsilon^2\,,\quad\forall p>1.
\end{equation}
Next we claim that:
\begin{equation}\label{eq:jp}
\|\psi\|_{W^{j,p}}+\|1-\rho\|_{W^{j,p}}\leq C_{j,p}\varepsilon^2\,,\quad\forall p>1,\forall j\ge 2.
\end{equation}
 We prove \eqref{eq:jp} by induction on $j$. For $j=2$ the result holds by \eqref{eq:rho1} and \eqref{eq:psi1}. Assuming the result holds for $j$, we see from \eqref{eq:Dpsi} that $\|\Delta\psi\|_{W^{j-1,p}}\le C\varepsilon^2$, implying that $\|\psi\|_{W^{j+1,p}}\le C_{j,p}\varepsilon^2$.
 Similarly, the estimate for $\|1-\rho\|_{W^{j+1,p}}$ follows from the second equation in \eqref{eq:EL}.
 Finally, \eqref{eq:thd0} follows from \eqref{eq:jp} and Sobolev
 embeddings. \qed
\end{proof}
\subsection{Uniqueness of the minimizers for small $\varepsilon$}
\begin{theorem}
  \label{th:uniq}
If $g$ is a smooth boundary condition of degree zero then there exists
$\varepsilon_0>0$ such that for all $\varepsilon\le \varepsilon_0$ the minimizer $u_\varepsilon$ for $E_\varepsilon$
over $H^1_g(\Omega)$ is unique.
\end{theorem}
\begin{proof}
We follow  an argument from \cite{cm99}.
By \rth{th:deg-zero} there exists $\varepsilon_1$ such that for $\varepsilon\le \varepsilon_1$ any
minimizer $u=u_\varepsilon$ satisfies $1/2\le|u|\le1$. Let
$v=v_\varepsilon$ be any minimizer
for  $\varepsilon\le \varepsilon_1$, whence also  $1/2\le|v|\le1$. We may then write
$u=\rho e^{i\varphi}$ and also $w:=v/u=\eta e^{i\psi}$ with
$1/2\le\eta\le 2$ in $\Omega$,  $\eta=1$ on $\partial\Omega$ and
$\psi=0$ on $\partial\Omega$. A direct computation yields
\begin{multline}
  \label{eq:3}
  E_\varepsilon(v)-E_\varepsilon(u)=\int_{\Omega}
  \rho^2(\eta^2-1)|\nabla\varphi|^2+\rho^2\eta^2(2\nabla\varphi\cdot\nabla\psi+|\nabla\psi|^2)\\
 +\frac{1}{\varepsilon^2}\int_{\Omega} (\eta^2-1)|\nabla\rho|^2+(\rho^2|\nabla\eta|^2+2\rho\eta\nabla\rho\cdot\nabla\eta)\,.
\end{multline}
Next we multiply the second equation in \eqref{eq:EL} by
$\rho(\eta^2-1)$ and integrate over $\Omega$ to find
\begin{equation}
  \label{eq:1}
  \left(\int_\Omega(\eta^2-1)|\nabla\rho|^2+2\rho\eta\nabla\rho\cdot\nabla\eta\right)+\varepsilon^2\int_\Omega\rho^2(\eta^2-1)|\nabla\varphi|^2=0\,.
\end{equation}
 Substituting \eqref{eq:1} in \eqref{eq:3} gives
 \begin{equation}
   \label{eq:2}
   E_\varepsilon(v)-E_\varepsilon(u)=\int_{\Omega}\varepsilon^{-2}\rho^2|\nabla\eta|^2+2\rho^2\eta^2\nabla\varphi\cdot\nabla\psi+\rho^2\eta^2|\nabla\psi|^2.
 \end{equation}
On the other hand, multiplying the first equation in \eqref{eq:EL} by
$\psi$ and integrating, we conclude that
$$
\int_\Omega\rho^2\nabla\varphi\cdot\nabla\psi=0\,.
$$
Plugging it in 
\eqref{eq:2} yields that 
\begin{equation}
  \label{eq:4}
   E_\varepsilon(v)-E_\varepsilon(u)=\int_{\Omega}\varepsilon^{-2}\rho^2|\nabla\eta|^2+2\rho^2(\eta^2-1)\nabla\varphi\cdot\nabla\psi+\rho^2\eta^2|\nabla\psi|^2.
\end{equation}
By \rth{th:deg-zero} we have $\|\nabla\varphi\|_\infty\le c_0$ for
some constant $c_0>0$. Hence, by the Cauchy-Schwarz inequality we have
\begin{equation}
  \label{eq:6}
\begin{aligned}
  \left|\int_\Omega
    2\rho^2(\eta^2-1)\nabla\varphi\cdot\nabla\psi\right|&\leq
  4c_0^2\int_\Omega(\eta^2-1)^2+\frac{1}{4}\int_\Omega
  |\nabla\psi|^2\\
&\le 4c_0^2\int_\Omega(\eta^2-1)^2+\int_\Omega
  \rho^2\eta^2|\nabla\psi|^2\,.
 \end{aligned}
\end{equation}
Applying Poincar\'e inequality to the function $\eta^2-1\in
H^1_0(\Omega)$ yields
\begin{equation}
  \label{eq:5}
  \int_\Omega(\eta^2-1)^2\le C_P\int_\Omega|2\eta\nabla\eta|^2\le 16C_p\int_\Omega|\nabla\eta|^2.
\end{equation}
Combining \eqref{eq:6}--\eqref{eq:5} with \eqref{eq:4} yields
\begin{equation*}
  E_\varepsilon(v)-E_\varepsilon(u)\ge\int_{\Omega}(\rho^2/\varepsilon^2-64C_Pc_0^2)|\nabla\eta|^2\ge \frac{1-256\varepsilon^2c_0^2C_P}{4\varepsilon^2}\int_\Omega|\nabla\eta|^2.
\end{equation*}
It follows from the above and our assumption $E_\varepsilon(v)=E_\varepsilon(u)$, that for
$\varepsilon<\frac{1}{16c_0\sqrt{C_P}}$ we must have $|\nabla\eta|=0$ in $\Omega$, whence
$\eta\equiv 1$. Plugging it in \eqref{eq:4} we finally get that
$\psi\equiv0$ and the equality $v=u$ follows. \qed
\end{proof}
\begin{remark}
  We do not know whether the uniqueness result of \rth{th:uniq} holds
  without the assumption that $\varepsilon$ is sufficiently small.
\end{remark}

\section{Boundary condition of degree $D\ge1$: preliminary estimates}
\label{sec:prelim}
 In this section we consider the case of boundary condition of nonzero
 degree. Without loss of generality we assume that $\deg g=D\ge 1$. We
 continue to assume that $\Omega$ is a smooth, bounded and simply
 connected domain in $\R^2$; whenever convenient, we will suppose that
 $\Omega$ is the unit disc $B_1=B_1(0)$. 
 
\subsection{Minimization within the radial class}
Consider the case $\Omega=B_R=B_R(0)$ and $g(Re^{i\theta})=e^{iD\theta}$ with $D\ge1$.
Define
\begin{equation*}
V:=\{f\in H^1_{loc}(0,R)\,:\,\sqrt{r}f',\frac{f}{\sqrt{r}}\in L^2(0,R), f(R)=1\}.
\end{equation*}
For $f\in V$ we have $fe^{iD\theta}\in H^1_g(\Omega)$ and 
\begin{equation*}
E_\varepsilon(fe^{id\theta})=2\pi \int_0^R \left(\frac{f'^2}{\varepsilon^2}+\frac{D^2}{r^2}f^2\right)r\,dr\,.
\end{equation*}
We first solve the minimization problem under the restriction that the
maps satisfy the above \enquote{$D$-radial symmetry} ansatz.  
\begin{lemma}
	\label{lem:rad}
	For every $D\ge1$ and $\varepsilon>0$ we have
\begin{equation*}
\min_{f\in V}E_\varepsilon(fe^{iD\theta})=\frac{2\pi D}{\varepsilon}
\end{equation*}	
and the unique minimizer is 
\begin{equation}
\label{eq:bar}
\bar f_{\varepsilon,D}(r)=\left(\frac{r}{R}\right)^{D\varepsilon}.
\end{equation}
\end{lemma}
\begin{proof}
	First we note that for every $f\in V$ the following pointwise inequality  holds on $(0,R)$:
	\begin{equation}
	\label{eq:cs}
	 \frac{rf'^2}{\varepsilon^2}+\frac{D^2f^2}{r}=\left(\frac{\sqrt{r}f'}{\varepsilon}\right)^2+\left(\frac{Df}{\sqrt{r}}\right)^2\geq \frac{2}{\varepsilon}ff'D.
	\end{equation}
	Integration of \eqref{eq:cs} over the interval $(0,R)$ yields $E_\varepsilon(fe^{iD\theta})\geq \frac{2\pi D}{\varepsilon}$.
	\par Equality holds in \eqref{eq:cs} iff
	\begin{equation}
	\label{eq:equality}
	 \sqrt{r}f'=Df\varepsilon/\sqrt{r}\text{  a.e. on } (0,R).
	\end{equation}
	A simple integration of \eqref{eq:equality} yields $f=\bar f_{\varepsilon,D}$ as given in \eqref{eq:bar}.\qed
      \end{proof}
 We remark that the special solutions given by \eqref{eq:bar} are
 well-known in the literature. They appeared for example in \cite{mrv}
 as part of the study of axially symmetric minimizers. In the next
 subsection, see \rcor{cor:rad} below, 
 we will prove that $\bar f_{\varepsilon,D}e^{iD\theta}$ is the
 minimizer for $E_\varepsilon$ over the whole class $H^1_g(B_R)$ (for
 $g(Re^{i\theta})=e^{iD\theta}$), i.e., without assuming the $D$-radial symmetry ansatz.  
 \subsection{Asymptotic behavior of the energy}
   In this subsection we will prove the following asymptotic formula
   for the energy: $E_\varepsilon(u_\varepsilon)=\frac{2\pi
   D}{\varepsilon}+O(1)$. We start with the  lower bound.
\begin{proposition}
	\label{prop:lb}
        Assume $g:\Omega\to S^1$ has degree $D>0$. Then we have
        \begin{equation}
\label{eq:lb}
E_\varepsilon(u)\geq \frac{2\pi D}{\varepsilon},~\forall u\in H^1_g(\Omega).
\end{equation}
\end{proposition}
\begin{proof}
 By density of smooth maps in $H^1_g(\Omega)$ it
  suffices to prove \eqref{eq:lb} for {\em smooth} $u$.
Applying the Cauchy-Schwarz inequality gives
\begin{equation}\label{eq:CS}
E_\varepsilon(u)=\int_{\Omega\cap\{u\neq 0\}} \!\!\!\Big(\varepsilon^{-2}|\nabla|u||^2+|u|^2|\nabla (u/|u|)|^2\Big)\geq 
\frac{2}{\varepsilon}\int_{\Omega\cap\{u\neq 0\}} \!\!\!|\nabla|u|||u||\nabla (u/|u|)|.
\end{equation}
For each $t\in(0,1)$ set $\gamma_t=\{x\in\Omega\,:\,|u|=t\}$. For
almost every $t\in(0,1)$, $\gamma_t$ is a union of a finite number closed smooth
curves, $C_1,C_2,\ldots,C_m$. For such $t$, the boundary of the set
$\Omega_t:=\{x\in\Omega\,:\,|u(x)|>t\}$ consists of
$\partial\Omega\cup\bigcup_{k=1}^mC_k$, whence, the total winding
number of the map $u/|u|:\gamma_t\to S^1$ 
around the origin equals $D$. Hence,
\begin{equation}
  \label{eq:44}
  \int_{\gamma_t} |\nabla (u/|u|)|\ge \big| \int_{\gamma_t} (u/|u|)\wedge
  (u/|u|)_\tau \,d\tau\big|=2\pi D\,.
\end{equation}
The direction chosen for the unit vector $\tau$ was
  dictated by the requirement  
   that $(\nu,\tau)$
will be a direct frame, where $\nu$ denotes the inward unit normal to $\Omega_t$.
Applying the coarea formula to the R.H.S.~of \eqref{eq:CS}, using \eqref{eq:44}, yields
\begin{equation}
	\label{eq:111}
E_\varepsilon(u)\ge \frac{2}{\varepsilon}\int_0^1 \int_{\gamma_t} t|\nabla (u/|u|)|\,d\tau\,dt\geq\frac{4\pi D}{\varepsilon}\int_0^1 t\,dt=\frac{2\pi D}{\varepsilon},
\end{equation}
and \eqref{eq:lb} follows.\qed
\end{proof}
\begin{corollary}
\label{cor:energy}	
We  have 
\begin{equation}
  \label{eq:97}
  \frac{2\pi D}{\varepsilon} \le E_\varepsilon(\ue)\le \frac{2\pi D}{\varepsilon}+C.
\end{equation}
\end{corollary}
\begin{proof}
	The lower bound follows from \rprop{prop:lb}. W.l.o.g.~we may
        assume that $0\in\Omega$. Fix any $R>0$ such that
        $B_R\subset\Omega$. For each $\varepsilon$ let $U_\varepsilon$
        be equal to
        $\bar f_{\varepsilon,D}(r)e^{iD\theta}$ in $B_R$ (see
        \eqref{eq:bar}) and complete it in $\Omega\setminus B_R$ by any $S^1$--valued smooth map which equals $e^{iD\theta}$ on $\partial B_R$ and $g$ on $\partial\Omega$. By  \rlemma{lem:rad} we have $E_\varepsilon(U_\varepsilon)\leq \frac{2\pi D}{\varepsilon}+C$.\qed
\end{proof}
\begin{corollary}
  \label{cor:rad}
	For $\Omega=B_R$ and $g(Re^{i\theta})=e^{iD\theta}$, the map
        $\bar f_{D}(r)e^{iD\theta}$, with
        $\bar f_{D}$ as in \eqref{eq:bar}, is a minimizer for $E_\varepsilon$ over $H^1_g(\Omega)$.
\end{corollary}
\begin{proof}
 This is an immediate consequence of \rlemma{lem:rad} and \rprop{prop:lb}.\qed
\end{proof}
\begin{remark}
  From \rprop{prop:B} below it follows that  $\bar
  f_{D}(r)e^{iD\theta}$ is the unique minimizer over $H^1_g(\Omega)$ for the boundary
  condition $g(Re^{i\theta})=e^{iD\theta}$.
\end{remark}
\begin{remark}
  The combination of the proof of \rprop{prop:lb}  with the result of \rcor{cor:energy} demonstrates that the principle of
  \enquote{equipartition of the energy} holds for our problem, i.e.,
  the contributions of the two terms in
 $E_\varepsilon(u_\varepsilon)$ are  essentially equal. It is
  well-known that this principle holds for {\em scalar} problems, like
  $\Gamma$-convergence of the
  Modica-Mortola functional, see
  \cite{modica,modica-mortola,sternberg} or its vector--valued
  analogues, see e.g.~\cite[Section 2]{sternberg} and
  \cite{fonseca-tartar}. In these works the equipartition is
associated with phase-separation and, more specifically, with the profile of a minimizer being asymptotically one-dimensional. The equipartition of energy in
our problem is of a quite different nature. Roughly speaking, it results from the approximate pointwise
equality $$\varepsilon|\nabla(\ln\rho_\varepsilon)|\sim
|\nabla\varphi_\varepsilon|$$ holding for a minimizer which can be
written locally as
$u_\varepsilon=\rho_\varepsilon e^{i\varphi_\varepsilon}$.  Thus, in our case, equipartition reflects a strong coupling
between the phase and the modulus of a minimizer.
\end{remark}
      \subsection{When the boundary condition is a Blaschke
        product}
      \label{subsec:blaschke}
In this subsection we will show that the case considered above
in \rcor{cor:rad}, where we were able to give a simple explicit formula
for the minimizers for each fixed $\varepsilon$, is a special case of
a more general family of boundary data. In fact, let $\Omega=B_1$ and
$g=F|_{\partial B_1}$ where $F\in C(\overline B_1)$ is analytic
function on $B_1$ that sends $\partial B_1$ to itself. It is well-known that such $F$ must be a finite Blaschke product, i.e., of the form
\begin{equation}\label{eq:B}
F(z)=e^{i\alpha}\Prod_{j=1}^D \left(\frac{z-a_j}{1-\bar a_jz}\right),
\end{equation}
for some $\alpha\in\R$ and $D$ points $a_1,a_2,\ldots,a_D$  (not necessarily distinct) in
$B_1$. Note that the choice $a_j=0$, $\forall j$ (and $\alpha=0$)
corresponds to the $D$-symmetric boundary data considered above. 
\begin{proposition}
	\label{prop:B}
	When $\Omega=B_1$ and $g=F|_{\partial B_1}$ with  $F$ as in
        \eqref{eq:B} we have for each $\varepsilon$: the map
        $U(z):=U_\varepsilon(z)=|F(z)|^{\varepsilon}\left(\frac{F(z)}{|F(z)|}\right)$
        is the unique minimizer of $E_\varepsilon$ over $H^1_g(B_1)$.
\end{proposition}
\begin{proof}
{\rm (i)} We first prove that $U$ is a minimizer.	Let us denote $\widetilde{\rho}=|F|$ in $B_1$ and $h=\ln|\widetilde\rho|$  in $E:=B_1\setminus\{a_1,\ldots,a_D\}$. 
	Locally in $E$ we may write
        $F=\widetilde{\rho}e^{i\varphi}=e^{h+i\varphi}$. The function
        $\varphi$ is then a harmonic conjugate of the harmonic
        function $h$, locally in $E$.
         Note that although $\varphi$ is defined only locally in $E$,
         its gradient is globally defined there since
         $\nabla\varphi=(F/|F|)\wedge\nabla(F/|F|)$.
      In particular, the equality
	\begin{equation}
	\label{eq:nab}
	\left|\nabla\left(\frac{F(z)}{|F(z)|}\right)\right|=|\nabla h|\,,
      \end{equation}
      holds globally in $E$.
	Consider $U$, defined in $B_1$ as in the statement of the proposition, i.e., $U=\rho\frac{F(z)}{|F(z)|}$ with $\rho=(\widetilde{\rho})^\varepsilon$, so locally in $E$ we have $U=\rho e^{i\varphi}$.  
	\par Next we notice that for $u=U$, the Cauchy-Schwarz
        inequality used in \eqref{eq:CS} reduces to an
        equality. Indeed, we need the pointwise equality
        $|\nabla\rho|/\varepsilon=\rho|\nabla\varphi|$, which is
        equivalent to
        \begin{equation}
	\label{eq:777} |\nabla(\ln\rho)|/\varepsilon=|\nabla\varphi|.
	\end{equation}
	Since $\ln\rho=\varepsilon\ln h$ we finally deduce \eqref{eq:777} from \eqref{eq:nab}. To sum-up, so far we proved that 
	\begin{equation}\label{eq:sofar}
	E_\varepsilon(U)=\frac{2}{\varepsilon}\int_E
        |\nabla\rho||\nabla(F/|F|)|=
        \frac{2}{\varepsilon} \int_E |\nabla\rho||\nabla(U/|U|)|.
	\end{equation}
	\par Next we continue to follow the proof of \rprop{prop:lb} for the case $u=U$. We denote
	\begin{equation}\label{eq:Gamma}
	\Gamma=\{t\in(0,1)\,:\,t\text{ is a regular value of }\rho\}=\{t\in(0,1)\,:\,t\text{ is a regular value of }\widetilde\rho\}.
	\end{equation}
	Clearly $\Gamma$ has full measure in $(0,1)$. For each
        $t\in\Gamma$ the set $\gamma_t:=\{\rho^{-1}(t)\}$ consists of
        a finite union of smooth closed curves, each encircles some of
        the points $\{a_1,\ldots,a_D\}$ (and the union of them
        encircle all the points). At each point of $\gamma_t$, with
        $t\in\Gamma$, we have
        $\frac{\partial\varphi}{\partial\tau}=\frac{\partial(\ln\widetilde{\rho})}{\partial\nu}>0$,
        since $|\nabla\widetilde{\rho}|>0$ on $\gamma_t$. Moreover,
        $\frac{\partial\varphi}{\partial\nu}=-\frac{\partial(\ln\widetilde{\rho})}{\partial\tau}=0$. Whence, for each $t\in\Gamma$ there holds 
	\begin{equation}\label{eq:ong}
	\int_{\gamma_t} |\nabla (U/|U|)|\,d\tau=\int_{\gamma_t} |\partial_\tau\varphi|\,d\tau=\int_{\gamma_t}\partial_\tau\varphi\,d\tau=\int_{\gamma_t} (U/|U|)\wedge (U/|U|)_\tau= =\int_{\partial\Omega} (U/|U|)\wedge (U/|U|)_\tau=2\pi D.
      \end{equation}
	Using \eqref{eq:ong} in conjunction with the coarea formula as in \eqref{eq:111} gives
	\begin{equation}\label{eq:3333}
	\int_E |\nabla\rho||\nabla(U/|U|)|=\int_\Gamma \int_{\gamma_t} t|\nabla (U/|U|)|\,d\tau\,dt= (2\pi D)\int_0^1 t\,dt=\pi D.
	\end{equation}
	Combining \eqref{eq:3333} with \eqref{eq:sofar} yields
        $E_\varepsilon(U)=2\pi D/\varepsilon$. Applying
        \rprop{prop:lb} we finally conclude that indeed $U$ is a
        minimizer. \\[2mm]
    (ii) Next we will prove the uniqueness assertion
    -- 
    the main idea is to analyse the equality cases in the inequalities
    we used in part (i). We shall need a result
  of F.\,H.\,Lin~\cite{lin89} who derived (in a more general setting) 
 the Euler-Lagrange equation satisfied by
 $\rho_\varepsilon^2=|u_\varepsilon|^2$:
 \begin{equation}
    \label{eq:41}
    \Delta(\rho_\varepsilon^2)=2\varepsilon^2\left(|\nabla u_\varepsilon|^2+\left(\frac{1}{\varepsilon^2}-1\right)|\nabla\rho_\varepsilon|^2\right)\,,
  \end{equation}
  by using variations of the form
$u^{(t)}=(1+t\phi(x))u_\varepsilon$ (the same equation can
be deduced from the second equation in \eqref{eq:EL} on the set $\{\rho_\varepsilon>0\}$).  
In particular, the function $\rho^2_{\varepsilon}$ is subharmonic in $\Omega$.
 We next recall
    known regularity properties of any minimizer $u_\varepsilon$. By
    \cite{hl93,lin91}, the function $u_\varepsilon$ is H\"older continuous in ${\bar
      B}_1$ and is real analytic in $B_1\setminus {\mathcal S}$, where
    ${\mathcal S}$ is a
    finite set consisting of the zeros of $u_\varepsilon$. By
    \eqref{eq:41} and the strong maximum principle
    $\rho_\varepsilon=|u_\varepsilon|$ satisfies
    $0<\rho_\varepsilon<1$ in $B_1$. Moreover, by Hopf lemma,
    $\frac{\partial(\rho^2_\varepsilon)}{\partial n}>0$ on $\partial
    B_1$. Next we fix $r_0$ satisfying  $\displaystyle\max_{1\le j\le D}|a_j|<r_0<1$ and
    \begin{equation}
      \label{eq:209}
      \frac{\partial\rho_\varepsilon}{\partial r}>0 \text{ in }
    \{r_0\le|x|\le1\}\,,
    \end{equation}
    and let
    \begin{equation*}
      T:=\max_{\{|x|=r_0\}}\rho_\varepsilon(x)\,.
    \end{equation*}
Since by assumption
    $E_\varepsilon(u_\varepsilon)=E_\varepsilon(U)=\frac{2\pi
      D}{\varepsilon}$, equality must hold when we plug
    $u_\varepsilon$ for $u$ in all the inequalities in
    \eqref{eq:CS},\eqref{eq:44},\eqref{eq:111}. Since by the maximum
    principle $\rho_\varepsilon<T$ in $B_{r_0}$, for each $t\in(T,1)$
    we have $\bar\gamma_t:=\{\rho_\varepsilon^{-1}(t)\}\subset
    \{r_0<|x|<1\}$. Moreover, thanks to \eqref{eq:209} every
    $t\in(T,1)$ is a regular value of $\rho_\varepsilon$, whence 
    the set $\bar\gamma_t$ consists of a single closed smooth curve (since the
    topology of $\bar\gamma_t$ should be the same as that of $\bar\gamma_1=S^1$ in the
    absence of a critical point).
    For each such $t$  we write locally
    $u_\varepsilon=\rho_\varepsilon e^{i\varphi_\varepsilon}$ on the
    curve $\bar\gamma_t$. 
    From the pointwise equality,
    $\nabla(u_\varepsilon/|u_\varepsilon|)=(u_\varepsilon/|u_\varepsilon|)\wedge(u_\varepsilon/|u_\varepsilon|)_\tau$
    that  holds in \eqref{eq:44}, we obtain that
    $\nabla\varphi_\varepsilon=\vec\tau\perp\nabla\rho_\varepsilon$,
  i.e., taking into account orientation, 
    \begin{equation}
      \label{eq:211}
     \big( (\varphi_\varepsilon)_x,
     (\varphi_\varepsilon)_y\big)=\lambda(x,y)\big(-(\rho_\varepsilon)_y,
     (\rho_\varepsilon)_x\big)~\text{ locally on }\bar\gamma_t\,,
   \end{equation}
    for some $\lambda(x,y)>0$.
    On the other hand,  the pointwise equality
    $$\frac{|\nabla\rho_\varepsilon|}{\varepsilon\rho_\varepsilon}=\left|\nabla\Big(\frac{u_\varepsilon}{|u_\varepsilon|}\Big)\right|\,,$$
    that must hold for $u=u_\varepsilon$ in
   \eqref{eq:CS} on each $\bar\gamma_t$ for
   $t\in(T,1)$, can be rewritten as
   \begin{equation}
     \label{eq:212}
     |\nabla(\ln\rho_\varepsilon^{1/\varepsilon})|=|\nabla\varphi_\varepsilon|~\text{
       locally on }\bar\gamma_t,\,t\in(T,1).
   \end{equation}
    Combining \eqref{eq:211} with \eqref{eq:212} we deduce that
    locally on the set $\bigcup_{T<t<1} \bar\gamma_t$, the pair of
    functions $\ln\rho_\varepsilon^{1/\varepsilon}$ and 
    $\varphi_\varepsilon$ satisfy the Cauchy-Riemann equations
    ($\varphi_\varepsilon$ being a complex conjugate of $\ln\rho_\varepsilon^{1/\varepsilon}$).
    In particular, this holds locally  on some annulus
    $A_{r_1}=\{r_1<|x|<1\}\subset\bigcup_{T<t<1} \bar\gamma_t$ for
    some $r_1\in(r_0,1)$ and {\em globally} on   $A_{r_1}\setminus{\mathcal
      L}$ where ${\mathcal L}=\{(s,0)\,:\,r_0<s<1\}$. Since $U=u_\varepsilon$ on $\partial B_1$
    we have $\varphi=\varphi_\varepsilon+2\pi J$ on $\partial B_1\setminus{\mathcal
      L}$ for some $J\in\Z$, and we may assume w.l.o.g. that
    $J=0$. Since we also have
    $\frac{\partial\varphi}{\partial\nu}=\frac{\partial\varphi_\varepsilon}{\partial\nu}=0$
    on $\partial B_1\setminus{\mathcal L}$ (using
    $\rho=\rho_\varepsilon=1$ on $\partial B_1$ and the Euler-Lagrange
    equations) and since a harmonic function is uniquely determined by
    its values and the values of its normal derivative on the
    boundary, we deduce that $\varphi=\varphi_\varepsilon$ in $A_{r_1}\setminus{\mathcal
      L}$. By the Euler-Lagrange equations we obtain that also
    $\rho=\rho_\varepsilon$ in $A_{r_1}\setminus{\mathcal
      L}$, whence $u_\varepsilon=U$ in $A_{r_1}$. Finally, as both
    $u_\varepsilon$ and $U$ are real analytic in
    $B_1\setminus({\mathcal S}\cup\{a_j\}_{j=1}^D)$ that coincide on
    the open subset $A_{r_1}$, they must coincide also on
    $B_1\setminus({\mathcal S}\cup\{a_j\}_{j=1}^D)$, and then also on $B_1$.
    \qed
      \end{proof}
      \section{Proof of \rth{th:main1}}
      \label{sec:main1}
 This section is devoted to the proof of \rth{th:main1}. Note that
 estimate \eqref{eq:42} was already established in \rcor{cor:energy}.
 In most of this section we assume, as we may w.l.o.g., that
 $\Omega=B_1$. 
 \subsection{Construction of the bad discs}
 \label{subsec:bd}
   The objective of this subsection is to show that the set where
 $|u_\varepsilon|$ is close to zero is \enquote{small}. This is
 established in \rprop{prop:bad-disc} below, where we show that for some
 $\beta<1$ the set $\{|u_\varepsilon|<\beta\}$ can be covered
 by a finite collection of discs of small radii whose number is
 bounded uniformly in $\varepsilon$. This is the same approach
 as that used in \cite{BBH} for studying minimizers of the Ginzburg-Landau
 energy, but the technique we use here is different. 

Recall that by  \rcor{cor:energy} we have for some $c_1>0$,
\begin{equation}\label{eq:upper}
E_\varepsilon(u_\varepsilon)\le \frac{c_1}{\varepsilon}\,,~\forall
\varepsilon\in(0,1).
\end{equation}
In the sequel we fix a $\beta\in(1/\sqrt{2},1)$ that for reasons to become
clear later we assume to satisfy
\begin{equation}
  \label{eq:16}
  \beta^2>\frac{D}{D+1}\,.
\end{equation}

\smallskip
\begin{lemma}
\label{lem:1/2}	
	Let $u_\varepsilon$ be a minimizer satisfying
	\begin{equation}\label{eq:cond}
	r_0\int_{\partial B_{r_0}}|\nabla u_\varepsilon|^2+\Big(\frac{1}{\varepsilon^2}-1\Big)|\nabla\rho_\varepsilon|^2\le\delta_0\,,
	\end{equation}
	 with $\delta_0$ as in \eqref{eq:delta0}  and $r_0>0$ satisfying 
	 \begin{equation}\label{eq:73}
	 \frac{c_1\varepsilon}{2\pi}\ln(1/r_0)<\frac{(1-\beta)^2}{4}.
	 \end{equation}
Then,  for $\varepsilon<\varepsilon_0$ we have
\begin{equation}\label{eq:74}
  |u_\varepsilon(0)|\ge \beta\,.
\end{equation}	 
\end{lemma}	
\begin{proof}
  For simplicity we shall drop the subscript $\varepsilon$.
Analogously to \eqref{eq:close} and \eqref{eq:mod-close} we have
\begin{align}\label{eq:close-r0}
|u(x_1)-u(x_2)|&\leq \int_{\partial B_{r_0}}\!\!|\nabla u|\leq (2\pi r_0)^{1/2}\Big(\int_{\partial B_{r_0}}\!\!|\nabla u|^2\Big)^{1/2}\leq \sqrt{2\pi\delta_0},~\forall x_1,x_2\in\partial B_{r_0},\\
\label{eq:mod-close-r0}
|\rho(x_1)-\rho(x_2)|&\leq \sqrt{2\pi\delta_0}\,\varepsilon,~\forall x_1,x_2\in\partial B_{r_0}.
\end{align}
Defining $\bar\rho$ as in \eqref{eq:bar-rho}, we find by \eqref{eq:upper}, analogously to \eqref{eq:721}: 
\begin{equation}
\label{eq:1721}
c_1\varepsilon\ge \int_{B_1\setminus B_{r_0}} |\nabla\rho|^2\ge\int_{B_1\setminus B_{r_0}} |\nabla\bar\rho|^2=
2\pi \int_{r_0}^1 \left|\frac{d\bar\rho}{dr}\right|^2r\,dr\ge \frac{2\pi(1-\bar\rho(r_0))^2}{\ln(1/r_0)}\,,
\end{equation}
whence, by \eqref{eq:73}
\begin{equation}\label{eq:1722}
1-\bar\rho(r_0)\leq \left\{ \frac{c_1\varepsilon}{2\pi}\ln(1/r_0)\right\}^{1/2}<\frac{1-\beta}{2}\,.
\end{equation}
By \eqref{eq:1722} and \eqref{eq:mod-close-r0} we get that 
\begin{equation*}
1-\rho(x)\leq \frac{1-\beta}{2}+O(\varepsilon)\text{  on }\partial B_{r_0},
\end{equation*}
so in particular,
\begin{equation}
  \label{eq:12}
  \rho^2\ge(3/4)^2>1/2~\text{ on }\partial B_{r_0}\,.
\end{equation}
From \eqref{eq:12} and  \eqref{eq:close-r0} we conclude that
$\deg(u/|u|,\partial B_{r_0})=0$. Therefore we may write on $\partial
B_{r_0}$, $u=\rho e^{i\varphi}$. Using the harmonic extensions of
$\rho$ and $\varphi$, as in the proof of \rprop{prop:uniform}, we
obtain (using \eqref{eq:H2} and \eqref{eq:cond}) that
 \begin{multline}
	   \label{eq:3122}
	  \int_{B_{r_0}} |\nabla u|^2+\left(\frac{1}{\varepsilon^2}-1\right)|\nabla\rho|^2
	   \le \int_{B_{r_0}} |\nabla\widetilde\varphi|^2+\frac{1}{\varepsilon^2}|\nabla\widetilde\rho|^2\le 
	   r_0\int_{\partial B_{r_0}}\left|\frac{\partial\varphi}{\partial\tau}\right|^2+\frac{1}{\varepsilon^2}\left|\frac{\partial\rho}{\partial\tau}\right|^2\\
	   \le 2r_0\int_{\partial B_{r_0}}\rho^2\left|\frac{\partial\varphi}{\partial\tau}\right|^2+\frac{1}{\varepsilon^2}\left|\frac{\partial\rho}{\partial\tau}\right|^2
	   \le 2r_0\int_{\partial B_{r_0}}|\nabla u|^2+\left(\frac{1}{\varepsilon^2}-1\right)|\nabla\rho|^2\le 2\delta_0.
	  \end{multline}

Next we continue as in the proof of \rprop{prop:uniform}, defining $r_j=r_0/2^j$ for $j\ge1$ and choosing successively, for $j\ge0$, 
 $r'_j\in(r_{j+1},r_j)$ satisfying 
 \begin{equation}
 \label{eq:fubini-jj}
 \int_{\partial B_{r_j^\prime}}  |\nabla u|^2+(1/\varepsilon^2-1)|\nabla\rho|^2\leq \frac{2}{{r_j^\prime}}\int_{B_{r_j}\setminus B_{r_j/2}} |\nabla u|^2+(1/\varepsilon^2-1)|\nabla\rho|^2\,. 
\end{equation}
This allows us to conclude, arguing as in \eqref{eq:8/9-1} and \eqref{eq:fubini-j}, that
\begin{equation}
  \label{eq:175}
\int_{B_{r_{j+1}}} |\nabla u|^2+(1/\varepsilon^2-1)|\nabla\rho|^2\le
\frac{4}{5}\int_{B_{r_{j}}} |\nabla
u|^2+(1/\varepsilon^2-1)|\nabla\rho|^2\le 2\delta_0\left(\frac{4}{5}\right)^{j+1}\,.
\end{equation}
Combining \eqref{eq:175} with \eqref{eq:fubini-jj} yields
\begin{equation*}
   \int_{\partial B_{r_j^\prime}}  |\nabla u|^2+(1/\varepsilon^2-1)|\nabla\rho|^2\leq \frac{4\delta_0}{r_j^\prime}\left(\frac{4}{5}\right)^j\,,
\end{equation*}
implying, in particular, that
\begin{equation}
  \label{eq:176}
  |\rho(x)-\rho(y)|\le
  (8\pi\delta_0)^{1/2}\varepsilon\left(\frac{4}{5}\right)^{j/2},~\forall x,y\in \partial B_{r_j^\prime}.
\end{equation}
As in \eqref{eq:723} and \eqref{eq:724} we get that
\begin{equation}
  \label{eq:13}
  |\bar\rho(r'_{j-1})-\bar\rho(r'_j)|\le \left(\frac{\delta_0\ln 4}{\pi}\right)^{1/2}\cdot\Big(\frac{4}{5}\Big)^{(j-1)/2}\varepsilon\,.
\end{equation}
Therefore, analogously  to \eqref{eq:bound} we obtain 
 \begin{equation}\label{eq:bound-9}
 1-\bar\rho(r'_j)\leq 1-\bar\rho(r'_0)+\sum_{i=1}^j|\bar\rho(r'_{i-1})-\bar\rho(r'_{i})|
 \leq  \frac{1-\beta}{2}+\left\{\sum_{i=1}^{j}\left(\frac{4}{5}\right)^{(i-1)/2}\right\}\left(\frac{\ln 4}{\pi}\delta_0\right)^{1/2}\varepsilon\,.
 \end{equation}
 Thanks to \eqref{eq:176}--\eqref{eq:bound-9}, we have for each
 $j$ that $1-\rho(x)\le \frac{1-\beta}{2}+O(\varepsilon)$ on $\partial
 B_{r'_j}$, which allows us to continue with the
 construction. Finally, letting $j$ go to $\infty$ in
 \eqref{eq:bound-9} we get that $1-\rho(0)\le
 (1-\beta)/2+O(\varepsilon)$ so, in particular, \eqref{eq:74} holds
 for $\varepsilon<\varepsilon_0$. \qed
\end{proof}
\begin{definition} $~$\\
	(i) We shall say that $0$ is a {\em good point} for $u_\varepsilon$ if there exists $r_0$ satisfying \eqref{eq:73} and \eqref{eq:cond}. \\
	(ii) We shall say that $a\in B_1$ is a {\em good point} of $u_\varepsilon$ if $0$ is a {good point} for $v_\varepsilon=u_\varepsilon\circ M_{a}$.
 \end{definition}
 \begin{corollary}$~$\\
 \label{cor:good}
\noindent (i) If $a$ is a good point for $u_\varepsilon$ then $|u_\varepsilon(a)|\ge\beta$. \\
\noindent (ii) If $|u_\varepsilon(0)|<\beta$ then 
\begin{align}\label{eq:bad}
r\int_{\partial B_r} |\nabla
  u_\varepsilon|^2+\big((1/\varepsilon)^2-1\big)|\nabla\rho_\varepsilon|^2&>\delta_0\quad\mathrm{when}\quad r>\rho_0(\varepsilon), \\
 \label{eq:Ebad}
\int_{B_{r_2}\setminus B_{r_1}} |\nabla u_\varepsilon|^2+\big((1/\varepsilon)^2-1\big)|\nabla\rho_\varepsilon|^2&>\delta_0 \ln(r_2/r_1)\ \mathrm{when}\   1>r_2>r_1\ge\rho_0(\varepsilon).
\end{align}
Here
$\rho_0(\varepsilon):=\exp(-\frac{\pi(1-\beta)^2}{2c_1\varepsilon})=\exp(-\frac{c_2}{\varepsilon})$,
with $c_2=\frac{\pi(1-\beta)^2}{2c_1}$.

\noindent (iii) There exists a constant $c_3>0$ such that, if 
$|u_\varepsilon(a)|<\beta$, then for
  $\rho_1=\rho_1(\varepsilon):=\rho_0(\varepsilon)^{1/2}$
we have
\begin{equation}
\label{eq:Ebad-eps}
\int_{D_{\tanh^{-1}\rho_1}(a)} |\nabla u_\varepsilon|^2+\big((1/\varepsilon)^2-1\big)|\nabla\rho_\varepsilon|^2\ge c_3/\varepsilon\,.
\end{equation}
 \end{corollary}
\begin{proof}
Assertion (i) and \eqref{eq:bad} are  immediate consequences of
\rlemma{lem:1/2}. The inequality \eqref{eq:Ebad} follows by integration of
\eqref{eq:bad}. The case $a=0$ in (iii) follows from (ii) applied with
$r_1=\rho_0(\varepsilon)$ and $r_2=\rho_1(\varepsilon)$. For general $a$ we use conformal
invariance. \qed
\end{proof}
\begin{definition}
	We denote the set of {\em bad points} of $u_\varepsilon$ by 
	\begin{equation}\label{eq:S}
	S=S_\varepsilon=\{x\in B_1\,:\,|u_\varepsilon(x)|<\beta\}.
	\end{equation}
      \end{definition}
 
      \begin{proposition}  
	\label{prop:bad-disc}	
	For each $\varepsilon>0$, there is a set of $m=m_\varepsilon$
        points
        $$\{x_j\}_{j=1}^m=\{x_j^{(\varepsilon)}\}_{j=1}^m\subset S,$$
        and a number $R=R^{(\varepsilon)}$ satisfying
        \begin{equation}
          \label{eq:208}
          -\frac{A}{\varepsilon}<\ln R<-\frac{B}{\varepsilon},\ \mbox{where }A,B>0,
        \end{equation}
         such that the
        (hyperbolic) discs $\{D_{\tanh^{-1}R}(x_j)\}_{j=1}^m$ are
        mutually disjoint and the following properties  hold:
  \begin{align}
    \label{eq:cover}
    	& (a)\ S\subset\bigcup_{j=1}^m D_{\tanh^{-1}(R)}(x_j)\,,\\
                  \label{eq:207}
    & (b)\ m\le N, \text{ for some $N$  independently of $\varepsilon$},\\
    \label{eq:bound-kappa}
   & (c)\ \kappa_j=\kappa_j^\varepsilon:=\deg(u_\varepsilon/|u_\varepsilon|,\partial
  D_{\tanh^{-1}R}(x_j))\in[1,C], \,\forall j, \text{ for some constant
             }C\ge1,\\
    \label{eq:28}
  & (d)\ R\int_{\partial  B_R}\Big(|\nabla (u_\varepsilon\circ M_{x_j})|^2+((1/\varepsilon^2)-1)
  |\nabla(\rho_\varepsilon\circ M_{x_j})|^2\Big)\le c_4,~\forall j,
     \text{ for some constant }c_4>0.
 \end{align}
		\end{proposition}
\begin{proof}
 The proof is divided to several steps.

\medskip\noindent\underline{\color{black}Step 1: Select an initial collection of bad discs.} Let $\rho_1=\rho_1(\varepsilon)$ be as defined in \rcor{cor:good}. Applying Vitali covering lemma for the collection of discs
$\{D_{\tanh^{-1}\rho_1}(x)\}_{x\in S}$ yields a finite collection of
mutually disjoint hyperbolic discs
$\{D_{\tanh^{-1}\rho_1}(x_j)\}_{j=1}^{m_\varepsilon}$, with
$\{x_j\}_{j=1}^{m_\varepsilon}=\{x_j^{(\varepsilon)}\}_{j=1}^{m_\varepsilon}\subset
S$, such that
\begin{equation}
  \label{eq:210}
  S\subset\bigcup_{j=1}^{m_\varepsilon} D_{5\tanh^{-1}\rho_1}(x_j)\,.
\end{equation}
Moreover, thanks to \eqref{eq:Ebad-eps} and \eqref{eq:upper} we have
$m_\varepsilon\le N$ for some $N$, for all $\varepsilon$.

\medskip\noindent\underline{\color{black}Step 2: Construct the final collection of bad discs.} First, we extend $\rho_0$ and $\rho_1$ from \rcor{cor:good} to an infinite sequence by
setting
\begin{equation}
  \label{eq:7}
  \rho_j=\rho_{j-1}^{1/2}\,,~j=2,3,\ldots, \text{ i.e., }\rho_j=\exp(-\frac{c_2}{2^j\varepsilon})\,.
\end{equation}
Consider the collection $\{D_{\tanh^{-1}\rho_2}(x_j)\}_{j=1}^m$ that
clearly covers $S$ by \eqref{eq:210}, since $\tanh^{-1}\rho_2\gg 5\tanh^{-1}(\rho_1)$. If
the discs are mutually disjoint we are done. Otherwise, if for example
$D_{\tanh^{-1}\rho_2}(x_1)\cap D_{\tanh^{-1}\rho_2}(x_2)\ne\emptyset$,
we keep $D_{\tanh^{-1}\rho_2}(x_1)$ and drop
$D_{\tanh^{-1}\rho_2}(x_2)$. We relabel the new centers and keep the
same notation for $m=m_\varepsilon$ (which is strictly smaller than the original
one) and consider the new collection
$\{D_{\tanh^{-1}\rho_3}(x_j)\}_{j=1}^m$ that also covers $S$. If these discs are all
mutually disjoint we are done, otherwise we eliminate some discs, taking
into account the intersections. We continue in this way until we reach $l$ for which
\begin{equation}
  \label{eq:8}
  S\subset\bigcup_{j=1}^m D_{\tanh^{-1}\rho_{l-1}}(x_j)~\text{ and
  }\{D_{\tanh^{-1}\rho_{l-1}}(x_j)\}_{j=1}^m\text{ are mutually disjoint.}
 \end{equation}
 This process must stop after at most $N$ steps ($N$ is given below
 \eqref{eq:210}).

\par
Let us assume for a moment that $x_j=0$, and then
$D_{\tanh^{-1}\rho_{l-1}}(0)=B_{\rho_{l-1}}$. By the upper bound
\eqref{eq:upper} we 
have
\begin{equation}
  \label{eq:9}
  \int_{B_{\rho_{l}}\setminus B_{\rho_{l-1}}}\!\!\!\!|\nabla
  u|^2+((1/\varepsilon^2)-1)|\nabla\rho|^2=\int_{\exp(-\frac{c_2}{2^{l-1}\varepsilon})}^{\exp(-\frac{c_2}{2^{l}\varepsilon})}
  \Big(\int_{\partial B_r}\big(|\nabla u|^2+((1/\varepsilon^2)-1) |\nabla\rho|^2\big)\Big)\,dr
  \le \frac{c_1}{\varepsilon}\,.
\end{equation}
Therefore, we can find $R=R^{(\varepsilon)}\in
\big(\exp(-\frac{c_2}{2^{l-1}\varepsilon}),
\exp(-\frac{c_2}{2^{l}\varepsilon})\big)$ such that
\begin{equation}
  \label{eq:10}
  R\int_{\partial  B_R}\Big(|\nabla
  u|^2+((1/\varepsilon^2)-1) |\nabla\rho|^2\Big)\le \frac{2^{l}c_1}{c_2}:=c_4.
\end{equation}
For a general $x_j$ (not necessarily $x_j=0$) we apply the above to
$\tilde u:=u\circ M_{x_j}$. This gives first $R_j$ such that
\eqref{eq:10} holds for $\tilde u$. Actually we can apply the argument
in a way that insures that the same $R^{(\varepsilon)}:=R=R_j$ works for all $j$, so that
\eqref{eq:28} holds true. This completes the construction of the bad
discs $\{D_{\tanh^{-1}R}(x_j)\}_{j=1}^m$.

\medskip\noindent\underline{\color{black}Step 3: Verify \eqref{eq:bound-kappa}.} Since $\rho\ge\beta$ on $\partial D_{\tanh^{-1}(R)}(x_j) $ for
all $j$, the degree
 $$\kappa_j=\kappa_j^\varepsilon=\deg(u/|u|,\partial D_{\tanh^{-1}(R)}(x_j))\in\Z$$ is well defined and we may write
 \begin{equation*}
   u=\rho e^{i(\kappa_j\theta+\eta)} \text{ on }\partial D_{\tanh^{-1}(R)}(x_j)\,,~j=1,\ldots,m,
 \end{equation*}
 for some scalar function $\eta$.
We first claim that
\begin{equation}
  \label{eq:14}
  |\kappa_j|\le C\,,~\forall j,
\end{equation}
 for some $C$, independently of $\varepsilon$. Again, it suffices to
 consider the case $x_j=0$. The only
 interesting case is when $\kappa_j\ne0$. Applying the argument used in
 the proof of \rprop{prop:lb} yields, denoting this time
 $\gamma_t=\{x\in B_R\,:\,|u|=t\}$,
\begin{equation}
	\label{eq:15}
E_\varepsilon(u;B_R)\ge \frac{2}{\varepsilon}\int_0^\beta
\int_{\gamma_t} t|\nabla (u/|u|)|\,dx\,dt\geq\frac{2}{\varepsilon}
(2\pi|\kappa_j|)\int_0^\beta
t\,dt=\frac{2\pi|\kappa_j|\beta^2}{\varepsilon}> \frac{\pi|\kappa_j|}{\varepsilon}\,,
\end{equation}
since $\beta^2> 1/2$. Our claim \eqref{eq:14} clearly follows from
\eqref{eq:15} and the upper bound \eqref{eq:upper}.

To conclude we show:
\begin{equation}
  \label{eq:17}
  \kappa_j> 0,~\forall j.
\end{equation}
     We first show the weaker inequality
     \begin{equation}
       \label{eq:99}
       \kappa_j\ge 0,~\forall j.
     \end{equation}
   Indeed,  combining \eqref{eq:97} with \eqref{eq:15} yields
    \begin{equation}
      \label{eq:19}
     \frac{2\pi D}{\varepsilon}+C\ge E_\varepsilon(u_\varepsilon)
     \ge\sum_{j=1}^m E_\varepsilon(u_\varepsilon;D_{\tanh^{-1}R}(x_j))\geq\frac{2\pi\beta^2}{\varepsilon}\sum_{j=1}^m|\kappa_j|\,.
   \end{equation}
   Sending $\varepsilon$ to $0$ in \eqref{eq:19} gives
   \begin{equation}
     \label{eq:20}
     D\ge\beta^2 \sum_{j=1}^m|\kappa_j|.
   \end{equation}
   Combining \eqref{eq:20} with \eqref{eq:16} yields
   $$
   D=|\sum_{j=1}^m \kappa_j|\le\sum_{j=1}^m |\kappa_j|<D+1\,.
   $$
   Therefore, necessarily $\sum_{j=1}^m |\kappa_j|=D$, implying that  
 $\kappa_j=|\kappa_j|$ for all $j$, and \eqref{eq:99} follows.

 To prove that the inequality in \eqref{eq:99} is strict, i.e.,
 \eqref{eq:17} holds, 
    we fix one $j$ for which we may assume  w.l.o.g.~that
    $x_j=0$. Looking for contradiction, suppose that
    $\kappa_j=0$. Then we may write $u=\rho e^{i\varphi}$ on $\partial
    B_R$ and let again $\widetilde\rho$ and $\widetilde\varphi$ denote,
    respectively, the harmonic extensions of $\rho$ and $\varphi$ to
    $B_R$. Analogously to \eqref{eq:3122} we obtain, using
    \eqref{eq:28},  that
    \begin{equation}
      \label{eq:98}
      \int_{B_R}|\nabla
      u|^2+\left(\frac{1}{\varepsilon^2}-1\right)|\nabla\rho|^2\le \int_{B_R}|\nabla
      \widetilde\varphi|^2+\frac{1}{\varepsilon^2}|\nabla\widetilde\rho|^2\le
      2c_4.
    \end{equation}
    But the  assumption $|u(0)|<\beta$ implies by
    \eqref{eq:Ebad-eps} that
    \begin{equation*}
      E_\varepsilon(u_\varepsilon;B_R)\ge E_\varepsilon(u_\varepsilon;B_{\rho_1})\ge\frac{c_3}{\varepsilon},
    \end{equation*}
     which clearly contradicts\eqref{eq:98}, for sufficiently small
     $\varepsilon$. \qed
   \end{proof}
   \subsection{Control of the phase oscillations away from the bad
     discs}
   To prove convergence of $\ue$ away from the bad discs the main
   difficulty is to prove  a bound on the oscillations of the
   phase. For that matter we shall use an
   appropriate 
   modification of the strategy employed in \cite{ms} for a different
   problem.  
We denote
\begin{equation}
  \label{eq:39}
  \Omega_\varepsilon=B_1\setminus\bigcup_{j=1}^m
D_{\tanh^{-1}R}(x_j)\,.
\end{equation}
 Whenever there is no
 confusion we shall drop the subscript $\varepsilon$. On $\Omega_\varepsilon$ we may write
\begin{equation}
  \label{eq:21}
  u(z)=\rho e^{i\eta(z)}\prod_{j=1}^m\left(\frac{M_{-x_j}(z)}{|M_{-x_j}(z)|}\right)^{\kappa_j}\,,
\end{equation}
 for some scalar function $\eta=\eta_\varepsilon$, which is unique up
 to addition of an integer multiple of $2\pi$. 
By adding an appropriate multiple of $2\pi$ we may assume then ~that 
\begin{equation}
\label{eq:fix}
\min_{\partial B_1}\eta\in[0,2\pi).
\end{equation}
Since $g$ is smooth, we deduce from \eqref{eq:fix} that
\begin{equation}
  \label{eq:107}
 \|\eta\|_{L^\infty(\partial B_1)}\leq C(g).
\end{equation}
By \eqref{eq:bound-kappa} and \eqref{eq:28}  it follows that
\begin{equation}
  \label{eq:27}
  |\eta(x)-\eta(y)|\le C,\quad \text{ for all }x,y\in \partial D_{\tanh^{-1}R}(x_j),\,j=1,\ldots,m.
\end{equation}
We shall use the following estimate for $\int_{\Omega_\varepsilon}
|\nabla\eta|^2$.
\begin{lemma}
  \label{lem:eta-energy}
  We have
  \begin{equation}
    \label{eq:22}
    \int_{\Omega_\varepsilon}|\nabla\eta|^2\leq \frac{C}{\varepsilon}\,.
  \end{equation}
  \end{lemma}
  \begin{proof}
  By the upper bound \eqref{eq:upper}, the representation
  \eqref{eq:21} and \eqref{eq:bound-kappa}  it suffices to show that
  \begin{equation}
    \label{eq:23}
    \int_{\Omega_\varepsilon}\left|\nabla\left(\frac{M_{-x_j}(z)}{|M_{-x_j}(z)|}\right)\right|^2\le \frac{C}{\varepsilon}\,,~j=1,\ldots,m.
  \end{equation}
  In fact,  \eqref{eq:23} follows easily by using conformal invariance:
  \begin{equation}
    \label{eq:24}
  \begin{aligned}
    \int_{\Omega_\varepsilon}\left|\nabla\left(\frac{M_{-x_j}(z)}{|M_{-x_j}(z)|}\right)\right|^2&\le
    \int_{B_1\setminus
                                                                                                  D_{\tanh^{-1}R}(x_j)}\left|\nabla\left(\frac{M_{-x_j}(z)}{|M_{-x_j}(z)|}\right)\right|^2\\
    &=\int_{B_1\setminus B_R}\left|\nabla\left(\frac{z}{|z|}\right)\right|^2=2\pi\ln\frac{1}{R}\le\frac{C}{\varepsilon}.
  \end{aligned}
\end{equation}
\qed
\end{proof}

Our first step consists of proving an $L^\infty$ bound for $\eta$. We
 will use the method of  selection of ``good rays'', that was  introduced in
 \cite{ms}. This will be done by removing from 
  $\Omega_\ve$ a collection of ``rays'', that in our settings will be
  usually arcs of circles orthogonal to $\partial B_1$, 
   connecting the boundaries 
 of  the holes $\partial D_{\tanh^{-1}R}(x_j)$, $j=1,\ldots, m$,  to the boundary of $B_1$. The choice of these \enquote{good
 rays} will depend on energy considerations.
 Consider first the case where $x_j=0$. For any $\alpha\in[0,2\pi)$
 let
 \begin{equation}
   \label{eq:26}
{\mathcal C}_0(\alpha):=\{ re^{i\alpha}:\,
 r\in[R,1)\}.
 \end{equation}
 In the general case, when $x_j$ is any point in $B_1$, we set
 \begin{equation}
 \label{eq:25}
 {\mathcal C}_{x_j}(\alpha):=\{ M_{x_j}(re^{i\alpha});\, r\in[R,1)\}.
 \end{equation}
Note that for $x_j\ne0$ the set ${\mathcal C}_{x_j}(\alpha)$ is 
an arc of a circle joining $x_j$ to $\partial B_1$ which
is orthogonal to $\partial B_1$ (a geodesic for the hyperbolic metric).

\begin{lemma}
\label{lem:alpha}
There exists $C>0$ such that for each $j=1,\ldots,m$ and
$\varepsilon\in(0,1/2)$ there exists $\alpha_j=\alpha_j
(\ve)\in[0,2\pi)$, such that the following holds,
\begin{equation}
\label{eq:ray}
|\eta(x)-\eta(y)|\le\frac{C}{\varepsilon},\quad \text{ for all }x,y\in
{\mathcal C}_{x_j}(\alpha_j)\cap\Omega_\varepsilon.
\end{equation}
\end{lemma}
\begin{proof}
By \eqref{eq:22} 
 there exists $\alpha_j\in[0,2\pi)$ such that
 \begin{equation}
 \label{eq:choice}
\int_{C_0(\alpha_j)\cap M_{-x_j}(\Omega_\varepsilon)}|\nabla(\eta\circ M_{x_j})|^2\,rdr\leq
\frac{1}{2\pi} \int_{M_{-x_j}(\Omega_\varepsilon)}|\nabla(\eta\circ M_{x_j})|^2
=\frac{1}{2\pi} \int_{\Omega_\varepsilon}|\nabla\eta|^2\le\frac{C}{\varepsilon}.
\end{equation}
Therefore, 
\begin{align*}
\int_{C_0(\alpha_j)\cap M_{-x_j}(\Omega_\varepsilon)} \left|\frac{\partial(\eta\circ
    M_{x_j})}{\partial r}\right|&\leq
\left(\int_{R}^{1}\frac{dr}{r}\right)^{1/2} 
\left(\int_{C_0(\alpha_j)\cap M_{-x_j}(\Omega_\varepsilon)} \left|\frac{\partial(\eta\circ
      M_{x_j})}{\partial r}\right|^2\,rdr\right)^{1/2}
\\
&\leq \left(\ln\left(\frac{1}{R}\right)\right)^{1/2}\left(\frac{C}{\varepsilon}\right)^{1/2}\le\frac{C}{\ve}.
\end{align*}
 Here, $\p/\p r$ stands for the tangential derivative along $C_0(\alpha_j)$.
 It follows that
 \begin{equation*}
|(\eta\circ M_{x_j})(x)-(\eta\circ M_{x_j})(y)|\le\frac{C}{\varepsilon},\quad \text{ for all }x,y\in C_0(\alpha_j)\cap M_{-x_j}(\Omega_\varepsilon),
\end{equation*}
which is clearly equivalent to \eqref{eq:ray}. \qed 
\end{proof}
 Next, we denote 
 $\omega_\ve:=\Omega_\ve\setminus\displaystyle\bigcup_{j=1}^m
 {\mathcal C}_{x_j}(\alpha_j).$
      For each $j$, let $\theta_j$ denote a polar coordinate around the
  point $x_j$, taking values in $[\alpha_j,\alpha_j+2\pi)$ associated
  with the factor $\frac{M_{-x_j}z}{|{M_{-x_j}z}|}$, i.e.,
  \begin{equation}
    \label{eq:29}
    \frac{M_{-x_j}z}{|{M_{-x_j}z}|}=e^{i \theta_j(z)}\,.
  \end{equation}
  Then the function
 \begin{equation}
 \label{eq:Theta}
\Theta=\sum_{j=1}^m \kappa_j\theta_j,
\end{equation}
is  {\it smooth} in  $\omega_\ve$ and satisfies
\begin{equation}
\label{eq:bound-theta}
\|\Theta\|_{L^\infty(\omega_\ve)}\leq 4\pi \sum_{j=1}^m |\kappa_j|.
\end{equation}
We define $\va=\va_\ve:=\eta+\Theta$ in $\omega_\ve$, so that 
\begin{equation*}
u=\rho e^{i\eta}\prod_{j=1}^m
\left(\frac{M_{-x_j}z}{|{M_{-x_j}z}|}\right)^{\kappa_j}=\rho e^{i
  (\Theta+\eta)}=\rho e^{i\va}\text{ in }\omega_\ve.
\end{equation*}
Hence $\va$ is a well-defined phase of $u$ in $\omega_\ve$.
\begin{lemma}
  \label{lem:MP}
  We have for all $0<\varepsilon<1/2$:
  \begin{equation}
    \label{eq:30}
    \|\eta_\varepsilon\|_{L^\infty(\omega_\varepsilon)}\le\frac{C}{\varepsilon}\,.
  \end{equation}
\end{lemma}
\begin{proof}
  First we notice, combining \eqref{eq:107}--\eqref{eq:27} with
  \eqref{eq:ray}, that
  \begin{equation}
  \label{eq:31}
    \|\eta\|_{L^\infty(\partial\omega_\varepsilon)}\le\frac{C}{\varepsilon}\,.
  \end{equation}
Therefore, by the definition of $\varphi$ we have
\begin{equation}
\label{eq:bound-eta-phi}
\limsup_{\delta\to 0}\, \sup\{ |\varphi (x)|:\, x\in\omega_\ve,\, \dist (x, \p\omega_\ve)\le\delta\} 
\leq \frac{C}{\varepsilon}\,.
\end{equation}
We apply the maximum principle to $\varphi$ on each component of the open set
  $\{ x\in\omega_\ve:\, \dist(x, \p\omega_\ve)>\delta\}$, on which
  $\varphi$ satisfies
  \begin{equation*}
    \Div(\rho^2\nabla\varphi)=0\,.
  \end{equation*}
 Then we let $\delta\to 0$ and use \eqref{eq:bound-eta-phi} to obtain
 that
 \begin{equation}
   \label{eq:32}
    \|\varphi\|_{L^\infty(\omega_\varepsilon)}\le\frac{C}{\varepsilon}\,.
 \end{equation}
Finally, \eqref{eq:30} follows from
\eqref{eq:32},\eqref{eq:bound-theta},\eqref{eq:bound-kappa} and the definition
of $\varphi$. \qed
\end{proof}
\subsection{An $L^p$-bound for the gradient, $p\in[1,2)$}
\label{subsec:Lp}
The main result of this subsection is
\begin{proposition}
\label{prop:W1p}
We have 
$\|\nabla u_\varepsilon\|_{L^p(B_1)}\leq C_p$, $1\leq p<2.$ 
\end{proposition}
The following simple
lemma will be needed in the proof of \rprop{prop:W1p}.
 \begin{lemma}
   \label{lem:mobius}
   For every $a\in B_1$ there holds
   \begin{equation}
     \label{eq:37}
     \left|\nabla\left(\frac{M_{-a}(z)}{|M_{-a}(z)|}\right)\right|\le\frac{C}{|z-a|}\,,\quad\forall
     z\in B_1.
   \end{equation}
 \end{lemma}
 \begin{proof}
   Using $(M_{-a})'(z)=\frac{1-|a|^2}{(1-{\bar a}z)^2}$, we get that
   \begin{equation*}
     \left|\nabla\left(\frac{M_{-a}(z)}{|M_{-a}(z)|}\right)\right|\le \frac{C}{|M_{-a}(z)|}\cdot
     \frac{1-|a|^2}{|1-{\bar a}z|^2}
     =\frac{C(1-|a|^2)}{|z-a||1-{\bar a}z|}\le\frac{C}{|z-a|}\,.
   \end{equation*}
   \qed
 \end{proof}
\begin{proof}[of \rprop{prop:W1p}]
Fix any $p\in(1,2)$.
By standard elliptic estimates, there exists a constant
$A_p=A_p(\Omega)$ such that the solution $w$ of the problem
\begin{equation}
  \label{eq:80}
\begin{cases}
  -\Delta w=\Div {\mathbf g} &\text{ in }\Omega\\
    w=0 &\text{ on }\partial \Omega
\end{cases}
,
\end{equation}
 with ${\mathbf g}\in(L^p(\Omega))^2$ satisfies
 \begin{equation}
   \label{eq:81}
 \|\nabla w\|_{L^p(\Omega)}\leq A_p \|{\mathbf g}\|_{L^p(\Omega)}.
\end{equation}
We now apply the bad discs construction of \rprop{prop:bad-disc}, but this time covering the bad set
\begin{equation*}
  S=S_\varepsilon=\{x\in B_1\,:\,|u(x)|<\tilde\beta\}\,,
\end{equation*}
with   $\tilde\beta\in[\beta,1)$  that satisfies 
\begin{equation}
  \label{eq:82}
 0<1-\tilde\beta<\frac{1}{4A_p}\,.
\end{equation}
In the sequel, $\Omega_\varepsilon$ denotes
 the set given in \eqref{eq:39} for the resulting bad discs from this choice of  $\tilde\beta$. Note that the
 number of discs and the value of $l$ may change as well,   but we shall use the same notation
 as before.

Let $H$ denote the harmonic function in $B_1$ satisfying $H=\eta$ on
$\partial B_1$. By \eqref{eq:107} and the maximum principle,
\begin{equation}
  \label{eq:100}
\|H\|_{L^\infty(B_1)}=\|\eta\|_{L^\infty(\partial
  B_1)}\leq C(g).
\end{equation}
Note that \rlemma{lem:mobius} implies that
\begin{equation}\label{eq:mobius-p}
   \left\|\prod_{j=1}^m \left(\frac{M_{-x_j}(z)}{|M_{-x_j}(z)|}\right)^{\kappa_j}\right\|_{W^{1,p}(B_1)}\leq C.
 \end{equation}
 Therefore 
\begin{equation*}
  \|\eta\|_{W^{1-1/p,p}(\partial B_1)}\leq C, 
\end{equation*}
 and also
 \begin{equation}
   \label{eq:86}
\|H\|_{W^{1,p}(B_1)}\leq C.
 \end{equation}

 Next we define the function $\xi_0$ in $B_{R^{1/2}}$ by
  \begin{equation}
  \label{hm2}
 \xi_0(z)=
 \begin{cases}
   0 & |z|\le R\\
    -\frac{\ln(|z|/R)}{\ln \sqrt{R}} & R<|z|<\sqrt{R}
 \end{cases}
.
  \end{equation}
  Note that
  \begin{equation}
    \label{eq:33}
    \int_{B_{\sqrt{R}}}|\nabla\xi_0|^2
    =\frac{2\pi}{(\ln \sqrt{R})^2}\int_{R}^{\sqrt{R}}\frac{dr}{r}=
    -\frac{2\pi}{\ln \sqrt{R}}\le  C\varepsilon\,.
  \end{equation}
   For $j=1,\ldots,m$ we set in ${\cal D}_j:=D_{\tanh^{-1}\sqrt{R}}(x_j)$:
   $\xi_j(z)=\xi_0(M_{-x_j}(z))$. From \eqref{eq:33} we deduce that
   \begin{equation}
     \label{eq:34}
     \int_{{\cal D}_j}|\nabla\xi_j|^2=\int_{B_{\sqrt{R}}}|\nabla\xi_0|^2\le C\varepsilon\,.
   \end{equation}
   We finally define a function $\xi$ in $B_1$ by
   \begin{equation}
     \label{eq:35}
     \xi(z)=\begin{cases}\xi_j(z) & \text{ if }z\in {\cal D}_j
       \text{ for some }j,\\
                 1 & \text{ on }B_1\setminus\bigcup_{j=1}^m{\cal D}_j.
     \end{cases}
   \end{equation}

   Note that for any $p\in[1,2)$ we have by \eqref{eq:34} and \eqref{eq:208},
   \begin{equation}
     \label{eq:36}
     \begin{aligned}
     \int_{B_1} |\nabla\xi|^p=\sum_{j=1}^m \int_{{\cal D}_j}|\nabla\xi|^p
     &\le \sum_{j=1}^m\left(\int_{{\cal
           D}_j}|\nabla\xi|^2\right)^{p/2}|{\cal D}_j|^{1-p/2}\\
     &\le C\varepsilon^{p/2}R^{1-p/2}\le
    C\varepsilon^{p/2}\exp\big(-\frac{(2-p)B}{2\varepsilon}\big)\,.
    \end{aligned}
   \end{equation}

In $B_1$ we set $\widetilde\eta:=\xi^2\eta$ and
$\widetilde H:=\xi^2H$. From
\eqref{eq:86} and \eqref{eq:36} we conclude that 
\begin{equation}
  \label{eq:1111}
  \|\widetilde H\|_{W^{1,p}(B_1)}\leq C.
\end{equation}

The function $\widetilde \eta$ satisfies
\begin{equation*}
\begin{aligned}
  -\Div(\rho^2\nabla \widetilde\eta)\,=&-\Div(\rho^2\xi^2\nabla\eta)-\Div(\rho^2\eta\nabla(\xi^2))\\
 \,=&\underbrace{-\xi^2\Div(\rho^2\nabla\varphi)}_{F_1}\underbrace{-\rho^2\nabla(\xi^2)\cdot\nabla\varphi}_{F_2}+\Div(\underbrace{\rho^2\xi^2\nabla\Theta}_{G_1})+\Div(\underbrace{-2 \rho^2\eta\xi\nabla\xi}_{G_2})
\\:=&F_1+F_2+\Div G_1+\Div G_2.
\end{aligned}
\end{equation*}
First we note that $F_1=0$ by \eqref{eq:EL}.
Therefore, 
\begin{equation}
  \label{eq:85}
\begin{cases}
-\Delta (\widetilde\eta-\widetilde
H)=F_2+\Div( G_1+G_2) +\Div(\rho^2\nabla \widetilde H)+\Div((\rho^2-1)\nabla (\widetilde\eta-\widetilde
                      H))&\text{in }B_1,\\              
  \widetilde\eta-\widetilde H=0&\text{on }\partial B_1.
\end{cases}
\end{equation}
By elliptic estimates, for any $p\in[1,2)$ there exists $B_p=B_p(\Omega)>0$ such that the solution $w$ of the problem
\begin{equation}
  \label{eq:88}
\begin{cases}
  -\Delta w=v &\text{ in }\Omega,\\
    w=0 &\text{ on }\partial \Omega,
\end{cases}
\end{equation}
 with $v\in L^1(\Omega)$, satisfies
 \begin{equation}
   \label{eq:89}
 \|\nabla w\|_p\leq B_p \|v\|_1.
 \end{equation}
We bound $F_2$ in $L^1$ by
\begin{multline}
  \int_{B_1}|F_2|= \int_{B_1} |\rho^2\nabla(\xi^2)\cdot\nabla\varphi|\le2\sum_{j=1}^m\int_{{\cal D}_j}
  |\nabla\xi||\nabla\varphi|
  \\\le C\Big(\sum_{j=1}^m\big(\int_{{\cal D}_j}|\nabla\xi|^2\big)^{1/2}\Big)
   \big(\int_{B_1}|\nabla u|^2\big)^{1/2}\le
   \big(C\varepsilon \big)^{1/2}\cdot
   \big(\frac{C}{\varepsilon}\big)^{1/2}\le C,
 \end{multline}
 where we used \eqref{eq:upper} and \eqref{eq:34}.

 Clearly \eqref{eq:mobius-p} implies a bound
 \begin{equation}
   \label{eq:38}
   \|G_1\|_{L^p(B_1)}\le C\,.
 \end{equation}
To bound $G_2$ in $L^p$ we use \eqref{eq:36} and \eqref{eq:30} to get
\begin{equation}
\int_{B_1}|G_2|^p\leq C\|\eta\|_\infty^p\|\nabla\xi\|_{L^p(B_1)}^p\le
\left(\frac{C}{\varepsilon^p}\right)\varepsilon^{p/2}
\exp\big(-\frac{(2-p)B}{2\varepsilon}\big)\le\exp(-\frac{c}{\varepsilon})\,,
\end{equation}
for some positive constant $c$.
A bound in $L^p(B_1)$ for $\rho^2\nabla \widetilde H$ follows from \eqref{eq:1111}.

We also note that
\begin{equation*}
  1-\rho^2\leq 2(1-\tilde\beta)~\text{ on }~\supp(\nabla(\widetilde\eta-\widetilde H))\subset\Omega_\varepsilon.
\end{equation*}
 Using the above in \eqref{eq:85} we get by
 \eqref{eq:81} and \eqref{eq:89} that
 \begin{align}
   \label{eq:87}
\begin{split}
 \|\nabla(\widetilde\eta-\widetilde H)\|_{L^p}&\leq A_p\left(\|(\rho^2-1)\nabla
 (\widetilde\eta-\widetilde H)\|_{L^p}+
 \|G_1\|_{L^p}+\|G_2\|_{L^p}+\|\rho^2\nabla \widetilde H\|_{L^p}\right)\\
&\phantom{=}+B_p \|F_2\|_{L^1}
\leq
2A_p(1-\tilde\beta)\|\nabla(\widetilde\eta-\widetilde H)\|_{L^p}+C.
\end{split}
 \end{align}
 
 Combining \eqref{eq:82} and \eqref{eq:87}, we find that $ \|\nabla(\widetilde\eta
 -\widetilde H)\|_{L^p}\leq
 C$, which in conjunction with \eqref{eq:1111} implies that
 $\|\nabla\widetilde\eta \|_{L^p}\leq C$. Since
 $\|\nabla\Theta\|_{L^p(\Omega_\varepsilon)}\leq C$, we obtain that
 \begin{equation}
 \label{hm4}
   \|\nabla u\|_{L^p(B_1\setminus\cup_{j=1}^m  {\cal D}_j)}\leq C.
 \end{equation}
 Finally we note that for each $j=1,\ldots,m$ we have
 \begin{equation}
   \label{eq:40}
   \begin{aligned}
   \int_{{\cal D}_j}|\nabla
   u|^p\le\left(\int_{{\cal D}_j}|\nabla u|^2\right)^{p/2}|{\cal D}_j|^{1-p/2}
   &\le C\varepsilon^{-p/2}R^{1-p/2}\\
   &\le  C\varepsilon^{-p/2}\exp\big(-\frac{(2-p)B}{2\varepsilon}\big)=o_\varepsilon(1).
   \end{aligned}
 \end{equation}
 The conclusion of Proposition \ref{prop:W1p}   follows from
 \eqref{hm4} and \eqref{eq:40}. \qed 
\end{proof}
\subsection{Some identities satisfied by $u_\varepsilon$}
\label{subsec:iden}
In this subsection we list some (essentially known) identities satisfied by the minimizers that will be useful in  the proofs of both \rth{th:main1} and \rth{th:main2}. An important property of the minimizers is that the associated Hopf
 differential is a holomorphic function (see \cite[Lemma~3.1]{hl93}). Note that in dimension two this
property is equivalent to the \enquote{divergence free} property of the
stress-energy tensor, that holds in higher dimensions (of the
domain and the target), 
see e.g., \cite{ahl} and the references therein. In this subsection we
represent a point in $\Omega$ as $z=x_1+ix_2$ and we continue to drop
the subscript $\varepsilon$.
\begin{proposition}
  \label{prop:hopf}
  For any $\varepsilon>0$ the function
  \begin{equation}
    \label{eq:77}
    \chi=\chi_\varepsilon=|u_{x_1}|^2-|u_{x_2}|^2-2i\,u_{x_1}\cdot u_{x_2}+
    \Big(\frac{1}{\varepsilon^2}-1\Big)\Big(|u|_{x_1}^2-|u|_{x_2}^2-2i\,|u|_{x_1}|u|_{x_2}\Big)
  \end{equation}
   is holomorphic in $\Omega$ and the Cauchy-Riemann equations hold in
   the classical sense in a neighborhood of the boundary.
 \end{proposition}
 We emphasize that in \eqref{eq:77} the dot product refers to scalar
 product of vectors in $\R^2$.
\begin{proof}
    To see that the Cauchy-Riemann equations are satisfied in the
    sense of distributions, we consider the effect of a family of diffeomorphisms
    generated by an arbitrary vector field X on the
    energy $E_\varepsilon$ (see \cite{ahl}) . Since
    $u$ is H\"older continuous on $\overline{\Omega}$, in a small enough
    neighborhood
    of the boundary it satisfies $|u|>0$. Therefore $u$ is
    smooth in that neighborhood. We can then verify by a direct computation that the Cauchy-Riemann
    equations hold for $\chi$ in this neighborhood using
    \eqref{eq:EL}. \qed
  \end{proof}
  From \rprop{prop:hopf} we deduce the following Pohozaev identity. 
  \begin{corollary}
  \label{cor:Pohozaev}
  Every minimizer $u=u_\varepsilon$ satisfies
  \begin{equation}
    \label{eq:Poho}
    \int_{\partial B_1} (|\partial_r u|^2-|\partial_\tau g|^2)+\Big(\frac{1}{\varepsilon^2}-1\Big)|\partial_r\rho|^2=0\,.
  \end{equation}
\end{corollary}
\begin{proof}
  We denote
  \begin{equation}
    \label{eq:84}
    U=\left(u,\Big(\frac{1}{\varepsilon^2}-1\Big)\rho\right)\,.
  \end{equation}
  Therefore
  \begin{equation}
    \label{eq:78}
   \chi= |U_{x_1}|^2-|U_{x_2}|^2-2i\,U_{x_1}\cdot U_{x_2}.
 \end{equation}
 Since $\chi$ is holomorphic in $B_1$ and continuous on
 $\overline{B}_1$ we have in particular,
 \begin{equation}
   \label{eq:90}
   0=\int_{\partial B_1}\chi z\,dz=i\int_0^{2\pi}\chi(e^{i\theta})e^{2i\theta}\,d\theta.
 \end{equation}
 A direct computation shows that
 \begin{equation}
   \label{eq:92}
   |U_\nu|^2-|U_\tau|^2=|x_1U_{x_1}+x_2U_{x_2}|^2-|-x_2U_{x_1}+x_1U_{x_2}|^2=\text{Re}\left(\chi(z) z^2\right)\quad\text{ on }\partial B_1.
 \end{equation}
 Combining \eqref{eq:90} with \eqref{eq:92} gives that
 \begin{equation*}
   \int_{\partial B_1}|U_\nu|^2-|U_\tau|^2=0\,,
 \end{equation*}
  which is equivalent to \eqref{eq:Poho}. \qed 
\end{proof}

Next we present  a weak formulation of the equation
satisfied by the phase of $u$.
\begin{proposition}
  \label{prop:wedge}
  We have
  \begin{equation}
    \label{eq:72}
    \frac{\partial}{\partial x_1}\Big(u_\varepsilon\wedge
    (u_\varepsilon)_{x_1}\Big)+
    \frac{\partial}{\partial x_2}\Big(u_\varepsilon\wedge (u_\varepsilon)_{x_2}\Big)=0
  \end{equation}
 in the sense of distributions.
\end{proposition}
\begin{proof}
  Fix $\phi\in C^\infty_c(\Omega)$ and for $t\in\R$ let $u_\varepsilon=(u_1,u_2)$ and
  $u_\varepsilon^{(t)}:=e^{it\phi}u_\varepsilon$. From the minimality of $u_\varepsilon$ we derive by a
  simple computation that
  \begin{equation}
    \label{eq:75}
    0=\frac{d}{dt}\big|_{t=0}
        E_\varepsilon(u_\varepsilon^{(t)})=2\int_\Omega\sum_{j=1}^2 \Big((u_2)_{x_j}u_1-(u_1)_{x_j}u_2\Big)\phi_{x_j}\,.
  \end{equation}
Since $\phi$ is arbitrary we immediately deduce \eqref{eq:72}. \qed 
\end{proof}
\subsection{An $L^2$-bound for $|\nabla u_\varepsilon|$ away from the singularities}

 We denote by $a_1,\ldots,a_N\in\overline{B_1}$ the different limits of the families
 $\{x_j^{(\ve)}\}$, $j=1,\ldots, m$ (possibly along a subsequence). Since two different families may converge to the same limit, we have $N\le m$. 
 At this point we do not
 exclude the possibility that some of the $a_i$'s belong to $\partial B_1$. Consider any $r>0$ satisfying
 \begin{equation}
 \label{hm6}
 r<\min\{|a_i-a_j|:\, i\neq j\}\text{ and }r<\dist (a_j,\p B_1),\ \fo
 j\text{ such that }a_j\in B_1.
 \end{equation}
  
  We denote
 \begin{equation*}
\widetilde\Omega_r:=B_1\setminus \bigcup_{j=1}^N \overline {B_r(a_j)},
\end{equation*}
and by $d_j$ the  degree of $u_\ve$ on $\partial (B_s(a_j)\cap B_1)$
for a small $\ve$ and (a small but fixed) $s$. The following equality
is clear: if $J_j:=\{ \ell:\, x_\ell^{(\ve)}\to a_j\}$, then
$d_j=\sum_{\ell\in J_j}\kappa_\ell$. 
\begin{theorem}
\label{th:bound}
 For each $r$ as in \eqref{hm6} we have
 \begin{equation}
\label{eq:enrgy-bd}
\int_{\widetilde\Omega_r}|\nabla u_\varepsilon|^2\leq C(r).
\end{equation}
\end{theorem}
\begin{proof}
Note that, dropping the subscript $\varepsilon$, 
\begin{equation}
  \label{eq:47}
  |\nabla u|^2=|\nabla\rho|^2+\rho^2|\nabla\varphi|^2=|\nabla\rho|^2+\rho^2|\nabla(\Theta+\eta)|^2\,.
\end{equation}
 Since $\int_{B_1}|\nabla\rho|^2\le C\varepsilon$ by
 \eqref{eq:upper}, and $\int_{\widetilde\Omega_r} |\nabla\Theta|^2\le
 C(r)$ thanks to \rlemma{lem:mobius} and \eqref{eq:bound-kappa}, we only need to find a bound for $\int_{\widetilde\Omega_r}|\nabla \varphi|^2$.
By the boundedness of  $\{\nabla\eta\}$ in $L^1(\Omega_\ve)$ (see \rprop{prop:W1p}), it follows that there exists $\widetilde r=\widetilde r(\ve)\in (r/2,r)$ such that
 \begin{equation}
\label{eq:s}
\sum_{j=1}^N \int_{\partial B_{\widetilde r}(a_j)\cap\Omega} |\nabla\eta|\,d\sigma\leq C_1(r).
\end{equation}
Similarly,  we can  find for each
\begin{equation*}
  j\in I:=\{k\in 1,\ldots,N~\text{ such that }a_k\in B_1\},
  \end{equation*}
a number $\beta_j\in[0,2\pi)$ such that the set
\begin{equation*}
\widetilde L_j=\widetilde L_j(\beta_j):=\{a_j+s\, e^{i\beta_j}:\, s\ge \widetilde r\}\cap \widetilde \Omega_{\widetilde r}
\end{equation*}
satisfies
\begin{equation}
\label{eq:tilde}
\int_{\widetilde L_j} \left|\frac{\partial\eta}{\partial s}\right|\,ds\leq C_2(r).
\end{equation}
By the argument of the proof of Lemma \ref{lem:MP} and using  \eqref{eq:s} and \eqref{eq:tilde}, we find that 
\begin{equation}
\label{eq:eta-tilde}
\|\eta\|_{L^\infty(\widetilde \Omega_{\widetilde r})}\leq C_3(r).
\end{equation}
For $\varepsilon$ sufficiently small we have
\begin{equation}
  \label{eq:115}
  |x_\ell^{(\ve)}- a_j|<\widetilde r/2,~\forall\ell\in J_j,\,j=1,\ldots,N.
\end{equation}
Next, we multiply  the equation
\begin{equation*}
  -\Div(\rho^2\nabla\eta)=\Div(\rho^2\nabla\Theta)
\end{equation*}
by $\eta$, and integrate over
$\widetilde\Omega_{\widetilde r}$. 
This yields 
 \begin{equation}
 \label{eq:ontilde}
\int_{\widetilde\Omega_{\widetilde r}}
\rho{^2}|\nabla\eta|^2=-\int_{\widetilde\Omega_{\widetilde r}}\rho{^2}\nabla\Theta\cdot\nabla\eta
+\int_{\p\widetilde\Omega_{\widetilde r}}\rho^2\frac{\partial\va}{\partial n}\eta
:=I_1+I_2.
\end{equation}


We first claim that
\begin{equation}
  \label{eq:213}
  |I_2|\leq C_4(r)\,.
\end{equation}
Indeed, we use \eqref{eq:Poho} and
\eqref{eq:eta-tilde} for the integral on $\partial
\widetilde\Omega_{\widetilde r}\cap\partial B_1$ and for the
integral on $\partial B_{\widetilde r}(a_j)\cap B_1$ we use \eqref{eq:s} and
the fact that thanks to \eqref{eq:115} we have
$$\left|\frac{\partial\Theta}{\partial n}\right|\leq
\frac{C}{\widetilde r}~\text{ on }\partial B_{\widetilde r}(a_j).
$$
Applying the Cauchy-Schwarz inequality to $I_1$ in conjunction with \eqref{eq:213}
in \eqref{eq:ontilde} yields
\begin{equation}
\label{eq:116}
  \int_{\widetilde\Omega_{\widetilde r}}
\rho^2|\nabla\eta|^2\leq C_4(r)+\int_{\widetilde\Omega_{\widetilde r}}\frac{\rho^2}{2}|\nabla\eta|^2+\int_{\widetilde\Omega_{\widetilde r}}\frac{\rho^2}{2}|\nabla\Theta|^2\,.
\end{equation}

Since $\int_{\widetilde\Omega_{\widetilde r}}(\rho^2/2)\,
|\nabla\Theta|^2\leq C_5(r)(|\log r|+1)$, we deduce from \eqref{eq:116}
that $\int_{\widetilde\Omega_{\widetilde r}}|\nabla\eta|^2\leq
C_6(r)$. It follows that also 
$\int_{\widetilde \Omega_{\widetilde r}}|\nabla\va|^2\leq C_7(r)$,
which in view of \eqref{eq:47} clearly implies
\eqref{eq:enrgy-bd}. \qed 
\end{proof}
\subsection{Convergence of $u_{\varepsilon_n}$}
 Next, we will prove convergence of $u_{\varepsilon_n}$ on
 $\overline{B}_1\setminus\{a_1,\ldots,a_N\}$.
 \begin{proposition}
   \label{prop:conv-int}
   Let $b\in B_1$ and $r_1>0$ be such that $B_{r_1}(b)\subset
   B_1\setminus\{a_1,\ldots,a_N\}$. Then $u_{\varepsilon_n}\to u_0$ in
   $C^k~(B_{r_1/2}(b))$ for all $k\ge0$, where $u_0$ is a smooth
   $S^1$--valued harmonic map. 
  
 \end{proposition}
 \begin{proof}
    Since $|u_\varepsilon|\ge\tilde\beta$ in $B_{r_1}(b)$ for small $\varepsilon$, we may write
    $u_\varepsilon=\rho_\varepsilon e^{i\varphi_\varepsilon}$.
   By \rth{th:bound}, $\int_{ B_r(b)}|\nabla u_\varepsilon|^2\le
   C$. Also, $\int_{ B_r(b)}|\nabla \rho_\varepsilon|^2\le
   C\varepsilon$ by \eqref{eq:upper}. Hence by Fubini we can find $\tilde r\in((3/4)r_1,r_1)$ such that
   \begin{equation}
     \label{eq:43}
     \int_{\partial B_{\tilde r}(b)}
     |\nabla\varphi_\varepsilon|^2+\varepsilon^{-1}|\nabla\rho_\varepsilon|^2\le C.
   \end{equation}
   Since $\{\varphi_\varepsilon\big|_{\partial B_{\tilde r}(b)}\}$ is
   bounded in $H^1(\partial B_{\tilde r}(b))$, by passing to a
   subsequence we may assume that
   \begin{equation}
     \label{eq:57}
     \varphi_\varepsilon\big|_{\partial B_{\tilde
         r}(b)}\to\varphi_0~\text{ in }H^{1/2}(\partial B_{\tilde
     r}(b)) \text{ and uniformly on }\partial B_{\tilde r}(b).
   \end{equation}
   As for $\rho_\varepsilon$, from \eqref{eq:43} we infer that 
   \begin{equation}
     \label{eq:177}
    \rho_\varepsilon\big|_{\partial B_{\tilde r}(b)}\to c_0~\text{ in }H^{1}(\partial B_{\tilde
     r}(b)) \text{ and uniformly on }\partial B_{\tilde r}(b),
   \end{equation}
 for some constant $c_0\ge0$.
   We denote by $\widetilde\varphi_0$ the harmonic extension of
   $\varphi_0$ to  $B_{\tilde r}(b)$, and set
   $u_0=e^{i\widetilde\varphi_0}$. We are going to prove that
   $u_\varepsilon\to u_0$ on $B_{\tilde r}(b)$ in different norms,
   starting with the $H^1$-norm.

   \par We denote as usual by $\widetilde\varphi_\varepsilon$ and
   $\widetilde\rho_\varepsilon$, respectively, the harmonic extensions
   of $\rho_\varepsilon$ and $\varphi_\varepsilon$. First, by
   \eqref{eq:57} we have
   \begin{equation}
     \label{eq:178}
     \lim_{\varepsilon\to0} \int_{B_{\tilde  r}(b)} |\nabla\widetilde\varphi_\varepsilon|^2=\int_{B_{\tilde  r}(b)} |\nabla\widetilde\varphi_0|^2.
   \end{equation}
   Next we claim that
   \begin{equation}
     \label{eq:46}
    1-C\varepsilon^{1/2}\le \rho_\varepsilon\le 1\quad\text{ on }\partial B_{\tilde r}(b).
   \end{equation}
  Indeed, assuming first that $b=0$, we have as in
  \eqref{eq:721}--\eqref{eq:rho-r'} that
  \begin{equation}
    \label{eq:45}
    1-\frac{1}{2\pi \tilde r}\int_{\partial B_{\tilde
        r}}\rho_\varepsilon\le C\varepsilon^{1/2}.
  \end{equation}
   Note the difference with respect to the situation in
   Subsection~\ref{sec:conv-minim}:   here we have at our
   disposal only the weaker upper bound
    $\int_{B_{\tilde r}}|\nabla\rho_\varepsilon|^2\le C\varepsilon$.
Since \eqref{eq:43} implies that
\begin{equation}
  \label{eq:152}
  |\rho_\varepsilon(x)-\rho_\varepsilon(y)|\le
C\varepsilon^{1/2},\quad\forall x,y\in\partial B_{\tilde r}\,,
\end{equation}
  we deduce \eqref{eq:46} from \eqref{eq:45}--\eqref{eq:152} in the case $b=0$. The general case follows
  again by applying a M\"obius transformation.
  \par An immediate consequence
  of \eqref{eq:46} is that $c_0=1$. Therefore, the bound $\int_{ B_r(b))}|\nabla \rho_\varepsilon|^2\le
   C\varepsilon$ implies that
  \begin{equation}
    \label{eq:179}
   \rho_\varepsilon\to 1 \text{ in }H^1(B_{\tilde r}(b)).
  \end{equation}
  
  \par  Next we use the harmonic extensions of $\rho_\varepsilon$ and
  $\varphi_\varepsilon$ to construct the  comparison map $v_\varepsilon={\widetilde\rho_\varepsilon}e^{i\widetilde\varphi_\varepsilon}$ on $B_{\tilde
      r}(b)$.
    Clearly,
      \begin{equation}
        \label{eq:54}
        E_\varepsilon(u_\varepsilon; B_{\tilde r}(b))\le
        E_\varepsilon(v_\varepsilon; B_{\tilde r}(b))\le
        \int_{B_{\tilde r}(b)}
        |\nabla\widetilde\varphi_\varepsilon|^2+\frac{1}{\varepsilon^2}|\nabla\widetilde\rho_\varepsilon|^2\,.
      \end{equation}
 Since $\int_{B_{\tilde r}(b)}|\nabla\widetilde\rho_\varepsilon|^2\le
 \int_{B_{\tilde r}(b)}|\nabla\rho_\varepsilon|^2$, we deduce from
 \eqref{eq:54} that
 \begin{equation}
   \label{eq:55}
   E_\varepsilon(u_\varepsilon; B_{\tilde r}(b))=\int_{B_{\tilde r}(b)}\rho^2
   |\nabla\varphi_\varepsilon|^2+\frac{1}{\varepsilon^2}|\nabla\rho_\varepsilon|^2
   \le
   \int_{B_{\tilde r}(b)}
        |\nabla\widetilde\varphi_\varepsilon|^2+\frac{1}{\varepsilon^2}|\nabla\widetilde\rho_\varepsilon|^2\le \int_{B_{\tilde r}(b)}
        |\nabla\widetilde\varphi_\varepsilon|^2+\frac{1}{\varepsilon^2}|\nabla\rho_\varepsilon|^2\,.
      \end{equation}
       Therefore, $ \int_{B_{\tilde
           r}(b)}\rho_\varepsilon^2|\nabla\varphi_\varepsilon|^2\le
       \int_{B_{\tilde r}(b)}|\nabla\widetilde\varphi_\varepsilon|^2$,
       and we obtain that 
      \begin{equation}
        \label{eq:56}
        \int_{B_{\tilde r}(b)}|\nabla u_\varepsilon|^2=
        \int_{B_{\tilde r}(b)}\rho_\varepsilon^2|\nabla\varphi_\varepsilon|^2+|\nabla\rho_\varepsilon|^2\le \int_{B_{\tilde r}(b)}
        |\nabla\widetilde\varphi_\varepsilon|^2+C\varepsilon.
      \end{equation}
      Next, consider a subsequence such that
      $u_{\varepsilon_n}\rightharpoonup u$ weakly in $H^1(B_{\tilde
        r}(b))$. By \eqref{eq:179}, 
      $u=e^{i\varphi_0}=u_0$ on $\partial B_{\tilde r}(b)$, whence
      \begin{equation}
        \label{eq:180}
         \int_{B_{\tilde r}(b)}|\nabla u_0|^2\le  \int_{B_{\tilde r}(b)}|\nabla u|^2.
      \end{equation}

      Finally, by \eqref{eq:56}  and \eqref{eq:178} 
 we have
      \begin{equation}
        \label{eq:58}
        \int_{B_{\tilde r}(b)}|\nabla u|^2\le \limsup \int_{B_{\tilde
            r}(b)}|\nabla u_{\varepsilon_n}|^2\le \limsup \int_{B_{\tilde
            r}(b)}|\nabla \widetilde\varphi_{\varepsilon_n}|^2=\int_{B_{\tilde
            r}(b)}|\nabla u_{0}|^2\,.
      \end{equation}
       Combining \eqref{eq:180} with \eqref{eq:58} we get that $u=u_0$
       and then deduce the strong convergence (up to passing to a
       subsequence), 
       $u_\varepsilon\to u_0$ in $H^1(B_{\tilde r}(b))$.

  Next we write in $B_{\tilde r}$, 
  $\varphi_\varepsilon=\widetilde\varphi_\varepsilon+\psi_\varepsilon$,
  analogously to the notation we used in the proof of
  \rth{th:deg-zero} (i.e.,  $\psi_\varepsilon=0$ on $\partial
  B_{\tilde r}$). Note that
  $\rho_\varepsilon,\varphi_\varepsilon$ and $\psi_\varepsilon$
  satisfy the equations \eqref{eq:EL}--\eqref{eq:psi}. Since
  $\varphi_\varepsilon\big|_{\partial B_{\tilde  r}(b)}$ is bounded in
  $H^1(\partial B_{\tilde  r}(b))$, it follows that
  \begin{equation}
    \label{eq:48}
    \|\widetilde\varphi_\varepsilon\|_{H^{3/2}(B_{\tilde r}(b))}\le C\,.
  \end{equation}
  Then from Sobolev embeddings it follows that
  \begin{equation}
    \label{eq:49}
   \|\widetilde\varphi_\varepsilon\|_{W^{1,4}(B_{\tilde r}(b))}\le C.
 \end{equation}
  From the invariance of the equation
  \begin{equation}
    \label{eq:181}
    \Delta\psi_\varepsilon=\Div((1-\rho_\varepsilon^2)\nabla\varphi_\varepsilon)
  \end{equation}
 with respect to scalings it follows that the constant $A_4$ in the inequality
 \begin{equation}
   \label{eq:60}
   \|\nabla\psi_\varepsilon\|_{L^4(B_{\tilde r}(b))}\le A_4\|(1-\rho_\varepsilon^2)\nabla\varphi_\varepsilon\|_{L^4(B_{\tilde r}(b))}
 \end{equation}
can be chosen independently of the radius $\tilde r$. We may assume 
that $\tilde\beta$ that was used to construct the bad discs satisfies
in addition 
\begin{equation}
  \label{eq:59}
  1-\tilde\beta<\frac{1}{4A_4}\,.
\end{equation}
By \eqref{eq:49}--\eqref{eq:59} we get that
\begin{equation*}
  \|\nabla\psi_\varepsilon\|_{L^4(B_{\tilde r}(b))}\le 2(1-\tilde\beta) A_4\big(C+\|\psi_\varepsilon\|_{L^4(B_{\tilde r}(b))}\big),
\end{equation*}
  implying that 
  \begin{equation}
    \label{eq:61}
     \|\nabla\psi_\varepsilon\|_{L^4(B_{\tilde r}(b))}\le C~\text{ and }~\|\nabla\varphi_\varepsilon\|_{L^4(B_{\tilde r}(b))}\le C.
   \end{equation}
 Next we deduce from the equation satisfied by $\rho_\varepsilon$ in
 \eqref{eq:EL} and elliptic estimates that
 \begin{equation}
   \label{eq:62}
   \|\nabla(\rho_\varepsilon-\widetilde\rho_\varepsilon)\|_{L^p(B_{\tilde r}(b))}\le C_p
  \|\Delta\rho_\varepsilon\|_{L^2(B_{\tilde r}(b))}\le
  C_p\varepsilon^2\|\nabla\varphi_\varepsilon\|_{L^4(B_{\tilde r}(b))}^2\le 
   C\varepsilon^2,~\forall p<\infty.
 \end{equation}
 In particular, we deduce from\eqref{eq:62} that
 $\|\rho_\varepsilon-\widetilde\rho_\varepsilon\|_{L^\infty(B_{\tilde r}(b))}\le C\varepsilon^2$.
 Since  $\|1-\widetilde\rho_\varepsilon\|_{L^\infty(B_{\tilde r}(b))}\le C\varepsilon^{1/2}$ by
 \eqref{eq:46} and the maximum principle,  it follows that
 \begin{equation}
     \label{eq:63}
     \|\rho_\varepsilon-1\|_{L^\infty(B_{\tilde r}(b))}\le C\varepsilon^{1/2}.
   \end{equation}
   \par  We clearly have:
   \begin{equation}
     \label{eq:182}
      \widetilde\rho_\varepsilon \text{ and }
 \widetilde\varphi_\varepsilon\text{ are bounded in }
 W^{j,p}_{\text{loc}}(B_{\tilde r}(b)),~\forall j, \forall p.
   \end{equation}
  Using \eqref{eq:63} in \eqref{eq:181}, taking into account
  \eqref{eq:182},  we can deduce, as in
 the proof of \rth{th:deg-zero} that $\{\nabla\varphi_\varepsilon\}_{\varepsilon>0}$
 are uniformly bounded in $L^p_{\text{loc}}(B_{\tilde r}(b))$, for all $p>1$.
We can now
 conclude the proof of the $C^k$-convergence by induction as in the
 proof of \rth{th:deg-zero}. \qed 
\end{proof}
 We will also  need a version of \rprop{prop:conv-int} in a
 neighborhood of the boundary.
\begin{proposition}
   \label{prop:conv-bd}
   Let $b\in \partial B_1$ and $r_1>0$ be such that $B_{r_1}(b)\subset
   \overline{B}_1\setminus\{a_1,\ldots,a_N\}$. Then, $u_{\varepsilon_n}\to u_0$ in
   $C^k(B_{r_1/2}(b)\cap B_1)$ for all $k\ge0$, where $u_0$ is a smooth
   $S^1$--valued harmonic map satisfying $u_0=g$ on
   $B_{r_1}(b)\cap\partial B_1$.
 \end{proposition}
 \begin{proof}
   As in the proof of \rprop{prop:conv-int} we may use Fubini to find $\tilde r\in((3/4)r_1,r_1)$ such that
   \begin{equation}
    \label{eq:50}
     \int_{\partial B_{\tilde r}(b)\cap B_1}
     |\nabla\varphi_\varepsilon|^2+\frac{1}{\varepsilon}|\nabla\rho_\varepsilon|^2\le C.
   \end{equation}
   Denoting by $q$ any of the two points in $\partial B_{\tilde
     r}(b)\cap \partial B_1$, we obtain by the Cauchy-Schwarz
   inequality that
   \begin{equation}
     \label{eq:51}
     |\rho_\varepsilon(x)-1|=
     |\rho_\varepsilon(x)-\rho_\varepsilon(q)|\le
     C\varepsilon^{1/2},\quad\forall x\in \partial B_{\tilde r}(b)\cap B_1\,,
   \end{equation}
    which is the analogue of \eqref{eq:46} in our setting. The rest of
    the proof follows by the same arguments as in the proof of
    \rprop{prop:conv-int}. \qed
  \end{proof}
  \subsection{Conclusion of the proof of \rth{th:main1}}
   As explained in the Introduction, we may assume that $\Omega=B_1$.
   \begin{proof}[of \rth{th:main1}]
      The inequality \eqref{eq:42} is the result of \rcor{cor:energy}. 
    The convergence result \eqref{eq:91} follows from \rprop{prop:conv-int} and
    \rprop{prop:conv-bd}. The fact that $d_j>0$ for all $j$ follows
    from \eqref{eq:bound-kappa}.
    \par Next we prove that $a_j\in B_1$ for all
    $j$, that is, singularities cannot occur on the boundary. 
The proof is the same as that of \cite[Theorem X.4]{BBH}, so we
  just describe the main idea. By Pohozaev identity \eqref{eq:Poho}
  and \rprop{prop:conv-bd} it follows that
  \begin{equation}
  \label{eq:52}
   \int_{\partial B_1} \left|\frac{\partial u_*}{\partial r}\right|^2<\infty.
  \end{equation}
Since by \rprop{prop:W1p} we also have
$u_{\varepsilon_n}\rightharpoonup u_*$ weakly in $W^{1,p}$, for all
$p\in(1,2)$ it follows that $u_*\in W^{1,p}(B_1; S^1)$ for all
 $p\in[1,2)$.
 Therefore, all the hypotheses of \cite[Lemma X.14]{BBH} are satisfied,
 and we can conclude that $u_*$ is smooth in a neighborhood of
 $\partial B_1$.

 Finally we show that $u_*$ is the canonical harmonic map associated
 with $g$, the singularities and their degrees. By
 \rprop{prop:W1p} we can
 pass to the limit $\varepsilon\to0^{+}$ in \eqref{eq:72} and 
 deduce that
 \begin{equation}
   \label{eq:76}
    \frac{\partial}{\partial x_1}\Big(u_*\wedge
    (u_*)_{x_1}\Big)+
    \frac{\partial}{\partial x_2}\Big(u_*\wedge (u_*)_{x_2}\Big)=0\,.
  \end{equation}
   But by \cite[Remark I.1]{BBH} the only $S^1$--valued harmonic map in
   $W^{1,1}(\Omega)$ satisfying \eqref{eq:76} is the canonical one. \qed
 \end{proof}
 \section{Proof of \rth{th:main2}}
 \label{sec:main2}
\subsection{An improved upper bound for $E_\varepsilon(u_\varepsilon)$}
 We begin with the easy part, the upper bound, in the estimate \eqref{eq:102}.
 \begin{proposition}
   \label{prop:ref-ub}
     Under the assumptions of \rth{th:main1} we have
     \begin{equation}
       \label{eq:66}
       \limsup_{\varepsilon\to0^{+}}
       E_\varepsilon(u_\varepsilon)-\frac{2\pi D}{\varepsilon}\le
         d^2_{H^{1/2}}(g,\mathcal{H}_D(\partial\Omega))\,.
     \end{equation}
     In fact, for each fixed $\varepsilon>0$ we have
     \begin{equation}
       \label{eq:71}
       E_\varepsilon(u_\varepsilon)-\frac{2\pi D}{\varepsilon}\le  d^2_{H^{1/2}}(g,\mathcal{H}_D(\partial\Omega)).
     \end{equation}
   \end{proposition}
   \begin{proof}
    As before we assume w.l.o.g.~that $\Omega=B_1$.  Fix any ${\bf b}\in B_1^D$. We know from
     Subsection~\ref{subsec:blaschke} that 
     \begin{equation*}
       U_{{\bf b},\varepsilon}(z)=| {\cal B}_{\bf b}(z)|^\varepsilon\left(\frac{{\cal B}_{\bf b}(z)}{|{\cal B}_{\bf b}(z)|}\right)
       \end{equation*}
        is a minimizer for $E_\varepsilon$ for its own boundary data,
        with $E_\varepsilon(U_{{\bf b},\varepsilon})=2\pi D/\varepsilon$.
Set $\bar\rho_\varepsilon:=|U_{{\bf b},\varepsilon}|=| {\cal B}_{\bf b}(z)|^\varepsilon$ and write
     \begin{equation}
       \label{eq:67}
       U_{{\bf b},\varepsilon}(z)=\bar\rho_\varepsilon(z)\exp^{i\Theta(z)}\,.
     \end{equation}
   Note that although  $\Theta$ is well-defined only {\em locally} in
   $\overline{B}_1\setminus\{b_1,\ldots,b_D\}$, its gradient
     $\nabla\Theta$ is globally defined.  Let $\psi$ be a smooth lifting of $g/{\cal B}_{\bf
        b}\big|_{\partial B_1}$, that is, $g=e^{i\psi}{\cal B}_{\bf
        b}$ on $\partial B_1$, and let $\widetilde\psi$ denote the
      harmonic extension of $\psi$ to $B_1$.
We set $v_\varepsilon=e^{i\widetilde\psi} U_{{\bf b},\varepsilon}$ and
note that $v_\varepsilon=g$ on $\partial B_1$. 
Using  $|v_\varepsilon|=|U_{{\bf b},\varepsilon}|=\bar
\rho_\varepsilon$ we get
\begin{equation}
  \label{eq:68}
  \begin{aligned}
  E_\varepsilon(u_\varepsilon)&\le E_\varepsilon(v_\varepsilon)=\int_{B_1}{\varepsilon}^{-2}|\nabla
  \bar\rho_\varepsilon|^2+\bar\rho_\varepsilon^2\big(|\nabla\Theta|^2+
  2\nabla\Theta\cdot\nabla\widetilde\psi+|\nabla\widetilde\psi|^2\big)
  \\&=
  E_\varepsilon(U_{{\bf
      b},\varepsilon})+2\int_{B_1}\bar\rho_\varepsilon^2\nabla\Theta\cdot\nabla\widetilde\psi+\int_{B_1}\bar\rho_\varepsilon^2|\nabla\widetilde\psi|^2\\
   &=\frac{2\pi
     D}{\varepsilon}+2\int_{B_1}\bar\rho_\varepsilon^2\nabla\Theta\cdot\nabla\widetilde\psi+\int_{B_1}\bar\rho_\varepsilon^2|\nabla\widetilde\psi|^2\,.
  \end{aligned} 
\end{equation}
 Next we recall that $\Theta$ is a harmonic conjugate of
 $h:=(1/\varepsilon)\ln\bar\rho_\varepsilon=\ln |{\cal B}_{\bf
   b}|$. The function $h$ is defined globally in $B_1$,
 having singularities at the points $b_1,\ldots,b_D$. Moreover,
 \begin{equation}
   \label{eq:70}
   h=0~\text{ and
   }\frac{\partial\Theta}{\partial\nu}=-\frac{\partial h}{\partial\tau}=0~\text{ on }\partial B_1.
 \end{equation}
 Therefore,
 \begin{equation}
   \label{eq:69}
   \int_{B_1}\bar\rho_\varepsilon^2\nabla\Theta\cdot\nabla\widetilde\psi=-\int_{B_1}\Div(\bar\rho_\varepsilon^2\nabla\Theta)\widetilde\psi+\int_{\partial B_1}\bar\rho_\varepsilon^2\left(\frac{\partial\Theta}{\partial\nu}\right)\widetilde\psi=
   -\int_{\partial B_1}\bar\rho_\varepsilon^2\left(\frac{\partial h}{\partial\tau}\right)\widetilde\psi=0\,,
 \end{equation}
  where we used $\Div(\bar\rho_\varepsilon^2\nabla\Theta)=0$ in $B_1$
  and \eqref{eq:70} on $\partial B_1$. Plugging \eqref{eq:69} in \eqref{eq:68} yields
  \begin{equation}
    \label{eq:214}
E_\varepsilon(u_\varepsilon)\le
E_\varepsilon(v_\varepsilon)\le\frac{2\pi D}{\varepsilon}+\int_{B_1}\bar\rho_\varepsilon^2|\nabla\widetilde\psi|^2\,.
 \end{equation}
  Since the configuration ${\bf b}\in B_1^D$ is arbitrary we deduce
  from the definition \eqref{eq:110} of $d_{H^{1/2}}^2(g,
  \mathcal{H}_D(\partial B_1))$ and \eqref{eq:214} that
  \eqref{eq:71} holds, whence also \eqref{eq:66}.\qed
\end{proof}
\subsection{The limit of $\ln\rho_\varepsilon/\varepsilon$ and $(\rho_\varepsilon-1)/\varepsilon$} 
We begin with a local $L^\infty$-bound for
$|\nabla\rho_\varepsilon|/\varepsilon$, away from $\partial B_1$ and the points $a_1,\ldots,a_N$.
\begin{lemma}
  \label{lem:bound-grad}
  For every small $\eta>0$ we have
  \begin{equation}
    \label{eq:103}
    |\nabla\rho_\varepsilon|/\varepsilon\le C_\eta~\text{ on
    }~B_{1-\eta}\setminus\bigcup_{j=1}^N B_\eta(a_j), \quad\forall \varepsilon\in(0,1)\,. 
  \end{equation}
\end{lemma}
\begin{proof}
  For simplicity we now drop the subscript $\varepsilon$.
   From \rcor{cor:Pohozaev} we get that
   \begin{equation*}
     \int_{\partial B_1}
     |u_\nu|^2+\frac{\rho_\nu^2}{\varepsilon^2}=\int_{\partial
       B_1} |u_\tau|^2\le C.
   \end{equation*}
   Therefore,
   \begin{equation}
     \label{eq:104}
     \int_{\partial B_1}
     (|u_x|^2+|u_y|^2)+\frac{1}{\varepsilon^2}(\rho_x^2+\rho_y^2)=
     \int_{\partial B_1}
     (|u_\nu|^2+|u_\tau|^2)+\frac{1}{\varepsilon^2}(\rho_\nu^2+\rho_\tau^2)\le C.
   \end{equation}
   Let us denote, as in \eqref{eq:84}, $U=\big(u,
   \big(\frac{1}{\varepsilon^2}-1\big)^{1/2}\rho\big)$ and
   consider the two harmonic functions 
   $h_1=\big|U_x\big|^2-\big|U_y\big|^2$ and $h_2=2U_x\cdot U_y$.
   From \eqref{eq:104} we deduce that
   \begin{equation}
     \label{eq:105}
     \int_{\partial B_1}|h_1|=  \int_{\partial
       B_1}\Big|\big|U_x\big|^2-\big|U_y\big|^2\Big|\le \int_{\partial
       B_1}\big|U_x\big|^2+\big|U_y\big|^2\le C.
   \end{equation}
   Similarly,
   \begin{equation}
     \label{eq:106}
      \int_{\partial B_1}|h_2|\le  \int_{\partial B_1}2|U_x||U_y|\le \int_{\partial
       B_1}\big|U_x\big|^2+\big|U_y\big|^2\le C.
   \end{equation}
   From \eqref{eq:105}--\eqref{eq:106} and the Poisson formula it
   follows that
   \begin{equation}
     \label{eq:108}
     \|h_1\|_{L^\infty(B_{1-\eta})},  \|h_2\|_{L^\infty(B_{1-\eta})}\le
         C_\eta.
       \end{equation}
     
       Thanks to  \rth{th:main1} we also have, 
        \begin{equation}
        \label{eq:112}
         |\nabla u|\le C_\eta~\text{ on
         }B_1\setminus\bigcup_{j=1}^N B_\eta(a_j).
       \end{equation}
       Combining \eqref{eq:108} with \eqref{eq:112} yields
       \begin{equation}
         \label{eq:183}
         \left|\Big(\frac{\rho_x}{\varepsilon}\Big)^2-\Big(\frac{\rho_y}{\varepsilon}\Big)^2\right|\le
         C_\eta~\text{ and}~\left|\Big(\frac{\rho_x}{\varepsilon}\Big)\Big(\frac{\rho_y}{\varepsilon}\Big)\right|\le
         C_\eta \quad\text{ on
         }B_{1-\eta}\setminus\bigcup_{j=1}^N B_\eta(a_j).
       \end{equation}
      Since
      $(\rho_x^2-\rho_y^2)^2+(2\rho_x\rho_y)^2=(\rho_x^2+\rho_y^2)^2$,
      \eqref{eq:103} follows from \eqref{eq:183}. \qed
     \end{proof}
      The next result provides a crucial bound for the energy away
      from the singularities of $u_*$.
      \begin{proposition}
        \label{prop:away}
        Let $\eta>0$ satisfy
        \begin{equation}
          \label{eq:eta-cond}
          \eta<\frac{1}{2}\min_{i\ne j} |a_i-a_j|~\text{ and
          }~\eta<\min_j (1-|a_j|).
        \end{equation}
        Then,
        \begin{equation}
          \label{eq:53}
          E_\varepsilon\big(u_\varepsilon;B_1\setminus\bigcup_{j=1}^NB_\eta(a_j)\big)\le C_\eta.
        \end{equation}
      \end{proposition}
      \begin{proof}
        For  $j=1,\ldots,N$ we denote
        \begin{equation}
          \label{eq:113}
          m_j=m_j(\varepsilon,\eta)=\min_{\partial
            B_\eta(a_j)}\rho_\varepsilon~\text{ and }~M_j=M_j(\varepsilon,\eta)=\max_{\partial
            B_\eta(a_j)}\rho_\varepsilon\,.
        \end{equation}
        Thanks to \eqref{eq:103} we have
        \begin{equation*}
          M_j-m_j\le C_\eta\varepsilon,~j=1,\ldots,N.
        \end{equation*}
        Actually, connecting pairs of  circles from $\{\partial B_\eta(a_j)\}_{j=1}^N$ to
        each other by segments allows us to deduce from \eqref{eq:103} that
        \begin{equation}
          \label{eq:114}
          |M_j-m_i|\le C_\eta\varepsilon,~i,j=1,\ldots,N.
        \end{equation}
Let us denote $\overline m=\min_j m_j$. By \eqref{eq:114} and
\eqref{eq:103} we have
\begin{equation}
  \label{eq:117}
  |\rho_\varepsilon-\overline m|\le C_\eta\varepsilon~\text{ on
  }~\bigcup_{j=1}^N \big(B_\eta(a_j)\setminus B_{\eta/2}(a_j)\big).
\end{equation}
Next we define  a function $S\in H^1_0(B_1)$ by
\begin{equation}
  \label{eq:118}
  S(x)=\begin{cases}
    1-\rho_\varepsilon(x) & x\in B_1\setminus\bigcup_{j=1}^N
    B_\eta(a_j),\\
    1-\big(\frac{2}{\eta}\big)\big((|x-a_j|-\frac{\eta}{2})\rho_\varepsilon(x)+(\eta-|x-a_j|)\overline
    m\big) & x\in B_\eta(a_j)\setminus B_{\eta/2}(a_j), 1\le j \le N,\\
1- \overline m & x\in B_{\eta/2}(a_j), 1\le j \le N.
\end{cases}
\end{equation}
Thanks to \eqref{eq:103} and
\eqref{eq:117} we have
\begin{equation}
  \label{eq:119}
  \int_{B_1}|\nabla S|^2\le\int_{B_1\setminus\bigcup_{j=1}^N
    B_\eta(a_j)}|\nabla\rho_\varepsilon|^2+C_\eta\varepsilon^2\le
  \varepsilon^2 E_\varepsilon(u_\varepsilon;
  B_1\setminus\bigcup_{j=1}^N B_\eta(a_j))+C_\eta\varepsilon^2.
\end{equation}
 Next we apply Trudinger's inequality to $S$, similarly to the way it
 was used in the proof of \cite[Lemma X.5]{BBH}. It yields, for some
 universal constants $\sigma_1,\sigma_2$,
 \begin{equation}
   \label{eq:120}
   \int_{B_1}\exp\left(\frac{|S|}{\sigma_1\|\nabla  S\|_2}\right)\le \sigma_2|B_1|.
 \end{equation}
 In particular, we obtain from \eqref{eq:120} that
 \begin{equation*}
   |B_{\eta/2}(a_1)|\exp\left(\frac{1-\overline m}{\sigma_1\|\nabla  S\|_2}\right) \le \sigma_2|B_1|\,,
 \end{equation*}
  which after some manipulations and application of \eqref{eq:119}
  leads to
  \begin{equation}
    \label{eq:122}
    1-\overline m\le C_\eta\varepsilon \Big(E_\varepsilon\big(u_\varepsilon;
  B_1\setminus\bigcup_{j=1}^N B_\eta(a_j)\big)+1\Big)^{1/2}.
\end{equation}

Next, the same argument that was used in the proof of \rprop{prop:lb} gives
\begin{equation}
  \label{eq:123}
  E_\varepsilon(u_\varepsilon;\bigcup_{j=1}^N
  B_\eta(a_j))\geq\frac{2}{\varepsilon} (2\pi D)\int_0^{\overline m}
  t\,dt=\frac{2\pi D}{\varepsilon}{\overline m}^2.
\end{equation}
Combining \eqref{eq:122}--\eqref{eq:123} with the upper bound from
\eqref{eq:97} yields
\begin{equation*}
  \frac{2\pi D}{\varepsilon}{\overline
    m}^2+C_\eta\frac{(1-\overline m)^2}{\varepsilon^2}\le \frac{2\pi D}{\varepsilon}+C,
\end{equation*}
implying that 
\begin{equation}
  \label{eq:125}
  1-\overline
 m\le C_\eta\varepsilon.
\end{equation}
Finally, plugging \eqref{eq:125} in 
\eqref{eq:123} yields $E_\varepsilon(u_\varepsilon;\bigcup_{j=1}^N
  B_\eta(a_j))\geq\frac{2\pi D}{\varepsilon}-C_\eta$, which together
  with \eqref{eq:97} leads to \eqref{eq:53}. \qed
\end{proof}
In the course of the proof of \rprop{prop:away} we also obtained the
necessary information needed to prove that
 $1-\rho_\varepsilon=O(\varepsilon)$ locally in
 $B_1\setminus\{a_1,\ldots,a_N\}$. More precisely:
 \begin{proposition}
   \label{prop:rhoto1}
   For every small $\eta>0$ we have
   \begin{equation}
     \label{eq:121}
     1-\rho_\varepsilon\le C_\eta\varepsilon~\text{ in
     }B_{1-\eta}\setminus\bigcup_{j=1}^N B_\eta(a_j).
   \end{equation}
 \end{proposition}
 \begin{proof}
   First, combining \eqref{eq:125} with \eqref{eq:114} yields
   \begin{equation}
     \label{eq:127}
     1-\rho_\varepsilon\le C_\eta\varepsilon~\text{ on }\bigcup_{j=1}^N\partial B_\eta(a_j).
   \end{equation}
   Any point $x\in B_{1-\eta}\setminus\bigcup_{j=1}^N B_\eta(a_j)$
  can be connected to the closest circle, say $\partial B_\eta(a_{j_0})$.
 Using \eqref{eq:127} in conjunction with \eqref{eq:103} we
conclude that $(1-\rho_\varepsilon)(x)\le C_\eta\varepsilon$. \qed 
\end{proof}
Next we strengthen further our estimate  for $1-\rho_\varepsilon$.
\begin{proposition}
  \label{prop:rho-precise}
  \begin{equation}
    \label{eq:126}
    \lim_{\varepsilon\to0}\frac{\rho_\varepsilon-1}{\varepsilon}=\lim_{\varepsilon\to0}\frac{\ln\rho_\varepsilon}{\varepsilon}=\Phi_0~\text{
      in }C^k_{\text{loc}}(\overline{B}_1\setminus\{a_1,\ldots,a_N\}),\text{ for
      all }k\ge 1,
  \end{equation}
  where $\Phi_0$ is the solution of \eqref{eq:101}. 
\end{proposition}
\begin{proof}
  The proof is divided to several steps.\\[2mm]
  \noindent\underline{Step 1: Convergence of $\frac{\rho_\varepsilon-1}{\varepsilon}\text{ in }
 C^k_{\text{loc}}(B_1\setminus\{a_1,\ldots,a_N\}).$ }\\[1mm]
  Let $x_0\in B_1\setminus\{a_1,\ldots,a_N\}$ be given. Choose
  $\eta>0$ such that $B_\eta(x_0)\subset
  B_1\setminus\{a_1,\ldots,a_N\}$. By \eqref{eq:EL} we have
  \begin{equation}
    \label{eq:128}
    \Delta\left(\frac{\rho_\varepsilon-1}{\varepsilon}\right)=\varepsilon\rho_\varepsilon|\nabla\varphi_\varepsilon|^2.
  \end{equation}
  Denoting as usual the harmonic extension
  of $\rho_\varepsilon$ by $\widetilde\rho_\varepsilon$, we set
  $w_\varepsilon:=\frac{\widetilde\rho_\varepsilon-1}{\varepsilon}$. It is a
  harmonic function that thanks to \rprop{prop:rhoto1} satisfies
  \begin{equation}
    \label{eq:129}
    \|w_\varepsilon\|_{L^\infty(\partial B_\eta(x_0))}\le C.
  \end{equation}
  It follows that
  \begin{equation}
    \label{eq:130}
    \|w_\varepsilon\|_{C^k(B_{3\eta/4}(x_0))}\le C,~\forall
    k\ge 1.
  \end{equation}
  In particular,
  \begin{equation}
    \label{eq:132}
     w_\varepsilon\to\Phi\text{ in }C^k(
   B_{\eta/2}(x_0))\text{ for all }k,
  \end{equation}
  and the limit $\Phi$  is a harmonic
   function. Now, by \eqref{eq:128} the function
   $f_\varepsilon:=(\frac{\rho_\varepsilon-1}{\varepsilon})-w_\varepsilon$
   satisfies
   \begin{equation}
     \label{eq:131}
     \left\{
       \begin{aligned}
     \Delta
     f_\varepsilon&=\varepsilon\rho_\varepsilon|\nabla\varphi_\varepsilon|^2~\text{
       in }B_\eta(x_0)\\
     f_\varepsilon&=0~\text{ on }\partial B_\eta(x_0).
   \end{aligned}
   \right.
   \end{equation}
  It follows from \eqref{eq:131} and \rth{th:main1} that
  $\|f_\varepsilon\|_{C^k(B_\eta(x_0))}=O(\varepsilon)$, for all $k\ge
  1$, which in conjunction with \eqref{eq:132} yields that
  \begin{equation}
    \label{eq:133}
   \frac{\rho_\varepsilon-1}{\varepsilon}\to\Phi\text{ in }C^k(B_{\eta/2}(x_0))\text{ for all }k.
 \end{equation}
 Since $x_0$ is arbitrary, we deduce the convergence
 \begin{equation}
   \label{eq:135}
   \frac{\rho_\varepsilon-1}{\varepsilon}\to\Phi\text{ in }
 C^k_{\text{loc}}(B_1\setminus\{a_1,\ldots,a_N\}).
 \end{equation}

  \noindent\underline{Step 2: Convergence of $\frac{\ln\rho_\varepsilon}{\varepsilon}\text{ in }
 C^k_{\text{loc}}(B_1\setminus\{a_1,\ldots,a_N\}).$ }\\[1mm]
 To deduce the same convergence for 
 $\ln\rho_\varepsilon/\varepsilon$, we note first that this function
 satisfies in $B_1\setminus\{a_1,\ldots,a_N\}$ the equation
 \begin{equation}
   \label{eq:134}
   \Delta
   (\ln\rho_\varepsilon/\varepsilon)=\varepsilon\Big(|\nabla\varphi_\varepsilon|^2-\Big(\frac{|\nabla\rho_\varepsilon|}{\varepsilon\rho_\varepsilon}\Big)^2\Big).
 \end{equation}
By \rth{th:main1} and \eqref{eq:135} we obtain that locally in
$B_1\setminus\{a_1,\ldots,a_N\}$, the  R.H.S.~of
\eqref{eq:134} is $O(\varepsilon)$. Therefore, by the same argument as
in the first part of the proof we can deduce that also
\begin{equation}
   \label{eq:162}
   \frac{\ln\rho_\varepsilon}{\varepsilon}\to\Phi\text{ in }
 C^k_{\text{loc}}(B_1\setminus\{a_1,\ldots,a_N\}),
 \end{equation}
noting that the limit must be the same $\Phi$ since locally in
$B_1\setminus\{a_1,\ldots,a_N\}$ we have 
$$
\frac{\ln\rho_\varepsilon}{\varepsilon}-\frac{\rho_\varepsilon-1}{\varepsilon}=O\Big(\frac{(1-\rho_\varepsilon)^2}{\varepsilon}\Big)=O(\varepsilon).
$$

 \noindent\underline{Step 3: Convergence of
  $\frac{\ln\rho_\varepsilon}{\varepsilon}\text{ and
  }\frac{\rho_\varepsilon-1}{\varepsilon}$ up to the boundary}\\[1mm]
 We recall
 that so far we haven't shown even that
 $\frac{|\nabla\rho_\varepsilon|}{\varepsilon}$ is bounded up to the
 boundary. Let $\eta$ satisfy
\begin{equation}
   \label{eq:136}
   0<\eta<\min\{1-|a_j|\}_{j=1}^N\,.
 \end{equation}
 Fix any point $b\in\partial B_1$.
 By  \rprop{prop:away} we have
 \begin{equation*}
   E_\varepsilon(u_\varepsilon;B_\eta(b)\cap B_1)\le C.
 \end{equation*}
  Therefore, by Fubini we can choose $\tilde\eta\in(\eta/2,\eta)$ such
  that
  \begin{equation}
    \label{eq:137}
    \int_{\partial B_{\tilde\eta}(b)\cap
      B_1}\frac{|\nabla\rho_\varepsilon|^2}{\varepsilon^2}\le C\,;
  \end{equation}
  note the improvement over \eqref{eq:50}.
  Denoting by $q$ any of the two points in $\partial B_{\tilde
     \eta}(b)\cap \partial B_1$, we obtain by the Cauchy-Schwarz
   inequality that
   \begin{equation}
    \label{eq:138}
     |\rho_\varepsilon(x)-1|=
     |\rho_\varepsilon(x)-\rho_\varepsilon(q)|\le
     C\varepsilon,\quad\forall x\in \partial B_{\tilde \eta}(b)\cap B_1\,,
   \end{equation}
    which is stronger than \eqref{eq:51}. We can now proceed as in the
    proof of the estimate around an interior point. In fact, setting
    $w_\varepsilon:=\frac{\widetilde\rho_\varepsilon-1}{\varepsilon}$,
    where, as usual, $\widetilde\rho_\varepsilon$ denotes the harmonic
    extension of $\rho_\varepsilon$ from $\partial\big(B_{\tilde\eta}(b)\cap B_1\big)$ to $B_{\tilde\eta}(b)\cap B_1$, we
    have thanks to \eqref{eq:138} that 
  \begin{equation}
    \label{eq:139}
    \|w_\varepsilon\|_{L^\infty(\partial (B_{\tilde\eta}(b)\cap B_1))}\le C.
  \end{equation}
 Therefore, analogously to \eqref{eq:130} we have 
  \begin{equation}
   \label{eq:163}
    \|w_\varepsilon\|_{C^k_{\text{loc}}(B_{\tilde\eta}(b)\cap B_1)}\le C,~\forall
    k\ge 1.
  \end{equation}
  This allows us to repeat the argument of Step\,2, using again the
  equation \eqref{eq:128}, to deduce that
  \begin{equation}
    \label{eq:140}
   \frac{\rho_\varepsilon-1}{\varepsilon}\to\Phi\text{ in }C^k(
   B_{\tilde\eta/2}(b)\cap B_1)\text{ for all }k.
 \end{equation}
 We can then argue as above to obtain that also
 \begin{equation}
   \label{eq:141}
   \frac{\ln\rho_\varepsilon}{\varepsilon}\to\Phi\text{ in }C^k(
   B_{\tilde\eta/2}(b)\cap B_1)\text{ for all }k.
 \end{equation}
  Since the point $b\in\partial B_1$ is arbitrary, we deduce
  that both convergences,
  $\frac{\rho_\varepsilon-1}{\varepsilon}\to\Phi$ and
  $\frac{\ln\rho_\varepsilon}{\varepsilon}\to\Phi$, hold in $C^k$-norm in a
  neighborhood of the boundary.\\[1mm]

   \noindent\underline{Step 4: Identification of the limit $\Phi$ as
     $\Phi_0$}\\[1mm]
   We already  know that $\Phi$ is a harmonic function in
   $B_1\setminus\{a_1,\ldots,a_N\}$, which is continuous in
   $\overline{B}_1\setminus\{a_1,\ldots,a_N\}$ and satisfies  $\Phi=0$ on
   $\partial B_1$. Recall the Hopf differentials
   $\{\chi_\varepsilon\}$ introduced in Subsection~\ref{subsec:iden}. In the
   proof of \rlemma{lem:bound-grad} we showed that
   $\{\chi_\varepsilon\}$ are bounded in $L^\infty_{\text{loc}}(B_1)$ (see
   \eqref{eq:108}). Therefore, we have $\chi_\varepsilon\to \chi_*$
   in $C^k_{\text{loc}}(B_1)$ where $\chi_*$ is holomorphic in $B_1$
   and locally bounded. In fact, thanks to Step~3 and \rth{th:main1} we can assert that
   the convergence actually holds in $C^k(B_1)$.  On the other hand,
   because of the convergences
   \begin{equation}
     \label{eq:142}
     \nabla u_\varepsilon\to \nabla(e^{i\varphi_*})~\text{ and
     }~\frac{\nabla\rho_\varepsilon}{\varepsilon}\to\nabla\Phi\text{ in
     }C^k({\overline{B}_1}\setminus\{a_j\}_{j=1}^N),
   \end{equation}
   established in \rth{th:main1}, and the previous steps, we have in $B_1\setminus\{a_j\}_{j=1}^N$:
   \begin{equation}
     \label{eq:143}
     \chi_*=\Big(\frac{\partial\varphi_*}{\partial
       x}\Big)^2-\Big(\frac{\partial\varphi_*}{\partial y}\Big)^2-
     2i \frac{\partial\varphi_*}{\partial x}\cdot\frac{\partial\varphi_*}{\partial y}+\Big(\frac{\partial\Phi}{\partial
       x}\Big)^2-\Big(\frac{\partial\Phi}{\partial y}\Big)^2-2i
     \Big(\frac{\partial\Phi}{\partial
       x}\Big)\cdot\Big(\frac{\partial\Phi}{\partial y}\Big)\,.
   \end{equation}
   Here and in the sequel we use $\varphi_*$ to denote the phase of
   $u_*$, but we keep in mind that this function is defined only {\em
     locally} in $B_1\setminus\{a_j\}_{j=1}^N$, and even there it is
   determined uniquely only up to an additive constant in
   $2\pi\Z$. Yet, the gradient $\nabla\varphi_*$ is globally defined
   in  $B_1\setminus\{a_j\}_{j=1}^N$. Since $\chi_*$ belongs to $L^\infty(B_1)$, we may take
   the modulus in both sides of
   \eqref{eq:143} and deduce that
   \begin{equation}
     \label{eq:109}
     |\nabla\Phi|^2=|\nabla\varphi_*|^2+O(1)~\text{ in }B_1\setminus\{a_j\}_{j=1}^N.
   \end{equation}
Since $|\nabla\varphi_*|\in L^p(B_1)$ for all $p\in[1,2)$
it follows from \eqref{eq:109} that also $|\nabla\Phi|\in L^p(B_1)$  for all
$p\in[1,2)$.

Since $\Phi$ is harmonic in  $B_1\setminus\{a_1,\ldots,a_N\}$ and
$|\nabla\Phi|\in L^1_{\text{loc}}(B_1)$, we must have
\begin{equation}
  \label{eq:145}
  \Delta\Phi=\sum_{j=1}^N(2\pi c_j)\delta_{a_j}~\text{ in the distributions sense},
\end{equation}
for some constants $\{c_j\}_{j=1}^N$. Therefore we have
\begin{equation}
  \label{eq:146}
  \Phi(z)=\sum_{j=1}^Nc_j\ln|z-a_j|+H~\text{ in }B_1,
\end{equation}
for some smooth harmonic function $H$.

We still need to determine the values of  $\{c_j\}_{j=1}^N$.
Fix any $j$ and assume for simplicity of notation that $a_j=0$.
In a punctured neighborhood of $0$, $B^*_\eta=B_\eta\setminus\{0\}$, we have
\begin{equation}
  \label{eq:147}
  e^{i\varphi_*}=e^{id_j\theta+f_j}\,,
\end{equation}
    where $f_j$ is a smooth harmonic function in a neighborhood of
    $0$ (including $0$). Similarly, in $B^*_\eta$
    we have also
    \begin{equation}
      \label{eq:148}
      \Phi(z)=c_j\ln|z|+h_j,
    \end{equation}
     with $h_j$ having the same properties as $f_j$. Rewriting
     \eqref{eq:143} as 
     \begin{equation*}
       \left(\frac{\partial\Phi}{\partial z}\right)^2=-\left(\frac{\partial\varphi_*}{\partial z}\right)^2+\chi_*/4,
     \end{equation*}
      and plugging \eqref{eq:147}--\eqref{eq:148}, yields
\begin{equation}
  \label{eq:150}
  \left(\frac{c_j}{z}+2\frac{\partial h_j}{\partial
      z}\right)^2=-\left(-i\frac{d_j}{z}+2\frac{\partial f_j}{\partial
      z}\right)^2+\chi_*~\text{ in }B^*_\eta.
\end{equation}
Multiplying \eqref{eq:150} by $z^2$ and sending $z$ to zero gives
$c_j^2=d_j^2$, so that $c_j=\pm d_j$. Since $\Phi\le0$ (as the limit of
$\ln\rho_\varepsilon/\varepsilon$) we conclude that $c_j=d_j$. Using
this for all $j$'s in \eqref{eq:145} clearly implies that
$\Phi=\Phi_0$, the function given in \eqref{eq:101}. \qed
\end{proof}
\subsection{A precise asymptotic estimate for the energy}
\label{subsec:precise}
Our next objective is to prove the lower bound in \eqref{eq:102}. Recall
that for the points $a_1,\ldots,a_N$ and degrees $d_1,\ldots,d_N$
given by \rth{th:main1} we associate the function $\Phi_0$ satisfying
\eqref{eq:101} and its conjugate harmonic function $\Theta_0$ (which
is well-defined only locally in $B_1\setminus\{a_1,\ldots,a_N\}$);
$\Theta_0$ is unique up to an additive constant in $2\pi\Z$ that we
can fix arbitrarily.  Once a representative of $\Theta_0$ is fixed, the
function $\phi=\varphi_*-\Theta_0$ is well defined on $\partial B_1$ and we denote by
$\widetilde\varphi$ its harmonic extension to $B_1$. We keep in mind
that $\widetilde\varphi$ is determined uniquely up to an additive
constant which is an integer multiple of $2\pi$.
\begin{lemma}
  \label{lem:fine-lb}
  For each small $\lambda>0$ we have
  \begin{equation}
  \label{eq:151}
E_\varepsilon(u_\varepsilon)\geq\frac{2\pi D}{\varepsilon}+\int_{B_1}|\nabla\widetilde\varphi|^2+o_\lambda(1)+o_\varepsilon^{(\lambda)}(1),
\end{equation}
 where $o_\varepsilon^{(\lambda)}(1)$ denotes a quantity that tends to $0$ with $\varepsilon$, for each
   fixed $\lambda$, while  $o_\lambda(1)$
   denotes a quantity that tends to $0$ with $\lambda$
   (independently of $\varepsilon$).
 \end{lemma}
 \begin{proof}
   Fix a small $\lambda>0$ and denote $\Omega_\lambda=B_1\setminus\bigcup_{j=1}^N
       B_\lambda(a_j)$. By \rprop{prop:rho-precise} and
   \rth{th:main1} we have
   \begin{equation}
     \label{eq:144}
     \begin{aligned}
     E_\varepsilon(u_\varepsilon; \Omega_\lambda)&=
     \frac{1}{\varepsilon^2}\int_{\Omega_\lambda}|\nabla\rho_\varepsilon|^2+\int_{\Omega_\lambda}\rho_\varepsilon^2|\nabla\varphi_\varepsilon|^2\\
   &=\int_{\Omega_\lambda}|\nabla\Phi_0|^2+\int_{\Omega_\lambda}|\nabla\varphi_*|^2+o_\varepsilon^{(\lambda)}(1).
   \end{aligned}
   \end{equation}
    Since $\varphi_*=\Theta_0+\widetilde\varphi$, we have
   \begin{multline}
     \label{eq:149}
      \int_{\Omega_\lambda}|\nabla\varphi_*|^2=
     \int_{\Omega_\lambda}\Big(|\nabla\Theta_0|^2+2\nabla\Theta_0\nabla\widetilde\varphi+|\nabla\widetilde\varphi|^2\Big)\\
     =\int_{\Omega_\lambda}|\nabla\Theta_0|^2+\int_{\Omega_\lambda}|\nabla\widetilde\varphi|^2+
    2\sum_{j=1}^N\int_{\partial
      B_\lambda(a_j)}\frac{\partial\Theta_0}{\partial\nu}\widetilde\varphi+
    2\int_{\partial B_1}\frac{\partial\Theta_0}{\partial\nu}\widetilde\varphi\,.
   \end{multline}
   Here $\nu$ stands for the outward normal w.r.t.~the domain $\Omega_\lambda=B_1\setminus\bigcup_{j=1}^N
       B_\lambda(a_j)$ on each component of its boundary. Next we use
       the fact that
       $\frac{\partial\Theta_0}{\partial\nu}=-\frac{\partial\Phi_0}{\partial\tau}$
       which implies in particular that
       $\frac{\partial\Theta_0}{\partial\nu}=0$ on $\partial
       B_1$. Therefore
       \begin{equation}
         \label{eq:124}
         \int_{\Omega_\lambda}|\nabla\varphi_*|^2=\int_{\Omega_\lambda}|\nabla\Theta_0|^2+\int_{\Omega_\lambda}|\nabla\widetilde\varphi|^2-2\sum_{j=1}^N\int_{\partial
      B_\lambda(a_j)}\frac{\partial\Phi_0}{\partial\tau}\widetilde\varphi\,.
       \end{equation}
 Since $\big|\frac{\partial\Phi_0}{\partial\tau}\big|\le C$ on each $\partial
 B_\lambda(a_j) $ we have
 \begin{equation}
   \label{eq:153}
   \Big|\int_{\partial
     B_\lambda(a_j)}\frac{\partial\Phi_0}{\partial\tau}\widetilde\varphi\Big|
   \le C \|\widetilde\varphi\|_\infty(2\pi\lambda)=o_\lambda(1).
 \end{equation}
 Note that above we could have replaced $\|\widetilde\varphi\|_\infty$
 by $\min_{m\in\N} \|\widetilde\varphi-2\pi m\|_\infty$.
 By \eqref{eq:124}--\eqref{eq:153} we get
 \begin{equation}
 \label{eq:193}
   \int_{\Omega_\lambda}|\nabla\varphi_*|^2=\int_{\Omega_\lambda}|\nabla\Theta_0|^2+\int_{\Omega_\lambda}|\nabla\widetilde\varphi|^2+o_\lambda(1).
 \end{equation}
 By \eqref{eq:144},\eqref{eq:193} and the relation
 $|\nabla\Theta_0|=|\nabla\Phi_0|$ we finally obtain that
 \begin{equation}
   \label{eq:154}
   E_\varepsilon(u_\varepsilon;\Omega_\lambda)=2\int_{\Omega_\lambda}|\nabla\Phi_0|^2+
   \int_{\Omega_\lambda}|\nabla\widetilde\varphi|^2+o_\lambda(1)+o_\varepsilon^{(\lambda)}(1).
 \end{equation}

 We continue to estimate the first integral on the R.H.S.~of \eqref{eq:154}.
First we define for each $j$,
$m_j=m_j(\lambda,\varepsilon)=\min_{x\in\partial
  B_\lambda(a_j)}\rho_\varepsilon(x)$.
By \rprop{prop:rho-precise} we have
\begin{equation}
  \label{eq:155}
  2\int_{\Omega_\lambda}|\nabla\Phi_0|^2=2\sum_{j=1}^N \int_{\partial
    B_\lambda(a_j)}
  \Phi_0\,\frac{\partial\Phi_0}{\partial\nu}=2\sum_{j=1}^N \int_{\partial
    B_\lambda(a_j)}\left(\frac{\rho_\varepsilon-1}{\varepsilon}+o_\varepsilon^{(\lambda)}(1)\right)
  \frac{\partial\Phi_0}{\partial\nu}.
\end{equation}
Note that thanks again to \rprop{prop:rho-precise} we have
\begin{equation}
  \label{eq:156}
  \left\|\frac{\rho_\varepsilon-1}{\varepsilon}-\frac{m_j-1}{\varepsilon}\right\|_{L^\infty(\partial
    B_\lambda(a_j))}\le \max_{x,y\in\partial B_\lambda(a_j)}
  |\Phi_0(x)-\Phi_0(y)|+o_\varepsilon^{(\lambda)}(1)\le o_\lambda(1)+o_\varepsilon^{(\lambda)}(1).
\end{equation}
Therefore, for each $j$ we have
\begin{multline}
  \label{eq:157}
  \int_{\partial
    B_\lambda(a_j)}\left(\frac{\rho_\varepsilon-1}{\varepsilon}+o_\varepsilon^{(\lambda)}(1)\right)
  \frac{\partial\Phi_0}{\partial\nu}=\int_{\partial
    B_\lambda(a_j)}\left(\frac{m_j-1}{\varepsilon}+o_\lambda(1)
  +o_\varepsilon^{(\lambda)}(1)\right)
  \frac{\partial\Phi_0}{\partial\nu}\\
  =-2\pi d_j \left(\frac{m_j-1}{\varepsilon}\right)+o_\lambda(1) +o_\varepsilon^{(\lambda)}(1),
\end{multline}
 where we used the fact that $\int_{\partial
   B_\lambda(a_j)}\frac{\partial\Phi_0}{\partial\nu}=-2\pi d_j$ thanks
 to \eqref{eq:101}.
 Plugging \eqref{eq:157} in \eqref{eq:155} yields
 \begin{equation}
   \label{eq:158}
   2\int_{\Omega_\lambda}|\nabla\Phi_0|^2=-4\pi\sum_{j=1}^N d_j \left(\frac{m_j-1}{\varepsilon}\right)+o_\lambda(1) +o_\varepsilon^{(\lambda)}(1).
 \end{equation}
On the other hand, the argument based on the coarea formula, used in
the proof of \rprop{prop:lb} (and again in \eqref{eq:15}), gives that
\begin{equation}
  \label{eq:159}
  E_\varepsilon(u_\varepsilon;B_\lambda(a_j))\ge \frac{2}{\varepsilon}\int_{
    B_\lambda(a_j)}\rho_\varepsilon|\nabla\rho_\varepsilon||\nabla\varphi_\varepsilon|
  \ge\frac{2\pi
    d_jm_j^2}{\varepsilon}\,,~\forall j.
\end{equation}
Combining \eqref{eq:154},\eqref{eq:158} and \eqref{eq:159} we obtain,
\begin{equation}
  \label{eq:160}
  \begin{aligned}
  E_\varepsilon(u_\varepsilon)&\ge\frac{4\pi}{\varepsilon}\sum_{j=1}^N d_j\Big(\frac{m_j^2}{2}-(m_j-1)\Big)+
  \int_{\Omega_\lambda}|\nabla\widetilde\varphi|^2+o_\lambda(1)+o_\varepsilon^{(\lambda)}(1)\\&=
  \frac{2\pi D}{\varepsilon}+\frac{2\pi}{\varepsilon}\sum_{j=1}^N d_j(m_j-1)^2 +
  \int_{\Omega_\lambda}|\nabla\widetilde\varphi|^2+o_\lambda(1)+o_\varepsilon^{(\lambda)}(1)\\&=\frac{2\pi
    D}{\varepsilon}+\int_{\Omega_\lambda}|\nabla\widetilde\varphi|^2+o_\lambda(1)+o_\varepsilon^{(\lambda)}(1)
  \,,
\end{aligned}
\end{equation}
where in  the last estimate we used the fact that $1-m_j\le
C_\lambda\varepsilon$, implying that $(m_j-1)^2/\varepsilon\le
C_\lambda\varepsilon=o_\varepsilon^{(\lambda)}(1)$.
 The desired conclusion \eqref{eq:151} follows from \eqref{eq:160}
 since 
$$
\int_{\bigcup_{j=1}^N B_\lambda(a_j)}
|\nabla\widetilde\varphi|^2=o_{\lambda}(1).
$$
\qed
\end{proof}
\subsection{Conclusion of the proof of \rth{th:main2}}
  \begin{proof}[of \rth{th:main2}]
   Assertion (i) follows from \rprop{prop:rho-precise}. The inequality
   \enquote{$\le$} in \eqref{eq:102} was proved in
   \rprop{prop:ref-ub}. To prove the inequality
   \enquote{$\ge$} we use \rlemma{lem:fine-lb}. We first fix $\lambda$
   and send $\varepsilon$ to $0$ to get
   \begin{equation}
     \label{eq:161}
     \liminf_{\varepsilon\to0} E_\varepsilon(u_\varepsilon)-\frac{2\pi D}{\varepsilon}\ge \int_{B_1}|\nabla\widetilde\varphi|^2+o_\lambda(1).
   \end{equation}
Then, sending $\lambda$ to $0$ in \eqref{eq:161} yields 
\begin{equation*}
     \liminf_{\varepsilon\to0} E_\varepsilon(u_\varepsilon)-\frac{2\pi D}{\varepsilon}\ge \int_{B_1}|\nabla\widetilde\varphi|^2,
\end{equation*}
 and the conclusion follows. Finally, assertion (iii) is a direct
 consequence of assertion (ii). \qed
\end{proof}
\section{Proof of \rprop{prop:W}.}
\label{sec:prop}
 This short section is devoted to the proof \rprop{prop:W} that
 provides an explicit expression for the \enquote{excess energy}
 $d^2_{H^{1/2}}(g,\mathcal{H}_D(\partial\Omega))$
 and clarifies its relation with the renormalized energy $W$.
 \begin{proof}[of \rprop{prop:W}]
   Let $\Theta_0$ be a conjugate harmonic function of $\Phi_0$ as
   described in the beginning of Subsection\,\ref{subsec:precise}, but
   this time for a general simply connected domain $\Omega$. We have
   \begin{equation}
     \label{eq:188}
     d^2_{H^{1/2}}(g,\mathcal{H}_D(\partial\Omega))=\int_\Omega|\nabla\widetilde\varphi|^2,
   \end{equation}
where $\widetilde\varphi$ is the harmonic extension of the function
$\psi$ given on $\partial\Omega$ as $\psi=\varphi_*-\Theta_0$, i.e., 
  $e^{i\psi}=g/f_0$, with $f_0=U_0\Big|_{\partial\Omega}$ where
  $U_0=e^{i\Theta_0}$. Therefore,
  \begin{equation}
    \label{eq:191}
    u_*=U_0e^{i\widetilde\varphi}\text{ in }\Omega.
  \end{equation}
  Next we apply \eqref{eq:184} twice, first
 for $u_*$,
 \begin{equation}
   \label{eq:189}
    \int_{\Omega_\lambda}|\nabla
      u_*|^2=2\pi\Big(\sum_{j=1}^Nd_j^2\Big)\ln(1/\lambda)+W({\bf a},{\bf d},g)+O(\lambda^2),\text{
        as }\lambda\to0^+,
 \end{equation}
 and then for $U_0$,
 \begin{equation}
   \label{eq:190}
    \int_{\Omega_\lambda}|\nabla
      U_0|^2=2\pi\Big(\sum_{j=1}^Nd_j^2\Big)\ln(1/\lambda)+W({\bf a},{\bf d},f_0)+O(\lambda^2),\text{
        as }\lambda\to0^+.
    \end{equation}
  Since $|\nabla U_0|=|\nabla\Theta_0|$ and $|\nabla u_*|=|\nabla\varphi_*|$, we infer from \eqref{eq:189}--\eqref{eq:190}
  and \eqref{eq:193} that
  \begin{equation}
    \label{eq:192}
    W({\bf a},{\bf d},g)=W({\bf a},{\bf d},f_0)+\int_\Omega|\nabla\widetilde\varphi|^2.
  \end{equation}
 An immediate consequence of \eqref{eq:192} is that the minimum of
 $\widetilde W(\bf a,\bf d)$ (see \eqref{eq:95}) is attained by
 $f_0$. Clearly \eqref{eq:96} follows from \eqref{eq:192} and \eqref{eq:188}.

 \par Finally we turn to the proof of \eqref{eq:187}. Here we need an explicit expression for $\widetilde W(\bf a,\bf d)$
 in the case $\Omega=B_1$. Since now we know that the minimum defining 
$\widetilde W(\bf a,\bf d)$ in \eqref{eq:95} is attained by $f_0$, we
can rely on the formula \eqref{eq:190} and compute an asymptotic
expansion for $\int_{\Omega_\lambda}|\nabla\Phi_0|^2$ as
$\lambda\to0$. This can be done rather easily but a
similar computation was already done in
\cite[Prop.\,1]{lr96}:
\begin{equation}
  \label{eq:194}
      \widetilde W({\bf a},{\bf d})=-2\pi\sum_{j\ne k} d_j
      d_k\ln|a_j-a_k|+2\pi\sum_{j,k} d_jd_k\ln |1-{\bar a}_ja_k|.
    \end{equation}
    Finally, by \eqref{eq:186} and \eqref{eq:194} we obtain that
    \begin{equation}
      \label{eq:195}
      W({\bf a},{\bf d},g)-\widetilde
     W({\bf a},{\bf d})=\int_{\partial B_1}{\widetilde\Phi}_0(g\times
     g_\tau)-2\pi\sum_{j=1}^Nd_jR_0(a_j)-
     2\pi\sum_{j,k=1}^Nd_jd_k\ln|1-a_j{\bar a}_k|\,,
   \end{equation}
    and the result follows from \eqref{eq:96}. \qed
 \end{proof}
\section{Appendix - the thin film limit of the 3D model}
In this short appendix we will show that the two dimensional minimization problem of
the energy $E_\varepsilon$ over $H^1_g(\Omega)=H^1_g(\Omega;\R^2)$ (see
\eqref{eq:energy}) can be viewed as a limit of a
problem defined on a thin film, $\Omega_h:=\Omega\times(0,h)\subset\R^3$ when the thickness $h$
goes to zero. We {\em fix} $\varepsilon$ and for each $h>0$  let
$w_h=w_{h,\varepsilon}$ denote a minimizer for the problem
\begin{equation}
  \label{eq:164}
  \min\left\{
    F_h(u):=\int_{\Omega\times(0,h)}\left(\frac{1}{\varepsilon^2}-1\right)|\nabla|u||^2+|\nabla
    u|^2: u\in {\cal V}_h\right\}\,,
\end{equation}
where
\begin{equation}
  \label{eq:165}
  {\cal V}_h=\big\{u\in H^1(\Omega_h;\R^3):u(x,y,z)=g(x,y)\text{ for
  }(x,y,z)\in\partial\Omega\times(0,h),\,u\perp{\vec e}_3\text{ on }\Omega\times\{0,h\}\big\},
\end{equation}
with ${\vec e}_3$ denoting a unit vector in the direction of the
$z$-axis.

Next, for any $u\in {\cal V}_h$ we use rescaling to define $\tilde
u\in H^1(\Omega\times(0,1);\R^3)$ by setting
\begin{equation}
  \label{eq:166}
  \tilde u(x,y,z)=u(x,y,hz).
\end{equation}
A simple computation yields that
\begin{equation}
\begin{split}
  \label{eq:167}
  \widetilde F_h(\tilde u):=h^{-1}F_h(u)=
  \int_{\Omega\times(0,1)}\left(\frac{1}{\varepsilon^2}-1\right)|\nabla_{x,y}|\tilde
  u||^2&+|\nabla_{xy} \tilde u|^2\\
  &+
  \frac{1}{h^2}\int_{\Omega\times(0,1)}\left(\frac{1}{\varepsilon^2}-1\right)\Big|\frac{\partial
    |\tilde
  u|}{\partial z}\Big|^2+\Big|\frac{\partial\tilde u}{\partial z}\Big|^2\,.
\end{split}
\end{equation}
So Problem \eqref{eq:164} is equivalent to the following
problem:
\begin{equation}
 \label{eq:168}
  \min\left\{
    \widetilde F_h(\tilde u):  \tilde u\in H^1(\Omega\times(0,1);\R^3), \tilde u=g\text{ on }
 \partial\Omega\times(0,1),\,\tilde u\perp{\vec e}_3\text{ on }\Omega\times\{0,1\}\right\},
\end{equation}
 for which the minimizer is given by $\widetilde
 w_h(x,y,z)=w_h(x,y,hz)$.
 \begin{proposition}
   \label{prop:thin-film}
   For a subsequence we have
     \begin{equation}
       \label{eq:170}
        \lim_{h_n\to0} \widetilde w_{h_n}=U_\varepsilon,
     \end{equation}
     where $U_\varepsilon(x,y,z)=u_\varepsilon(x,y)$, with 
     $u_\varepsilon$ being a minimizer for $E_\varepsilon$ over
 $H^1_g(\Omega)$.
 \end{proposition}
 \begin{proof}
   Let $u_\varepsilon$ be any minimizer for $E_\varepsilon$ over
   $H^1_g(\Omega)$. Clearly
 $U_\varepsilon$ is an admissible map for
 \eqref{eq:168}, whence
 \begin{equation}
   \label{eq:169}
   \widetilde F_h(\widetilde w_h)\le \widetilde F_h(U_\varepsilon)=E_\varepsilon(u_\varepsilon).
 \end{equation}
 It follows from \eqref{eq:169} and \eqref{eq:167} that
 \begin{equation}
   \label{eq:171}
   \lim_{h\to0} \int_{\Omega\times(0,1)}\Big|\frac{\partial
    |
  \widetilde w_h|}{\partial z}\Big|^2+\Big|\frac{\partial \widetilde w_h}{\partial z}\Big|^2=0.
 \end{equation}
 Let $\widetilde w_{h_n}\rightharpoonup V_\varepsilon$ weakly in
 $H^1(\Omega\times(0,1);\R^3)$. In particular, for the trace we have,
 $\widetilde w_{h_n}\to V_\varepsilon$  strongly in
 $L^2(\Omega\times\{0,1\};\R^3)$ and a.e., so that 
 \begin{equation}
   \label{eq:172}
   V_\varepsilon\perp{\vec e}_3\text{ on }\Omega\times\{0,1\}.
 \end{equation}
 It follows from \eqref{eq:171} that
 $V_\varepsilon$ is independent of the $z$-variable,
 i.e.,$V_\varepsilon(x,y,z)=V_\varepsilon(x,y)$, while by  \eqref{eq:172} $V_\varepsilon$ is $\R^2$-valued. 
 Passing to the limit in \eqref{eq:169}, using weak lower
 semicontinuity, we get
 \begin{equation}
   \label{eq:173}
   \int_{\Omega\times(0,1)}\left(\frac{1}{\varepsilon^2}-1\right)|\nabla_{x,y}|V_\varepsilon||^2+
   |\nabla_{x,y}V_\varepsilon|^2=E_\varepsilon(V_\varepsilon)\le E_\varepsilon(u_\varepsilon).
 \end{equation}
 We conclude that $V_\varepsilon(x,y)$ is a minimizer for $E_\varepsilon$
 over $H^1_g(\Omega)$ and that $\{\widetilde w_{h_n}\}$ converges
 strongly to $V_\varepsilon$ in $H^1(\Omega\times(0,1);\R^3)$. \qed
 \end{proof}
   

\end{document}